\documentclass[a4paper,11pt,reqno]{article}

\usepackage{latexsym,amsmath,amssymb,amsfonts,amsthm}
\usepackage{mathrsfs}
\usepackage{mathtools}
\usepackage{dsfont}
\usepackage{bbm}
\usepackage{soul}     
\usepackage{a4wide}
\usepackage[utf8]{inputenc}
\usepackage[usenames, dvipsnames]{xcolor}
\usepackage{graphicx}
\usepackage{sectsty}
\usepackage[colorlinks=true, linkcolor=blue, citecolor=ForestGreen]{hyperref}

\newcommand{\rosso}[1]{{\color{black} {#1}}}
\newcommand{\blu}[1]{{\color{black} {#1}}}

\newcommand{\mb}[1]{{\rosso{#1}}}

\newcommand{\xupref}[2]{\hspace{-0.3ex}\stackrel{\eqref{#1}}{#2}} 

\newtheorem{theorem}{Theorem}[section]
\newtheorem{lemma}[theorem]{Lemma}
\newtheorem{proposition}[theorem]{Proposition}

\newtheorem{definition}[theorem]{Definition}
\newtheorem{remark}[theorem]{Remark}
\newtheorem{example}[theorem]{Example}

\numberwithin{equation}{section}

\newcommand{\e}{\varepsilon}

\newcommand{\N}{\mathbb N}

\newcommand{\R}{\mathbb R}

\newcommand{\s}{\mathbb S}

\renewcommand{\div}{{\rm div}}

\DeclareMathOperator*{\loc}{loc}

\newcommand{\leb}{{\mathscr L}}
\newcommand{\hn}{{\mathcal H}^{n-1}}
\newcommand{\hu}{{\mathcal H}^{1}}

\newcommand{\de}{\,\mathrm{d}}
\renewcommand{\setminus}{\backslash}
\newcommand{\mres}{\mathbin{\vrule height 1.6ex depth 0pt width 0.13ex\vrule height 0.13ex depth 0pt width 1.3ex}}

\newcommand{\asymm}{{\bigtriangleup}}
\newcommand{\per}{{\mathscr{P}}}
\newcommand{\BV}{{\rm BV}}

\newcommand{\ap}{{\mathcal{AP}}}
\newcommand{\ac}{{\mathcal{X}}}
\newcommand{\acreg}{{\mathcal{X}_{\mathrm{reg}}}}
\newcommand{\nl}{{\mathcal{N}}}
\newcommand{\f}{{\mathscr{F}}}
\newcommand{\fbar}{{\mkern2mu\overline{\mkern-2mu\mathscr{F}}}}
\newcommand{\fc}{{\mathcal{F}}}
\newcommand{\g}{{\mathscr{G}}}
\newcommand{\q}{{\mathcal{Q}}}
\newcommand{\cil}{{\mathcal{C}}}

\allowdisplaybreaks

\title{\rosso{Area quasi-minimizing partitions with a graphical constraint: relaxation and two-dimensional \blu{partial} regularity}}
\author{Marco Bonacini\thanks{\addmarco\ \emailmarco} \and Riccardo Cristoferi\thanks{\addric\ \emailric}}
\date{\today}

\newcommand{\email}[1]{E-mail: \tt #1}
\newcommand{\emailmarco}{\email{marco.bonacini@unitn.it}}
\newcommand{\emailric}{\email{riccardo.cristoferi@ru.nl}}
\newcommand{\addmarco}{\emph{Department of Mathematics, University of Trento, Italy.}}
\newcommand{\addric}{Corresponding Author, \emph{Department of Mathematics - IMAPP, Radboud University, Nijmegen, The Netherlands.}}

\begin{document}
	
\maketitle

\begin{abstract}
We consider a variational model for periodic partitions of the upper half-space into three regions, where two of them have prescribed volume and are subject to the geometrical constraint that their union is the subgraph of a function, whose graph is a free surface. The energy of a configuration is given by the weighted sum of the areas of the interfaces between the different regions, and a general volume-order term. We establish existence of minimizing configurations via relaxation of the energy involved, in any dimension. Moreover, we prove partial regularity results for volume-constrained minimizers in two space dimensions. Thin films of diblock copolymers are a possible application and motivation for considering this problem.
\end{abstract}



\section{Introduction}

The goal of this paper is to initiate the analytical investigation of variational models for partitions with quasi-minimal surface area, subject to a geometrical graph constraint. The admissible configurations of the model that we consider here consist of two phases (i.e.\ regions of the space with prescribed volume) which are confined by a flat substrate on the bottom side, and by the graph of a Lipschitz function on the upper side. The upper interface between the two phases and the region above them corresponds to a free surface.

The imposition of a graph constraint on the admissible configurations is not new in the mathematical literature and appeared in particular in variational models for epitaxially strained elastic films, see \cite{BonCha02,ChaSol07,FonFusLeoMor07,DavPio19,CriFri20}; however, to the best of our knowledge this is the first instance where a similar constraint is enforced on a system with multiple phases, and this constitutes the main novelty of this paper.

We are interested in the description of optimal configurations minimizing an energy functional given by the sum of the surface measures of the different interfaces between the phases, possibly with different weights. We also include in the total energy a general volume-order term (allowing, for instance, for possible nonlocal interactions among the two phases); the results in this paper are valid under quite general assumptions on this term, whereas its explicit form would be crucial for the characterization of optimal or equilibrium configurations.

This paper is the first step in the rigorous investigation of the properties of the energy and of optimal configurations of the system. In particular, we discuss the lower semicontinuity properties of the energy, which permits to prove existence of minimizing configurations via relaxation in any dimension. Moreover, we establish several regularity properties of minimizers in dimension two. Further investigations on the fine structure of optimal configurations will be the subject of future work.

A possible application of the variational model that we introduce is the description of equilibrium configurations of thin films of diblock copolymers, see Section~\ref{subsection:thinfilms} for details.

\subsection{The model}
We now pass to an introductory description of the model and of the main results obtained in this paper. For the precise definitions and assumptions we refer to Section~\ref{sect:def}.

\begin{figure}
	\centering
	\includegraphics[scale=0.8]{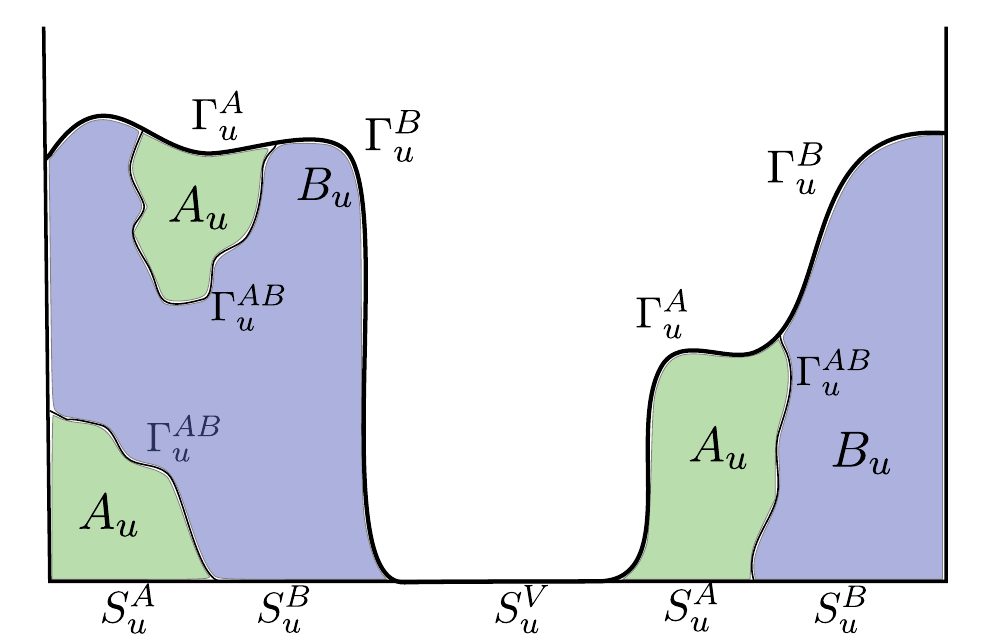}
	\caption{The different phases and interfaces of an admissible configuration.}
	\label{fig:admissibleconfigurations}
\end{figure}

We consider a configuration described by a phase variable $u$ defined in the upper half-space of $\R^n$, in general dimension $n\geq2$, taking values $+1$, $-1$, and $0$, representing the two phases $A_u=\{u=1\}$, $B_u=\{u=-1\}$, and the region above them $V_u=\{u=0\}$. For mathematical convenience we extend $u$ by a fixed value (say, $u=2$) also in the lower half-space. Having in mind the application to thin films of diblock copolymers, we will often use the terminology \emph{film} to denote the region occupied by the two phases $A_u\cup B_u$, \emph{substrate} to indicate the lower-half space, and \emph{void} to indicate the region $V_u$ above the film.

Admissible configurations are those for which the region $A_u\cup B_u$ is confined in the subgraph of a function $h_u$ over the flat substrate (see Figure~\ref{fig:admissibleconfigurations}). As customary in this kind of problems, to focus on the effect of the surface energy on the equilibrium configurations, we work with lateral periodic boundary conditions. We also impose the total volume of the film and the ratio bewteen the two constituent phases by means of two mass constraints.

We consider a sharp-interface model in which the short-range interaction energy $\g(u)$ of a configuration $u$ is assumed to be proportional to the surface measure of the interfaces between the different phases, with possibly different surface tensions. The interfaces involved are: $\Gamma_u^{AB}$ (between the two phases inside the film), $\Gamma^A_u$, $\Gamma^B_u$ (between each phase and the void), and, since also the contact between the film and the substrate costs surface energy, $S^A_u$, $S^B_u$ (between each phase and the substrate), $S_u^V$ (between the substrate and the void).

In addition to the interfacial energy $\g(u)$, we consider in the total energy a general volume-order term $\nl(u)$. For the results contained in this paper, the precise form of this term does not play a role, and the only property that we use is that $\nl(\cdot)$ is Lipschitz continuous with respect to the symmetric difference of sets, namely
\begin{equation} \label{intro:nl-lipschitz}
|\nl(u) - \nl(v)| \leq L_{\nl} \left(\, |A_u\triangle A_v|
	+ |B_u\triangle B_v|  \,\right)
\end{equation}
for some constant $L_{\nl}$.

Then the total energy $\f(u)$ of a regular configuration $u$, whose profile is given by a Lipschitz function $h_u$, writes as
\[
\begin{split}
\f(u) & \coloneqq \g(u) + \gamma\nl(u) \\
& \coloneqq \sigma_A\hn(\Gamma_u^A) + \sigma_B\hn(\Gamma_u^B) + \sigma_{AB}\hn(\Gamma^{AB}_u) \\
&\qquad + \sigma_{AS}\hn(S^A_u) + \sigma_{BS}\hn(S^B_u) + \sigma_S\hn(S^V_u) + \gamma\nl(u)
\end{split}
\]
(see Section~\ref{sect:def} for the precise definition of all the terms involved).

\subsection{Main results}
The paper is divided into two main parts, where we study several properties of the energy $\f$. In the first part we focus on its lower semicontinuity with respect to the $L^1$-topology, and we identify in Theorem~\ref{thm:relaxation} the lower semicontinuous envelope $\overline{\f}$, defined over a larger class of possibly irregular profiles.
In particular, the relaxation procedure allows to consider configurations whose free boundary is described by a function $h_u$ of bounded variation: it then might be unbounded and with jump discontinuities.

The functional $\fbar$ has the same form of the original functional $\f$, namely it is the sum of a surface energy contribution $\overline{\g}(u)$ and of the nonlocal interaction $\gamma\nl(u)$. Notice that the relaxation affects only the surface part of the energy, as the nonlocal term is continuous with respect to $L^1$-convergence. As might be expected, the new surface energy $\overline{\g}$ has relaxed surface tension coefficients, due to the possibility of reducing the energy by inserting a thin layer of a phase between two other phases (wetting). The non-standard aspect of this procedure is that, due to the additional constraints of the model (namely, the only admissible configurations are subgraphs), not all the possible infiltrations are allowed; this prevents us to apply directly the well-known results about the relaxation of surface energy of clusters in $\R^n$, see \cite{AmbBra90}.

Concerning the proof, whereas the liminf inequality follows by a standard argument adapted to our setting (see Proposition \ref{prop:lsc}), the construction of a recovery sequence requires extra care (see Proposition \ref{prop:recovery}). Indeed, first we approximate a non-regular profile by a Lipschitz one, thanks to a construction by Chambolle and Solci \cite{ChaSol07}; next, when one of the surface tension coefficients between two phases changes in the relaxation process, we need to approximate the corresponding interface by carefully inserting thin layers of the other phases, preserving both the graph constraint and the mass constraint.

Finally, the existence of a solution to the mass constrained minimization problem for the relaxed functional
\begin{equation}\label{eq:min_problem_intro}
\min \{ \fbar(u) \,:\, |A_u|+|B_u|=M,\, |A_u|=m \},
\end{equation}
where $0<m<M$, follows by a standard application of the direct method (see Theorem~\ref{thm:existence}).

\medskip
In the second part of the paper we turn our attention to the study of regularity properties of solutions to \eqref{eq:min_problem_intro}.
This is where the main mathematical challenges are, stemming from the fact that admissible competitors have to satisfy the additional condition of being subgraphs. Indeed, if no graph constraint is in force, then partial regularity of minimizing clusters could be obtained by a standard strategy, which would amount to first showing that volume-constrained minimizers are quasi-minimizers of the surface energy, and then to proving an elimination property (see \cite{Leo01}) which allows to reduce locally to the case of only two interfaces. Once this is done, partial regularity follows from classical results (see \cite{GonMasTam83}). In our case, though, we cannot apply directly those results, as they require to make \emph{arbitrary} perturbations, thus possibly exiting the restricted class of admissible configurations.
Therefore, we need to perform delicate geometric constructions, and to combine several ideas in order to prove regularity. 

We next summarize our main strategy.
In Lemma~\ref{lem:penalization} we remove the mass constraints by showing that every solution to \eqref{eq:min_problem_intro} is also a solution to a suitable penalized problem. The proof of this fact follows a rather standard contradiction argument, which amounts to show that if a minimizer of the penalized problem does not satisfy the volume constraint, then it is possible to modify it and reduce its energy - which would be a contradiction - provided that the constant in front of the penalization term is large enough. When there is just one mass constraint and the problem is in the whole space $\R^n$, this can be achieved by a suitable rescaling of the minimizer. A refined argument by Esposito and Fusco \cite{EspFus11} shows that the same can be done by a local perturbation of the set, which brings the mass of the perturbed set closer (but not necessarily equal) to the desired mass, reducing the energy at the same time. However the local variation constructed in \cite{EspFus11}  is \emph{radial}, and is not suitable in our case since the competitor that is constructed in this way might not satisfy the graph constraint. Instead, we perform a local rescaling in the \emph{vertical} direction, so that the perturbed configuration remains the subgraph of an admissible profile and can be used to contradict the minimality of the starting configuration. Another relevant difference is that in our case two mass constraints are in force; we can however avoid the use of the Implicit Function Theorem (used in arguments like that in \cite[Lemma~29.14]{Mag}) and deal with the two constraints one at a time.

The fact that minimizers solve a penalized minimum problem, together with the Lipschitz continuity of the nonlocal energy, immediately implies (see Proposition~\ref{prop:quasimin}) that every solution $u$ to \eqref{eq:min_problem_intro} is a quasi-minimizer of the surface energy $\overline{\g}$, in the sense that there exists $\Lambda>0$ such that
\begin{equation}\label{eq:quasi_minimality_intro}
\overline{\g}(u) \leq \overline{\g}(v) + \Lambda \bigl( |A_u\asymm A_v| + |B_u\asymm B_v| \bigr),
\end{equation}
for all admissible competitors $v$. Notice that in this formulation, admissible competitors do not have to obey the mass constraints, but they still have to satisfy the graph constraint, and thus the regularity of quasi-minimizers does not follow directly from classical results. We denote by $\mathcal{A}_{\Lambda,M}$ the class of quasi-minimizers satisfying the inequality \eqref{eq:quasi_minimality_intro} and with total mass $M$, see Definition~\ref{def:quasimin}.
By using \eqref{eq:quasi_minimality_intro} we then show that $h_u$ is bounded, see Proposition~\ref{prop:truncation}.

The next main result, which is proved in Subsection~\ref{subsection:regularity2} through a series of propositions, concerns the regularity of quasi-minimizers in dimension $n=2$. In view of the previous discussion, it applies in particular to any solution of the minimum problem \eqref{eq:min_problem_intro}.

\begin{theorem}[Partial regularity in dimension $n=2$] \label{thm:regularity}
Assume that $n=2$, and that the surface tension coefficients satisfy the strict triangle inequalities 
\begin{equation}\label{eq:coeff_triangle_inequality}
\sigma_{AB} < \sigma_A + \sigma_B, \quad\quad
\sigma_{A} < \sigma_B + \sigma_{AB}, \quad\quad
\sigma_{B} < \sigma_A + \sigma_{AB}.
\end{equation}
Let $u\in\mathcal{A}_{\Lambda,M}$ be a quasi-minimizer, according to Definition~\ref{def:quasimin}. Then the followings hold.
\begin{enumerate}
\item\label{item1-reg} \emph{(Infiltration)} There exists $\e_0>0$ (depending only on $M$, $\Lambda$, and the surface energy coefficients) such that, for any square $\q_r(z_0)$ centered at $z_0\in\R^n$ with side length $r\in(0,1)$, the following implications hold:
\[
|V_u\cap \q_r(z_0)| < \e_0 r^2 \quad\quad
\Rightarrow
\quad\quad
|V_u\cap \q_{\frac{r}{2}}(z_0)| = 0,
\]
and, if $\q_r(z_0)$ does not intersect the substrate,
\[
|(A_u\cup B_u) \cap \q_r(z_0)| < \e_0 r^2 \quad\quad
\Rightarrow
\quad\quad
|(A_u\cup B_u) \cap \q_{\frac{r}{2}}(z_0)| = 0.
\]

\item\label{item2-reg} \emph{(Lipschitz regularity of the graph)} There exists a finite set $\Sigma$, containing the jump points of $h_u$, such that $h_u$ is locally Lipschitz outside $\Sigma$. 

\item\label{item3-reg} \emph{(Singular set)} At the upper end of a jump point of $h_u$, the graph has a vertical tangent. At the points of $\Sigma$ that are not jump points of $h_u$, the left or the right derivative of $h_u$ is infinite. The graph of $h_u$ does not contain interior or exterior cusps.

\item\label{item4-reg} \emph{(Internal regularity of $\Gamma^{AB}_u$)} For every $\alpha\in(0,1/2)$ the interface $\partial A\cap\partial B$ is a locally a $C^{1,\alpha}$-curve in $\{(x,y)\in\R^2 : 0<y<h_u(x)\}$.

\item\label{item5-reg} \emph{($C^{1,\alpha}$-regularity of the graph)} If $x_0\notin\Sigma$ is such that $(x_0,h_u(x_0))\in\partial^*A\cup\partial^*B$, then $h_u$ is of class $C^{1,\alpha}$ in a neighbourhood of $x_0$, for every $\alpha\in(0,1/2)$.
\end{enumerate}
\end{theorem}

Conditions \eqref{eq:coeff_triangle_inequality} are known to be needed in order to get regularity for minimizing clusters (see \cite{Leo01, Whi96}). Indeed, consider the simple case of a flat interface between the phase $A$ and the void $V$: if for instance one had $\sigma_{A}=\sigma_{B}+\sigma_{AB}$, then it would be energetically equivalent to insert a thin layer of the phase $B$ between $A$ and $V$, in such a way that these two phases do not touch anymore. In other words, the \emph{strict} triangular inequalities are natural conditions to prevent small infiltrations between pair of phases.

The elimination property is well-known in the case of minimal clusters (see \cite{Leo01}). The idea of the proof is to construct a suitable competitor by \emph{filling} the minority phase in $\mathcal{Q}_r(z_0)$ with one of the other phases. Again, in our case filling $A_u$ or $B_u$ by $V_u$ might lead to a configuration which violates the graph constraint. Therefore, the proof of the infiltration for $V_u$ (Proposition~\ref{prop:infiltrationV}) and for $A_u\cup B_u$ (Proposition~\ref{prop:infiltrationAB}) uses a two step strategy: first, we prove the elimination property in a semi-infinite strip, where it is possible to fill $A_u\cup B_u$ with $V_u$, without violating the graph constraint; then, we show that a minimal configuration having small volume percentage of the void (or of the subgraph) in a cube must necessarily have a small volume percentage of the same in the semi-infinite strip, so that it is possible to conclude by using the first step.

The proof of the Lipschitz regularity follows an idea by Chambolle and Larsen \cite{ChaLar03} (see also \cite{FonFusLeoMor07,FusMor12}): we show an \emph{interior ball condition} (see Proposition~\ref{prop:innerball}), namely that there exists a uniform radius $\rho_0>0$ such that, for each $z$ on the graph of $h_u$, it is possible to find a ball with radius $\rho_0$ tangent to the graph of $h_u$ only at the point $z$ and contained in the subgraph of $h_u$. This property implies (Proposition~\ref{prop:lipschitz}) that $h_u$ has only a finite number of jump points, and that $h_u$ is locally Lipschitz continuous outside a finite set (where the inner ball is tangent to the graph horizontally).

Since in two dimensions the graph $h_u$ is closed, for each point $z$ on the internal interface between the two phases it is possible to find a ball centered at $z$ that does not intersect the graph, nor the substrate. Therefore, since internal interfaces do not have any graph constraint to satisfy, their $C^{1,\alpha}$-regularity follows from classical results (see Remark~\ref{rem:interior_regularity}).

Finally, the proof of the $C^{1,\alpha}$ regularity of the graph (Proposition~\ref{prop:C1alpha_graph}) is also based on an elimination property for the two sets $A_u$, $B_u$ separately. To obtain this, we observe that thanks to the Lipschitz regularity of $h_u$, for every point $(x_0,h_u(x_0))\in\partial^*A\cup\partial^*B$ with $x_0\not\in\Sigma$ we can find a rectangle such that the graph of $h_u$ does not intersect its upper and lower sides. This property allows to perform a local perturbation which preserves the graph constraint.

\subsection{Application to thin films of diblock copolymers} \label{subsection:thinfilms}

We now discuss a possible application of the variational model considered here for the description of the morphology of optimal or equilibrium configurations of thin films of diblock copolymers \blu{under some additional assumptions that will be discussed later.}

Block copolymers are an important class of soft materials (see \cite{BatFre99}). They are composed by chemically bonded linear chains of monomers. The competition between the repulsion among different subchains and the entropy cost associated with chain stretching is the mechanism behind the extraordinary self-assembly property of block copolymers, that leads to the creation of fascinating patterns exhibiting interesting periodicity properties (see \cite{ThoAndHenHof88}).

When block copolymers are constrained in a thin film, the landscape of observed configurations can be significantly different from that of the bulk case, due to the influence of film surfaces and the interactions of the blocks with the interfaces.
It is indeed observed that in the vicinity of an external interface the microdomains tend to align parallel to that surface \cite{Fre87}. As noted in the physical literature, ``\emph{as film thickness decreases, a regime may be encountered where the constraining effects of both interfaces are felt throughout the film and a transition from the bulk, 3D morphology to a 2D thin film morphology may result}'' \cite{RadCarTho96}.

An important distinction must be made between \emph{unconfined films} supported by a solid, flat substrate, where one interface of the film is free, and \emph{confined films}, where the copolymers and constrained between two hard walls with a fixed thickness.
The behaviour in these cases is usually illustrated (see \cite{Mat97}) by considering symmetric diblock copolymers, where the preferred bulk configuration is lamellar: in this case, the copolymer tends to form multilayered structures of lamellae parallel to the interface, in which each period ($L$) consists of two monolayers. This induces a quantization of the film thickness, which is forced to be a multiple of the natural spacing of the lamellae $H\approx\frac{kL}{2}$, with $k$ even if the upper and lower surfaces have an affinity for the same component of the diblock copolymer, and $k$ odd if the two surfaces have opposite affinities. However, when the film thickness $H$ and the natural spacing $L$ are not commensurate, this causes compression of the chain of polymers, namely stress in the film, that in the unconfined case is released by locally modifying the thickness of the profile by forming terraces (see Figure~\ref{fig:formation}, top-right), islands and holes (see Figure~\ref{fig:formation}, bottom-left); in the confined case, the frustration is relieved by changing the orientation of the lamellae (see Figure~\ref{fig:formation}, bottom-right).

Besides lamellar configurations, other structures have been observed for asymmetric block copolymers, like for instance spherical \cite{YokMatKra00} or cylindrical \cite{VanVan95} mesophases, which show the same phenomena of thickness quantization or change in orientation of the microdomains. See also \cite{Dar07} for a review of the possible phases that have been observed and a discussion of their many applications, and \cite[Figure~4]{HuaZhuMan21} for an illustration of the phase diagram in the case of confined films, showing twenty different morphologies depending on the volume fraction and on the film thickness. The possibility of accessing a larger class of equilibrium configurations has been exploited for many applications (see \cite{Seg05}), ranging from litography to mass transport.
\blu{Patterns in thin films of block copolymers have been investigated numerically (see, for instance, \cite{HillMill17,LyaHorZveSev06,Mat97,Mat98,ParKayMun,Sta12}).}

\begin{figure}
	\centering
	\includegraphics[scale=0.7]{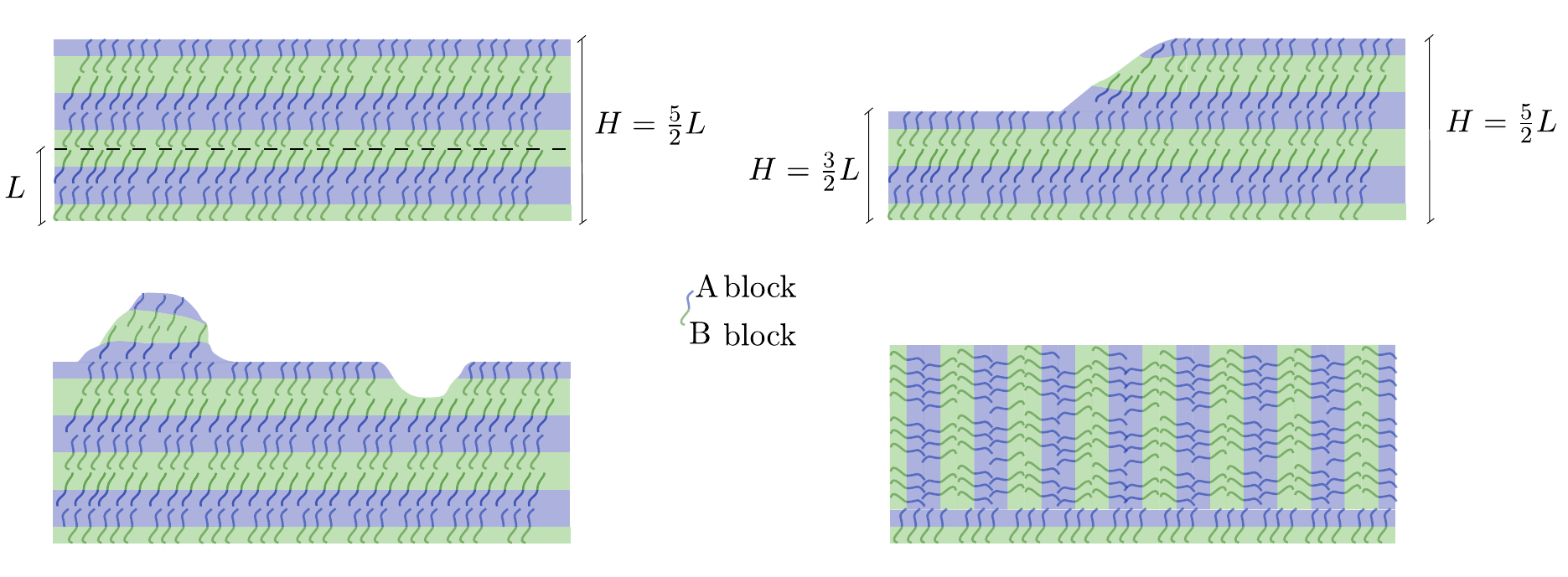}
	\caption{Schematic cartoon of three different ways in which a film of symmetric diblock copolymers can release the stretching stress caused by imposing an initial thickness $H$ (see top-left) that is not commensurate with the natural spacing of lamellae $L$.}
	\label{fig:formation}
\end{figure}

\medskip

Mathematical models aimed at describing the behaviour of block copolymers from physics and chemistry can be roughly divided into two categories: (self-consistent) mean fields models (see, for instance, \cite{MatBat96,MatSch94}) and density functional theory models.
A celebrated mean field model for block copolymers was derived by Ohta and Kawasaki in \cite{OhtaKaw86} for the case of diblock copolymers (two monomers) in the strong segregation regime by using several approximations (infinite temperature and thermodynamic limit). It has successfully been used to derive qualitative properties related to both the dynamics and the statics of diblock copolymers. In mathematical terms, the Ohta-Kawasaki is a phase-field model given by the sum of a Cahn-Hilliard-type functional (replaced by a perimeter term in the sharp-interface version) and a nonlocal interaction term.
\mb{
By using a notation similar to the one implemented above, such an energy can be written in the form
\begin{equation} \label{intro:OK}
\sigma_{AB}\hn(\Gamma^{AB}_u) + \gamma\nl(u),
\end{equation}}
where the first term models the short-range interaction between different monomers, related to the surface energy of the interfaces dividing the regions of high concentration of the two monomer species, while the second represents their long-range interaction. The emergence of highly nontrivial pattern configurations at a mesoscopic scale is precisely due to the competition between these two kinds of energies.

\mb{The model considered in this paper can be viewed as a variant of the Ohta-Kawasaki model suitable to describe thin films of diblock copolymers in the unconfined case. The two regions $A_u$, $B_u$ of an admissible configuration $u$ represent the two phases of the diblock copolymers, and are confined by a solid, flat substrate on the bottom side, while the upper surface is exposed.
As a volume-order term $\nl(u)$ in the total energy, we consider a long-range interaction responsible for the repulsion force between different monomers, and thus acting only on the two sets $A_u$ and $B_u$ - see Example~\ref{example:nlenergy} for the precise definition of $\nl(u)$. The region above the film (which could be void, air or a liquid solvent) is modeled as a homopolymer, in the framework of the density functional theory for blends of diblock copolymers with homopolymers derived by Choksi and Ren \cite{ChoRen05} (see also \cite{BonKnu16, GenPel08, GenPel09} for related studies in the mathematical literature), and it only interacts with the diblock copolymer via the surface energies.}

As it can be seen by looking \mb{at the Ohta-Kawasaki energy \eqref{intro:OK}}, surface effects with the exterior are usually neglected in models for block copolymers in the bulk, since they are of several orders of magnitude lower than the other effects considered.
When confined in thin films, though, the surface interactions of the two phases with the substrate and with the air (i.e., the additional terms in the energy $\f$ \mb{compared with \eqref{intro:OK}}) become important. This is how, at least heuristically, the change in the energy landscape is justified in the physics literature.
Under the additional assumption that the configurations of interest can be described by a graph over the substrate, the model consider in this paper could be of help in the study of such a class of equilibrium stable configurations of block copolymers confined in thin films.
We would like to thank the anonymous referees for pointing out that this latter additional assumption is not easily justified from the physical point of view. Indeed, despite the fact that in the physical literature authors refer to the \emph{thickness} of the film, this does not exclude the possibility of having a film with \emph{holes}, or arranging with \emph{tubes} that violate the graph constrain. We were not able to find any paper, either in the physical or in the experimental literature that clearly disregard such possibility.

Finally, we would like to point out that our model is not a dimension reduction model, like that investigated in \cite{DeSKohMueOtt}.

\subsection{Remarks}
We conclude this introduction with a few more remarks.
The extension to the case of more than two phases is relatively straightforward and the arguments presented here can be directly generalized, at the price of a more demanding notation and of a larger number of different cases to be taken into consideration. It could also be possible to extend our results to different kinds of boundary conditions, or if surface interactions with horizontal walls are presents.

The proofs of the results in this paper follow several well known arguments used to treat similar problems and most of the techniques are fairly standard. However, the implementation of such ideas in our context, where the graph constraint is in force, poses several additional challenges, mainly due to the need of adapting, and in some cases significantly modifying, the construction of suitable competitors. This is particularly relevant for the construction of a recovery sequence (Proposition~\ref{prop:recovery}), the local deformation map used in the proof of the penalization argument (Lemma~\ref{lem:penalization}), and the elimination-type properties (Proposition~\ref{prop:infiltrationV} and Proposition~\ref{prop:C1alpha_graph}).

The generalization to higher dimensions of the two-dimensional regularity theory for quasi-minimal partitions subject to a graph constraint, developed in Section~\ref{subsection:regularity2}, is not straightforward and would require new ideas: while we believe that the elimination property might be obtained by refined but similar arguments, the inner ball condition leading to the Lipschitz regularity of the graph is a purely two-dimensional strategy.

Future directions of investigation are the following: firstly, as discussed above, the extension of the regularity results to higher dimensions (in particular in the physical dimension three), and the investigation of finer regularity properties in two dimension; secondly, a description of optimal configurations by means of the Euler-Lagrange equations satisfied by a minimizer, which involve an interplay between the curvature and the nonlocal potential on the regular parts of the interfaces; thirdly, an analysis of the possible singularities (jump points, points where three different interfaces meet, and points where an interface meets the substrate), in particular by deriving rigorously Young's law for the triple points. Finally, an ambitious task would be to study specific configurations (for instance lamellar patterns with terrace formations) and investigate their stability properties, possibly by means of second variation arguments.

\medskip\noindent\textbf{Structure of the paper.}
The paper is organized as follows. In Section~\ref{sect:def} we introduce the main notation, the class of admissible configurations and the total energy of the system. In Section~\ref{sect:relax} we compute the relaxation of the energy (Theorem~\ref{thm:relaxation}) and we use this result to prove the existence of minimizing configurations (Theorem~\ref{thm:existence}). In Section~\ref{sect:regularity} we first show that solutions to the minimum problem \eqref{eq:min_problem_intro} are quasi-minimizers of the surface energy under a graph constraint (Subsection~\ref{subsection:penalization}), and then we prove Theorem~\ref{thm:regularity} on the regularity of quasi-minimizers in dimension two (Subsection~\ref{subsection:regularity2}).


\section{The model} \label{sect:def}


\subsection{Notation for functions of bounded variation and perimeters}
The profile of the film will be modeled by the (generalized) graph of a periodic function with finite total variation in $(0,L)^{n-1}$ ($n\geq2$), where $L>0$ is a fixed parameter, and its subgraph will represent the reference configuration of the film. We therefore firstly recall a few notions from the theory of BV-functions (see \cite{AFP}), in order to fix the notation used in the paper. Given $h\in L^1_{\loc}(\Omega)$, where $\Omega\subset\R^{m}$ is an open set ($m\geq1$), its total variation is defined as
\begin{equation*}
|Dh|(\Omega) \coloneqq \sup\biggl\{ \int_{\Omega}h\,\div\phi\de x \,:\, \phi\in C^\infty_{\mathrm c}(\Omega;\R^{m}),\, |\phi|\leq 1 \biggr\} \,,
\end{equation*}
and this quantity is finite if and only if the distributional derivative $Dh$ of $h$ is a bounded Radon measure on $\Omega$. We let $\BV(\Omega) \coloneqq \{ h\in L^1(\Omega) \,:\, |Dh|(\Omega)<\infty \}$. If $h\in\BV(\Omega)$, at each point $x\in\Omega$ the approximate upper and lower limits
\begin{equation} \label{h+-}
\begin{split}
h^+(x) &\coloneqq \inf\biggl\{ t\in\R \,:\, \limsup_{\rho\to0}\frac{ \leb^m( \{ h>t\}\cap B_\rho(x) )}{\omega_m\rho^m} = 0 \biggr\} \,,\\
h^-(x) &\coloneqq \sup\biggl\{ t\in\R \,:\, \limsup_{\rho\to0}\frac{ \leb^m( \{ h<t\}\cap B_\rho(x) )}{\omega_m\rho^m} = 0 \biggr\}
\end{split}
\end{equation}
are well-defined, where $\leb^m$ is the $m$-dimensional Lebesgue measure, $B_\rho(x)\subset\R^m$ is the ball centered at $x$ with radius $\rho$, and $\omega_m=\leb^m(B_1(0))$. The \emph{jump set} of $h$ is then defined as the set
\begin{equation} \label{jump}
J_h \coloneqq \{ x\in\Omega \,:\, h^-(x)<h^+(x)\}\,,
\end{equation}
and it is well-known that $J_h$ is a $(\mathcal{H}^{m-1},m-1)$ rectifiable set, with normal $\nu_h(x)$ at $\mathcal{H}^{m-1}$-a.e. point $x\in J_h$.

We also recall that a set $E\subset\Omega$ has finite perimeter in $\Omega$ if $|D\chi_E|(\Omega)<\infty$, where $\chi_E(x)=1$ if $x\in E$, $\chi_E(x)=0$ if $x\notin E$; the \emph{perimeter} of $E$ in $\Omega$ is then defined as
\begin{equation} \label{perimeter}
\per(E;\Omega) \coloneqq |D\chi_E|(\Omega) \,.
\end{equation}
We introduce the \emph{essential boundary} of $E$
\begin{equation} \label{essboundary}
\partial_eE \coloneqq \Omega\setminus (E^0\cup E^1),
\end{equation}
where, for $t\in[0,1]$, $E^t$ denotes the set of points where $E$ has Lebesgue density $t$.
Another relevant subset of the boundary of a set of finite perimeter is the \emph{reduced boundary} $\partial^*E$ (see \cite{AFP}). At every point of the reduced boundary the measure-theoretic outer normal $\nu_E$ is defined, the Lebesgue density of $E$ is equal to $1/2$, and it is well-known that $\partial_eE$ coincides with $\partial^*E$ up to a $\mathcal{H}^{m-1}$-negligible set. We finally recall that a \emph{Caccioppoli partition} of $\Omega$ is a finite partition $\{E_i\}_{i\in\{1,\ldots,N\}}$ of $\Omega$, $N\in\N$, such that $\sum_{i=1}^N\per(E_i;\Omega)<+\infty$. For a Caccioppoli partition $\{E_i\}_{i}$ , $\mathcal{H}^{m-1}$-a.e. point of $\Omega$ belongs to one of the sets $(E_i)^1$ or to one of the intersections $\partial^*E_i\cap\partial^*E_j$ ($i\neq j$).


\subsection{Admissible configurations}
We now describe the class of admissible configurations. Throughout the paper, we will denote by $x=(x',x_n)$ the generic point in $\R^n\equiv\R^{n-1}\times\R$, and by $\R^n_+\coloneqq \R^{n-1}\times[0,\infty)$. The canonical basis of $\R^{n}$ will be denoted by $(e_1,\ldots,e_{n})$, and the Lebesgue measure on $\R^n$ by $|\cdot|\coloneqq\leb^n(\cdot)$. Given $L>0$, we also set
\begin{equation} \label{QL}
Q_L\coloneqq [0,L)^{n-1} \subset\R^{n-1}, \qquad Q_L^+\coloneqq Q_L\times[0,+\infty).
\end{equation}

We assume that the substrate occupies the infinite region
\begin{equation} \label{substrate}
S\coloneqq\R^{n-1}\times(-\infty,0).
\end{equation}
We introduce the class of \emph{admissible profiles}
\begin{equation} \label{AP}
\begin{split}
\ap(Q_L) \coloneqq &\Bigl\{ h:\R^{n-1}\to[0,+\infty) \,:\, h\in\BV_{\loc}(\R^{n-1}), \, h \text{ is $Q_L$-periodic} \Bigr\} \,.
\end{split}
\end{equation}
The reference configuration of the film is represented by the subgraph of an admissible profile $h\in\ap(Q_L)$: we denote it and its periodic extension by
\begin{equation} \label{omega}
\begin{split}
\Omega_h &\coloneqq \Bigl\{ (x',x_n)\in Q_L\times\R \,:\, 0<x_n<h(x') \Bigr\}\,, \\
\Omega_h^\# &\coloneqq \Bigl\{ (x',x_n) \in \R^{n-1}\times\R \,:\, 0<x_n<h(x') \Bigr\} \,,
\end{split}
\end{equation}
respectively. Notice that, as $h$ has finite total variation, the set $\Omega_h$ has finite perimeter.
We also define, for $h\in\ap(Q_L)$, the \emph{free profile}
\begin{equation} \label{gamma}
\Gamma_h \coloneqq \Bigl\{ (x',x_n)\,:\,  x'\in Q_L,\, h^-(x')\leq x_n \leq h^+(x') \Bigr\} \,,
\end{equation}
and we denote by $\Gamma_h^\#$ its periodic extension.
Notice that if $0<x_n<h^-(x')$ then $x\in(\Omega_h^\#)^1$, while if $x_n>h^+(x')$ then $x\in(\Omega_h^\#)^0$; therefore $\partial_e(\Omega_h^\#\cup S)$ is a subset of $\Gamma_h^\#$ (and coincides with $\Gamma_h^\#$ up to a $\mathcal{H}^{n-1}$-negligible set).

The region $\Omega_h$ occupied by the film is \rosso{partitioned into} two disjoint sets of finite perimeter $A$, $B$ representing the two phases of the system. We identify these two phases with the level sets of a marker function $u:\Omega_h\to\{\pm1\}$ with bounded variation, so that $A=\{u=1\}$ and $B=\{u=-1\}$. As $\Omega_h$ is in general not an open set, it will be convenient to consider $u$ as a piecewise constant function defined in the full space $\R^n$, taking two additional values $u=0$ and $u=2$ in the region above the film and in the substrate, respectively.
This is made precise by the following definition.

\begin{definition}[Admissible configurations] \label{def:admissibleconfig}
Let $I\coloneqq\{\pm1,0,2\}$. The class $\ac$ of \emph{admissible configurations} is the space of functions $u:\R^n\to I$ satisfying the following properties:
\begin{enumerate}
	\item $u\in\BV_{\loc}(\R^n;I)$,
	\item $u(x'+Le_i,x_n)=u(x',x_n)$  for all $(x',x_n)\in\R^n$, $i=1,\ldots,n-1$,
	\item there exists $h_u\in\ap(Q_L)$ such that $\Omega_{h_u}^\#=\{u=1\}\cup\{u=-1\}$,
	\item $S=\{u=2\}$, where $S$ is the substrate defined in \eqref{substrate}
\end{enumerate}
(the previous identities have to be understood in the almost everywhere sense with respect to $\leb^n$).
The class of \emph{regular admissible configurations} is defined as
\begin{equation} \label{acreg}
\acreg \coloneqq \Bigl\{ u\in\ac \,:\, h_u \text{ is Lipschitz continuous} \Bigr\} .
\end{equation}
We consider the space $\ac$ endowed with the $L^1$-convergence: we say that a sequence $\{u_k\}_{k\in\N}\subset\ac$ converges in $\ac$ to $u\in\ac$ if $u_k\to u$ in $L^1(Q_L^+)$.
\end{definition}

Given an admissible configuration $u\in\ac$, we have a partition of the strip $Q_L^+$ into three sets of finite perimeter, which will be denoted by
\begin{equation} \label{ABV}
A_u\coloneqq\{ u= 1\}\cap Q_L^+, \qquad B_u\coloneqq\{ u=-1\}\cap Q_L^+, \qquad V_u\coloneqq\{u=0\}\cap Q_L^+,
\end{equation}
and as usual we will denote by $A_u^\#$, $B_u^\#$ and $V_u^\#$ their periodic extensions.
Notice that $A_u\cup B_u=\Omega_{h_u}$ and $\partial^*V_u^\#=\Gamma_{h_u}^\#$ (up to a $\hn$-negligible set).
In other words, the admissible configurations are just periodic partitions of the upper half-space into three sets of locally finite perimeter $A$, $B$, $V$, with the constraint that $A\cup B$ is the subgraph of a $\BV$-function. The jump set $J_u$ of $u$ coincides (up to a $\hn$-negligible set) with the union of their reduced boundaries:
\begin{equation} \label{interfaces0}
\hn\biggl( J_u\setminus \bigcup_{\substack{i,j\in I\\i\neq j}}\partial^*\{u=i\} \cap \partial^*\{u=j\} \biggr) = 0,
\end{equation}
with $(u^+(x),u^-(x)) = (i,j)$ (up to a permutation) for every $x\in\partial^*\{u=i\}\cap\partial^*\{u=j\}$.

As we want to consider different values of the surface tension for all the possible different interfaces between the phases, it is convenient to introduce the following notation (see Figure~\ref{fig:admissibleconfigurations}):
\begin{equation} \label{interfaces1}
\Gamma_u^A \coloneqq \Gamma_{h_u}\cap \partial^*A_u^\#\,, \qquad \Gamma_u^B \coloneqq \Gamma_{h_u}\cap \partial^*B_u^\#\,, \qquad \Gamma^{AB}_u \coloneqq \partial^*A_u^\#\cap\partial^*B_u^\#\cap Q_L^+\,,
\end{equation}
and
\begin{equation} \label{interfaces2}
S^A_u \coloneqq \partial^*A_u\cap\bigl(Q_L\times\{0\}\bigr), \quad S^B_u \coloneqq \partial^*B_u\cap\bigl(Q_L\times\{0\}\bigr), \quad S^V_u\coloneqq \Gamma_{h_u}\cap\bigl(Q_L\times\{0\}\bigr).
\end{equation}
The set $S^V_u$ represents the possible region in which the substrate is exposed. In view of \eqref{interfaces0}, the disjoint union of these interfaces coincides with the jump set $J_u$ of $u$ inside the periodicity strip:
\begin{equation} \label{interfaces3}
J_u\cap Q_L^+ = \Gamma_u^A\cup\Gamma_u^B\cup\Gamma^{AB}_u\cup S^A_u\cup S^B_u\cup S^V_u \cup N
\end{equation}
with $\hn(N)=0$.


\subsection{The energy of regular configurations}
We now introduce the energy associated with a regular configuration $u\in\acreg$. This energy will be extended to the whole space $\ac$ of admissible configurations in Section~\ref{sect:relax} via a relaxation procedure. \rosso{The total energy is the sum of a surface penalization of the interfaces between the phases, with possibly different surface tension coefficients, and a general volume-order, possibly nonlocal, energy $\nl(u)$. Here and the rest of the paper we will always assume that $\nl:\ac\to[0,+\infty)$ is a given function satisfying the following property: given a positive constant $M>0$, there exists $L_{\nl}\in(0,+\infty)$, depending on $M$, such that for all $u,v\in\ac$ with $|\Omega_{h_u}|,|\Omega_{h_v}|\leq M$ it holds
\begin{equation} \label{eq:ass-nl}
|\nl(u) - \nl(v)| \leq L_{\nl} \left(\, |A_u\triangle A_v| + |B_u\triangle B_v|  \,\right).
\end{equation}
Under the previous assumption, we can define the energy of a regular configuration as follows.}

\begin{definition}[Energy] \label{def:energy}
Given positive coefficients $\sigma_A, \sigma_B, \sigma_{AB}, \sigma_{AS}, \sigma_{BS}, \sigma_S,\gamma>0$, we define the total energy of a regular configuration $u\in\acreg$ as
\begin{equation} \label{energy}
\begin{split}
\f(u) &\coloneqq \sigma_A\hn(\Gamma_u^A) + \sigma_B\hn(\Gamma_u^B) + \sigma_{AB}\hn(\Gamma^{AB}_u) + \gamma\,\nl(u) \\
&\qquad + \sigma_{AS}\hn(S^A_u) + \sigma_{BS}\hn(S^B_u) + \sigma_S\hn(S^V_u) \,.
\end{split}
\end{equation}
\end{definition}

By introducing the surface energy density
\begin{equation} \label{density}
\Psi(i,j)\coloneqq
\begin{cases}
\sigma_A & \text{if } (i,j)=(1,0),\\
\sigma_B & \text{if } (i,j)=(-1,0),\\
\sigma_{AB} & \text{if } (i,j)=(1,-1),\\
\sigma_{AS} & \text{if } (i,j)=(1,2),\\
\sigma_{BS} & \text{if } (i,j)=(-1,2),\\
\sigma_{S} & \text{if } (i,j)=(0,2),
\end{cases}
\end{equation}
with $\Psi(i,j)=\Psi(j,i)$, we can write in a more compact notation an equivalent representation of the energy in terms of the jump set of the piecewise constant function $u$ (see \eqref{interfaces3}):
\begin{equation} \label{energy2}
\begin{split}
\f(u) = \int_{J_u\cap Q_L^+} \Psi(u^+,u^-)\de\hn + \gamma\,\nl(u).
\end{split}
\end{equation}

\rosso{
\begin{example}[Thin films of diblock copolymer] \label{example:nlenergy}
As discussed in the Introduction, a possible application of the variational model introduced above is in the description of thin films of diblock copolymers. In this case, the sets $A_u$ and $B_u$ associated with an admissible configuration $u\in\ac$ represent the two phases occupied by the diblock copolymer, and $V_u$ represents the void (or homopolymer) above the film.

Following \cite{ChoRen05} and modeling the phase $V_u$ as an homopolymer, one can introduce a nonlocal interaction energy $\nl(u)$ between the two phases $A_u$, $B_u$ as follows. For $u\in\ac$ we let $\bar{u}\coloneqq\int_{Q_L^+}u(x)\de x= |A_u|-|B_u|$, and we define
\begin{equation} \label{nlenergy}
\nl(u) \coloneqq \int_{Q_L^+} |\nabla \phi_u(x)|^2\de x \,,
\end{equation}
where the potential $\phi_u:Q^+_L\to\R$ associated to the configuration $u\in\ac$ is the solution to
\begin{equation*} \label{potential1}
-\Delta \phi_u = u-\bar{u} \quad\text{in }Q_L^+, \qquad\qquad \int_{Q_L^+}\phi_u(x)\de x=0,
\end{equation*}
with periodic boundary conditions on the lateral boundary $\partial Q_L\times(0,+\infty)$ and zero Neumann boundary condition at the interface $Q_L\times\{0\}$ with the substrate. By arguing as in \cite[Lemma~2.6]{AceFusMor13}, one can check that $\nl(u)$ obeys the assumption \eqref{eq:ass-nl}.
\end{example}}


\section{Relaxation and existence of minimizers} \label{sect:relax}

The goal of this section is to compute the lower semicontinuous envelope $\fbar$ of the functional $\f$ with respect to the convergence in $\ac$, under a volume constraint: for every $u\in\ac$
\begin{multline} \label{relaxation}
\fbar(u) \coloneqq \inf\,\Bigl\{ \liminf_{k\to\infty}\f(u_k)\,:\, u_k\in\acreg,\, |A_{u_k}|=|A_u|,\, |B_{u_k}|=|B_u|,\, u_k\to u \text{ in }\ac \Bigr\}.
\end{multline}
In the following theorem, which is proved in Subsections~\ref{subsection:liminf} and \ref{subsection:limsup}, we give a representation formula for the relaxed functional $\fbar$.

\begin{theorem}[Relaxation] \label{thm:relaxation}
Assume that $\sigma_{AB}\leq\sigma_A+\sigma_B$. Then the functional $\fbar$, defined in \eqref{relaxation}, is given by
\begin{equation} \label{relaxation2}
\fbar(u) = \int_{J_u\cap Q_L^+} \overline{\Psi}(u^+,u^-)\de\hn + \gamma\,\nl(u)
\end{equation}
for all $u\in\ac$, where
\begin{equation} \label{densityrelax1}
\overline\Psi(i,j)\coloneqq
\begin{cases}
\bar{\sigma}_A\coloneqq\min\{ \sigma_A, \sigma_B + \sigma_{AB} \} & \text{if } (i,j)=(1,0),\\
\bar{\sigma}_B\coloneqq\min\{ \sigma_B, \sigma_A + \sigma_{AB} \} & \text{if } (i,j)=(-1,0),\\
\sigma_{AB} & \text{if } (i,j)=(1,-1),\\
\min\{ \sigma_{AS}, \sigma_{BS} + \sigma_{AB} \} & \text{if } (i,j)=(1,2),\\
\min\{ \sigma_{BS}, \sigma_{AS} + \sigma_{AB} \} & \text{if } (i,j)=(-1,2),\\
\min\{ \sigma_{S}, \sigma_{AS}+\bar\sigma_A, \sigma_{BS}+\bar\sigma_B \} & \text{if } (i,j)=(0,2),
\end{cases}
\end{equation}
and $\overline{\Psi}(i,j)=\overline{\Psi}(j,i)$.
\end{theorem}

\begin{remark}
From the proof of Theorem~\ref{thm:relaxation}, it also follows that the representation formula \eqref{relaxation2} continues to hold if we drop the mass constraints in the definition \eqref{relaxation} of $\fbar$.
\end{remark}

\begin{remark}
The assumption $\sigma_{AB}\leq\sigma_A+\sigma_B$ prevents the possibility of reducing the energy by inserting of a thin layer of void between the phases $A$ and $B$. \rosso{In the diblock copolymer application, this is} justified by the fact that the subchains of type $A$ and $B$ are chemically bonded together. In case the opposite inequality holds, the relaxed functional would have a different surface tension ($\sigma_A+\sigma_B$) only for the vertical interfaces of $\Gamma^{AB}$ connected to the graph.
\end{remark}

\begin{remark}
The choice of the $L^1$ topology is justified by the fact that, \rosso{in the application we have in mind,} we do not consider elastic effects, that would lead to cracks inside the copolymer phases. In case these effects have to be taken into account, a natural topology would be the Hausdorff convergence of the epigraph of the profile, as in \cite{BonCha02,FonFusLeoMor07}; the corresponding relaxed functional would contain additional terms accounting for vertical cracks, connected to the free profile of the film, inside the two phases.
\end{remark}

The existence of minimizers of the relaxed functional $\fbar$ follows by a standard application of the direct method of the Calculus of Variations. We fix two positive real numbers $M>0$ and $m\in(0,M)$, which represent the total volume of the film and the volume of the phase $A$, respectively. 

\begin{theorem}[Existence of minimizers] \label{thm:existence}
Under the assumptions of Theorem~\ref{thm:relaxation}, the constrained minimization problem
\begin{equation} \label{eq:min1}
\min \bigl\{ \fbar(u)  \,:\,  u\in\ac,\, |A_u|=m,\, |B_u|=M-m \bigr\}
\end{equation}
admits a solution. Furthermore, if $\bar{u}\in\ac$ is a solution of the above problem, then
\begin{equation} \label{eq:min2}
\fbar(\bar{u})=\inf \bigr\{ \f(u)  \,:\, u\in\acreg,\, |A_u|=m,\, |B_u|=M-m \bigr\}.
\end{equation}
\end{theorem}

The remaining part of this section is devoted to the proof of Theorem~\ref{thm:relaxation}.
Since by \rosso{assumption \eqref{eq:ass-nl}} the term $\nl$ in the energy is continuous with respect to the convergence in $\ac$, it is sufficient to compute the relaxation of the surface energy. This is proved, as usual, in two steps: denoting by $\fc$ the right-hand side of \eqref{relaxation2}, in the first step (Proposition~\ref{prop:lsc}) it is shown that the energy $\fc(u)$ is smaller than the liminf of the energies of every sequence approximating $u$; in the second step (Proposition~\ref{prop:recovery}), we prove the sharpness of the lower bound, constructing a recovery sequence made of regular configurations.


\subsection{Lower semicontinuity} \label{subsection:liminf}
The lower semicontinuity of the interface part of the energy \eqref{relaxation2} follows essentially from the same type of arguments as in \cite{AmbBra90}. It is indeed well-known (see also \cite{Whi96}) that, for an isotropic surface energy defined on Caccioppoli partitions of a domain $\Omega$, where each interface has a cost proportional to its area, the validity of the triangle inequalities between the surface tensions is a necessary and sufficient condition for the lower semicontinuity of the functional. However, in our case we do not deal with generic Caccioppoli partitions, but we have a geometric restriction on the admissible configurations; this is reflected in the fact that the surface tension coefficients $\overline{\Psi}(i,j)$ do not satisfy all the possible triangle inequalities, but only those corresponding to actual configurations of the system. For this reason we cannot directly deduce the lower semicontinuity from \cite{AmbBra90}, but the proof is based on the same type of arguments and we will only sketch the main ideas. The main tool is the following lower semicontinuity lemma, whose proof follows easily by adapting the ideas in \cite[Proposition~3.1]{Mor97}.

\begin{lemma} \label{lemma:lsc3}
Let $F^1,F^2\subset B_1$ be disjoint sets of finite perimeter with $F^1\cup F^2 = B_1$, and let $m>2$.
Suppose that, for $i,j\in\{1,\ldots,m\}$, $\lambda_{ij}=\lambda_{ji}$ are nonnegative coefficients such that
\begin{equation} \label{triangle}
\lambda_{12} \leq \lambda_{1i_1}+\lambda_{i_1i_2} + \ldots + \lambda_{i_{j-1}i_j} + \lambda_{i_j2} \qquad\text{for all }i_1,\ldots i_j \in\{3,\ldots,m\}\text{ distinct}.
\end{equation}
For every $k\in\N$ let $(F^1_k,F^2_k,\ldots,F^m_k)$ be a Caccioppoli partition of $B_1$ into $m$ sets, such that $F^1_k\to F^1$, $F^2_k\to F^2$, and $F^i_k\to\emptyset$ in $L^1(B_1)$, $i=3,\ldots,m$, as $k\to\infty$. Then
\begin{equation*}
\lambda_{12}\hn(\partial^*F^1\cap\partial^*F^2) \leq \liminf_{k\to\infty}\sum_{\substack{i,j=1\\i<j}}^m \lambda_{ij}\hn(\partial^*F^i_k\cap\partial^*F^j_k) \,.
\end{equation*}
\end{lemma}

\begin{proposition} \label{prop:lsc}
Denote by $\fc$ the right-hand side of \eqref{relaxation2}. For every $u\in\ac$ and for every sequence $\{u_j\}_{j\in\N}\subset\acreg$ such that $u_j\to u$ in $\ac$ there holds
\begin{equation} \label{lsc}
\fc(u) \leq\liminf_{j\to\infty}\f(u_j).
\end{equation}
\end{proposition}

\begin{proof}
As already observed, it is sufficient to consider the surface part of the energy, as $\nl(u)$ is continuous with respect to the convergence in $\ac$. Without loss of generality we can assume that the sequence $\f(u_j)$ is bounded and that the measures $\mu_j\coloneqq\Psi(u_j^+,u_j^-)\hn\mres J_{u_j}$ locally weakly* converge in $\R^n$ to a positive Radon measure $\mu$. We need to show that $\mu\geq\overline{\Psi}(u^+,u^-)\hn\mres J_u$. By \cite[Theorem~2.56]{AFP} it is sufficient to show that
\begin{equation} \label{prooflsc1}
\limsup_{\rho\to0^+}\frac{\mu(B_\rho(x))}{\omega_{n-1}\rho^{n-1}} \geq \overline{\Psi}(u^+(x),u^-(x)) \qquad\text{for }\hn\text{-a.e. }x\in J_u.
\end{equation}
This can be proved by a standard blow-up argument: for fixed $x\in J_u$, for a suitable sequence $\rho_k\to0^+$ and for a suitable subsequence, we have that the rescaled functions $v_k(y)\coloneqq u_{j_k}(x+\rho_k y)$ converge in $L^1(B_1)$ as $k\to\infty$ to the function
\begin{equation*}
w(y)\coloneqq
\begin{cases}
u^+(x) & \text{in } \{y\in B_1 : y\cdot\nu_u(x)>0 \}, \\
u^-(x) & \text{in } \{y\in B_1 : y\cdot\nu_u(x)<0 \},
\end{cases}
\end{equation*}
and $\limsup_{\rho\to0^+}\frac{\mu(B_\rho(x))}{\omega_{n-1}\rho^{n-1}} \geq \liminf_{k\to\infty} \frac{1}{\omega_{n-1}}\int_{B_1\cap J_{v_k}} \Psi(v_k^+,v_k^-)\de\hn$, so that the claim \eqref{prooflsc1} will follow once we prove that
\begin{equation} \label{prooflsc2}
\liminf_{k\to\infty} \int_{B_1\cap J_{v_k}} \Psi(v_k^+,v_k^-)\de\hn \geq \omega_{n-1}\overline{\Psi}(w^+,w^-) .
\end{equation}
In view of \eqref{interfaces3}, in order to show \eqref{prooflsc1} we now have to distinguish among six possible cases, depending on which interface contains the point $x$.

\smallskip\noindent\textit{Case 1: $x\in\Gamma^A_u$.} In this case, in the blow-up limit we have that half of the ball $B_1$ is filled with the pure phase $A_u$, and the other half ball is filled with the phase $V_u$; that is, up to a permutation $w^+=1$, $w^-=0$. Notice that $x_n>0$ and therefore for $k$ large enough the ball $B_{\rho_k}(x)$ does not intersect the substrate and can contain only the phases $A$, $B$, and $V$; hence, the rescaled functions $v_k$ can only take the values $\{\pm1,0\}$ in $B_1$, that is, $v_k^\pm\in\{\pm1,0\}$. Since by definition of $\overline{\Psi}$ the triangle inequality $\overline{\Psi}(1,0)\leq\overline{\Psi}(1,-1)+\overline{\Psi}(-1,0)$ holds and $\overline{\Psi}\leq\Psi$, the claim \eqref{prooflsc2} follows from Lemma~\ref{lemma:lsc3}, applied to $F^1\coloneqq\{w=1\}$, $F^2\coloneqq\{w=0\}$, $F^1_k\coloneqq\{v_k=1\}\to F^1$, $F^2_k\coloneqq\{v_k=0\}\to F^2$, $F^3_k\coloneqq\{v_k=-1\}\to\emptyset$.

\smallskip\noindent\textit{Case 2: $x\in\Gamma^B_u$.} This is analogous to the previous case.

\smallskip\noindent\textit{Case 3: $x\in\Gamma^{AB}_u$.} In this case $w^+=1$, $w^-=-1$, and since $B_{\rho_k}(x)$ does not intersect the substrate for $k$ large enough, we have $v_k^\pm\in\{\pm1,0\}$. Then \eqref{prooflsc2} follows again by Lemma~\ref{lemma:lsc3} in view of the triangle inequality $\overline{\Psi}(1,-1)\leq\overline{\Psi}(1,0)+\overline{\Psi}(-1,0)$, which holds by definition of $\overline{\Psi}$ and by the assumption $\sigma_{AB}\leq\sigma_A+\sigma_B$.

\smallskip\noindent\textit{Case 4: $x\in S^A_u$.} In this case $w^+=1$, $w^-=2$. In principle, all the four phases can be present in a neighbourhood of the point $x$; however, by the geometric constraint the limit interface between the phase $A$ and the substrate $S$ cannot be approximated by the boundary of $V$. Therefore, in order to apply Lemma~\ref{lemma:lsc3}, we first need to get rid of the possible infiltration of the phase $V$.

We denote by $A_k\coloneqq\{v_k=1\}$, $B_k\coloneqq\{v_k=-1\}$, $V_k\coloneqq\{v_k=0\}$ the phases of $v_k$ in the upper half ball $B_1^+$, and the corresponding interfaces by
\begin{equation*}
\Gamma_k^A \coloneqq \partial^*A_k\cap\partial^*V_k\cap B_1\,, \quad \Gamma_k^B \coloneqq \partial^*B_k\cap\partial^*V_k\cap B_1\,, \quad \Gamma^{AB}_k \coloneqq\partial^*A_k\cap\partial^*B_k\cap B_1,
\end{equation*}
\begin{equation*}
S^A_k \coloneqq \partial^*A_k\cap\partial S\cap B_1, \quad S^B_k \coloneqq \partial^*B_k\cap S\cap B_1, \quad S^V_k\coloneqq \partial^*V_k\cap S\cap B_1.
\end{equation*}
Then we modify $v_k$ by ``filling'' the region $V_k$ with either $A_k$ or $B_k$, according to the following rule:
\begin{equation*}
\tilde{v}_k(y)\coloneqq
\begin{cases}
v_k(y) & \text{if } y\in B_1\setminus V_k, \\
1 & \text{if } y\in V_k \text{ and } \hn(\Gamma^B_k)\leq\hn(\Gamma^A_k), \\
-1 & \text{if } y\in V_k \text{ and } \hn(\Gamma^A_k)<\hn(\Gamma^B_k).
\end{cases}
\end{equation*}
Notice that $\tilde{v}_k\to w$ in $L^1(B_1)$, and that the partition of the unit ball determined by $\tilde{v}_k$ does not contain the phase $V$. Therefore, using the inequality $\Psi(1,0)+\Psi(-1,0)\geq\Psi(-1,1)$,
\begin{align*}
\int_{B_1\cap J_{v_k}} \Psi(v_k^+,v_k^-)\de\hn
& = \Psi(1,0)\hn(\Gamma^A_k) + \Psi(-1,0)\hn(\Gamma^B_k) + \Psi(-1,1)\hn(\Gamma^{AB}_k) \\
& \qquad + \Psi(1,2)\hn(S^A_k) + \Psi(-1,2)\hn(S^B_k) + \Psi(0,2)\hn(S^{V}_k) \\
& \geq \Psi(-1,1)\min\bigl\{\hn(\Gamma^A_k),\hn(\Gamma^B_k)\bigr\} + \Psi(-1,1)\hn(\Gamma^{AB}_k) \\
& \qquad + \Psi(1,2)\hn(S^A_k) + \Psi(-1,2)\hn(S^B_k) + \Psi(0,2)\hn(S^{V}_k) \\
& = \int_{B_1\cap J_{\tilde{v}_k}} \Psi(\tilde{v}_k^+,\tilde{v}_k^-)\de\hn + \Psi(0,2)\hn(S^V_k) \\
& \qquad -\max\{\Psi(1,2),\Psi(-1,2)\} \hn(S^V_k).
\end{align*}
By observing that $\hn(S^V_k)\to0$ as $k\to\infty$, from the previous inequality we obtain
\begin{equation*}
\liminf_{k\to\infty} \int_{B_1\cap J_{v_k}} \Psi(v_k^+,v_k^-)\de\hn
\geq \liminf_{k\to\infty} \int_{B_1\cap J_{\tilde{v}_k}} \Psi(\tilde{v}_k^+,\tilde{v}_k^-)\de\hn.
\end{equation*}
To deduce \eqref{prooflsc2} we can now apply Lemma~\ref{lemma:lsc3} to the partition of $B_1$ determined by $\tilde{v}_k$, which contains only the three phases $A$, $B$, $S$ and that converges to the configuration where the upper half-ball is filled by $A$, and the lower half-ball is filled by $S$. Therefore to apply Lemma~\ref{lemma:lsc3} one only needs to check the triangle inequality $\overline{\Psi}(1,2)\leq\overline{\Psi}(1,-1)+\overline{\Psi}(-1,2)$, which holds by definition of $\overline{\Psi}$.

\smallskip\noindent\textit{Case 5: $x\in S^B_u$.} This is analogous to Case~4, with the roles of phases $A$ and $B$ exchanged.

\smallskip\noindent\textit{Case 6: $x\in S^V_u$.} In this case $w^+=0$, $w^-=2$, and all the four phases can be present in a neighbourhood of the point $x$. We deduce \eqref{prooflsc2} by applying once more Lemma~\ref{lemma:lsc3}, since that all the possible triangle inequalities \eqref{triangle} hold for $\lambda_{12}=\overline{\Psi}(0,2)$ in view of the definition of $\overline{\Psi}$.
\end{proof}


\subsection{Recovery sequence}\label{subsection:limsup}

The goal of this section is to prove the following result, which combined with Proposition~\ref{prop:lsc} completes the proof of Theorem~\ref{thm:relaxation}.

\begin{proposition} \label{prop:recovery}
Denote by $\fc$ the right-hand side of \eqref{relaxation2}. For every $u\in\ac$ there exists a sequence $\{u_j\}_{j\in\N}\subset\acreg$ such that $u_j\to u$ in $\ac$, $|A_{u_j}|=|A_u|$, $|B_{u_j}|=|B_u|$, and
\begin{equation} \label{recovery}
\fc(u) =\lim_{j\to\infty}\f(u_j).
\end{equation}
\end{proposition}

\begin{proof}
Fix $u\in\ac$ and let $h_u\in\ap(Q_L)$ be the corresponding admissible profile. The proof is divided into several steps (see Figure~\ref{fig:modifications} for the modifications performed in Step~2, 3, and 4).

\medskip\noindent\textit{Step 1: approximation of $h_u$ with a regular profile.}
In this step we construct a sequence $\widetilde{u}_j\in\ac$ such that $\widetilde{u}_j\to u$ in $\ac$ and $\fc(\widetilde{u}_j)\to\fc(u)$, with the additional property that the corresponding profiles $h_{\widetilde{u}_j}$ are smooth. By a diagonal argument this will allow us, in the following steps, to work under the assumption that the limiting profile is smooth, and to construct a recovery sequence only in this case.

For each $j\in\N$ it is possible to find a  $Q_L$-periodic function $f_j \in C^\infty(\R^{n-1})$ such that
\begin{equation}\label{eq:approx1}
\|f_j-h_u\|_{L^1(Q_L)}+\hn(\Gamma_{h_u} \cap \Omega_{f_j})\leq\frac{1}{j}\,,
\end{equation}
\begin{equation}\label{eq:approx2}
\left| \int_{Q_L} \sqrt{1+|\nabla f_j(x')|^2} \de x' - \hn(\Gamma_{h_u})  \right|\leq\frac{1}{j}\,,
\end{equation}
and
\begin{equation}\label{eq:zerosetofhj}
\left|\, \hn(\{f_j=0\}) - \hn(\{h_u=0\}) \,\right|\leq\frac{1}{j}\,.
\end{equation}
The proof of the first two statements is contained in Step~1 of the proof of \cite[Proposition~4.1]{ChaSol07}, while the last statement is proved in \cite[Remark~4.4]{ChaSol07}.
Define the function $\widetilde{u}_j: \R^n
\to\{0,-1,1,2\}$ as
\begin{equation}\label{eq:wideuj}
\arraycolsep=1.8pt\def\arraystretch{1.4}
\widetilde{u}_j(x)\coloneqq\left\{
\begin{array}{lll}
u(x) & & \text{ if } x\in \Omega_{f_j}^\#\cap\Omega_{h_u}^\#,\\
1 & & \text{ if } x\in \Omega_{f_j}^\#\setminus\Omega_{h_u}^\#,\\
0 & & \text{ if } x\in\R^n_+\setminus\Omega_{f_j}^\#,\\
2 & & \text{ if } x\in S.
\end{array}
\right.
\end{equation}
This  modification amounts to fill the (small) region in $\Omega_{f_j}\setminus\Omega_{h_u}$ by the phase $A$, and to remove the possible parts of the phases $A$ and $B$ outside $\Omega_{f_j}$ by replacing them with $V$. Notice that $\widetilde{u}_j\in\acreg$ with $f_j=h_{\widetilde{u}_j}$ and
\begin{equation}\label{eq:conv_uj_1}
\| \widetilde{u}_j - u \|_{L^1(Q_L\times\R)} \leq \frac{2}{j}\,.
\end{equation}

First, we show that
\begin{equation}\label{eq:convhninterna}
\hn(\Gamma^{AB}_{\widetilde{u}_j})\to \hn(\Gamma^{AB}_u).
\end{equation}
Define the Radon measures $\mu^A\coloneqq D\chi_{A_u^\#}$, $\mu^B\coloneqq D\chi_{B_u^\#}$, and, for $j\in\N$, define $\mu^A_j\coloneqq D\chi_{A_{\widetilde{u}_j}^\#}$ and $\mu^B_j\coloneqq D\chi_{B_{\widetilde{u}_j}^\#}$. Then \eqref{eq:approx2} and \eqref{eq:zerosetofhj} yield
\begin{equation} \label{proof:recovery_step1_1}
\lim_{j\to\infty}|\mu^A_j + \mu^B_j|(Q_L^+) = \lim_{j\to\infty}|D\chi_{\Omega_{f_j}^\#}|(Q_L^+) = |D\chi_{\Omega_{h_u}^\#}|(Q_L^+) = |\mu^A + \mu^B|(Q_L^+),
\end{equation}
and, since $A_{\widetilde{u}_j}\to A_u$, $B_{\widetilde{u}_j}\to B_u$, using also the periodicity,
\begin{equation} \label{proof:recovery_step1_2}
|\mu^A|(Q_L^+)\leq \liminf_{j\to\infty}|\mu_j^A|(Q_L^+),\quad\quad\quad
|\mu^B|(Q_L^+)\leq \liminf_{j\to\infty}|\mu_j^B|(Q_L^+).
\end{equation}
Combining the previous estimates we obtain \eqref{eq:convhninterna} from
\begin{align*}
\hn(\Gamma^{AB}_u)
& = \frac{|\mu^A|(Q_L^+) + |\mu^B|(Q_L^+) - |\mu^A+\mu^B|(Q_L^+)}{2} \\
&\leq \liminf_{j\to\infty}\frac{|\mu_j^A|(Q_L^+) + |\mu_j^B|(Q_L^+) - |\mu_j^A+\mu_j^B|(Q_L^+)}{2} 
=\liminf_{j\to\infty}\hn(\Gamma^{AB}_{\widetilde{u}_j})\\
&\leq \liminf_{j\to\infty} \biggl(\hn(\Gamma^{AB}_{u}) + \hn(\Gamma_{h_u} \cap \Omega_{f_j}) \biggr) \leq\hn(\Gamma^{AB}_{u}),
\end{align*}
where last step follows by \eqref{eq:approx1}.
Next, we claim that
\begin{equation}\label{eq:convGammaA}
\hn(\Gamma^A_{\widetilde{u}_j})\to \hn(\Gamma^A_u),\quad\quad
\hn(\Gamma^B_{\widetilde{u}_j})\to \hn(\Gamma^B_u),
\end{equation}
and that
\begin{equation}\label{eq:convSubs}
\hn(S^A_{\widetilde{u}_j})\to \hn(S^A_u),\quad\quad
\hn(S^B_{\widetilde{u}_j})\to \hn(S^B_u).
\end{equation}
Using \eqref{eq:convhninterna}, \eqref{proof:recovery_step1_1}, and \eqref{proof:recovery_step1_2}, we have
\begin{align*}
|\mu^A|(Q_L^+) + |\mu^B|(Q_L^+)
& \leq\liminf_{j\to\infty} \, \bigl[ \, |\mu^A_j|(Q_L^+) + |\mu^B_j|(Q_L^+) \, \bigr] \\
&= \liminf_{j\to\infty} \, \bigl[ \, |\mu^A_j+\mu^B_j|(Q_L^+) + 2\hn(\Gamma^{AB}_{\widetilde{u}_j}) \, \bigr] \\
& = |\mu^A+\mu^B|(Q_L^+) + 2\hn(\Gamma^{AB}_u)
 = |\mu^A|(Q_L^+) + |\mu^B|(Q_L^+),
\end{align*}
hence
\begin{equation}\label{eq:convper}
|\mu^A_j|(Q_L^+) \to |\mu^A|(Q_L^+),
\quad\quad
|\mu^B_j|(Q_L^+)\to|\mu^B|(Q_L^+).
\end{equation}
Denote now, for $\e>0$, $Q^\e \coloneqq Q_L\times(\e,+\infty)$, and notice that for $\leb^1$-almost every $\e>0$ we have $\hn(J_u\cap\{x_n=\e\})=0$. For all such $\e$, thanks to \eqref{proof:recovery_step1_1} and to \eqref{eq:convper} we obtain
\begin{equation*}
|\mu^A_j|(Q^\e) \to |\mu^A|(Q^\e),
\qquad
|\mu^B_j|(Q^\e) \to |\mu^B|(Q^\e),
\qquad
|\mu^A_j+\mu^B_j|(Q^\e)\to |\mu^A+\mu^B|(Q^\e),
\end{equation*}
and in turn, arguing as in the proof of \eqref{eq:convhninterna},
$\hn(\Gamma^{AB}_{\widetilde{u}_j}\cap Q^\e) \to \hn(\Gamma^{AB}_u\cap Q^\e).$
Then for almost every $\e>0$
\begin{align*}
\hn(\Gamma^A_{\widetilde{u}_j}\cap Q^\e)
& = |\mu^A_j|(Q^\e) - \hn(\Gamma^{AB}_{\widetilde{u}_j}\cap Q^\e) \\
& \to |\mu^A|(Q^\e) - \hn(\Gamma^{AB}_u\cap Q^\e)
= \hn(\Gamma^A_u\cap Q^\e),
\end{align*}
and similarly $\hn(\Gamma^B_{\widetilde{u}_j}\cap Q^\e)\to\hn(\Gamma^B_u\cap Q^\e)$. From these two convergences \eqref{eq:convGammaA} follows: indeed, if \eqref{eq:convGammaA} fails then for some $\eta>0$ we would have (using the fact that $\hn(\Gamma^A_{\widetilde{u}_j}\cup\Gamma^B_{\widetilde{u}_j})\to\hn(\Gamma^A_u\cup\Gamma^B_u)$)
$$
\limsup_{j\to\infty}\hn(\Gamma^A_{\widetilde{u}_j})\geq\hn(\Gamma^A_u)+\eta, \qquad \liminf_{j\to\infty}\hn(\Gamma^B_{\widetilde{u}_j})\leq \hn(\Gamma^B_u)-\eta
$$
(or the symmetric inequalities with $A$ and $B$ exchanged). This yields
$$
\liminf_{j\to\infty}\hn(\Gamma^B_{\widetilde{u}_j}\setminus Q^\e)\leq \hn(\Gamma^B_u\setminus Q^\e)-\eta \qquad\text{for every $\e>0$,}
$$
which is a contradiction since $\hn(\Gamma^B_u\setminus Q^\e)\to0$ as $\e\to0$. This proves \eqref{eq:convGammaA}.
Finally, by writing
\begin{align*}
|\mu^A_j|(Q_L^+) & = \hn(\Gamma^A_{\widetilde{u}_j})
+\hn(\Gamma^{AB}_{\widetilde{u}_j}) + \hn(S^A_{\widetilde{u}_j}), \\
|\mu^A|(Q_L^+) & = \hn(\Gamma^A_u)
+\hn(\Gamma^{AB}_u) + \hn(S^A_u)
\end{align*}
(and similarly for $B$), we conclude that also \eqref{eq:convSubs} holds by using \eqref{eq:convhninterna}, \eqref{eq:convGammaA}, and \eqref{eq:convper}.

Thanks to \eqref{eq:conv_uj_1}, \eqref{eq:convhninterna}, \eqref{eq:convGammaA}, and \eqref{eq:convSubs} we obtain $\fc(\widetilde{u}_j)\to\fc(u)$, as desired.

\begin{figure}
\begin{center}
\includegraphics[scale=1]{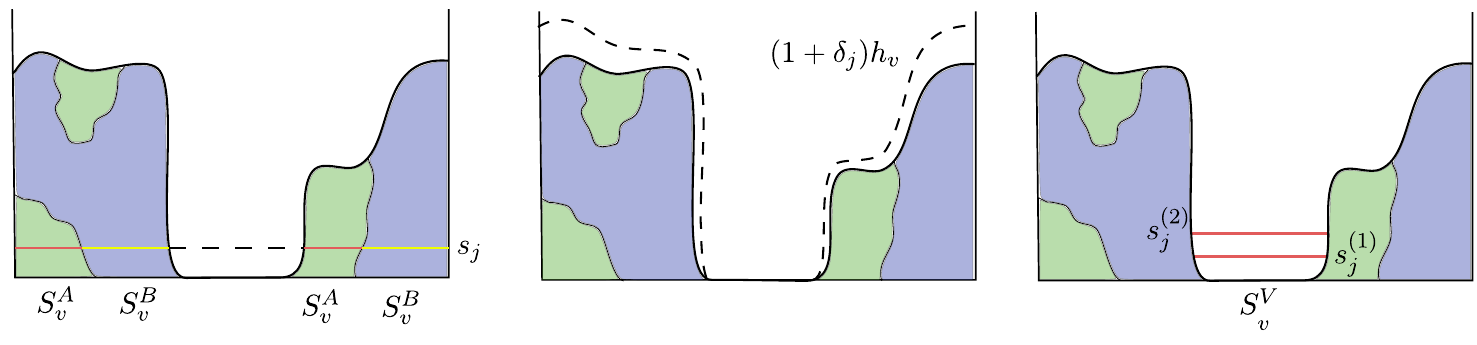}
\caption{The modifications we perform in Step~2 (left), 3 (center), and 4 (right).}
\label{fig:modifications}
\end{center}
\end{figure}

\medskip\noindent\textit{Step 2: the non-exposed substrate.}
Assume $v\in\acreg$. We construct a sequence $\{v_j\}_{j\in\N}\subset\acreg$ such that
\begin{equation}\label{eq:conv_uj_2}
\lim_{j\to\infty}\| v_j - v \|_{L^1(Q_L^+)}=0 ,
\end{equation}
that allows to recover the relaxed coefficients $\overline{\Psi}(1,2)$ and $\overline{\Psi}(-1,2)$ with the non-exposed substrate in the limit energy, in the sense that
\begin{equation} \label{eq:conv_energy_2b}
\f(v_j)\to \f(v) + \bigl( \overline{\Psi}(1,2)-\Psi(1,2) \bigr) \hn(S_v^A) + \bigl( \overline{\Psi}(-1,2)-\Psi(-1,2) \bigr) \hn(S_v^B).
\end{equation}
In the case where $\sigma_{AS}\leq\sigma_{BS}+\sigma_{AB}$ and $\sigma_{BS}\leq\sigma_{AS}+\sigma_{AB}$, the relaxed surface tensions $\overline{\Psi}(1,2)$ and $\overline{\Psi}(-1,2)$ coincide with the original ones $\Psi(1,2)$ and $\Psi(-1,2)$; in this case there is nothing to do, and we just take $v_j\coloneqq v$ for each $j\in\N$. Assume instead
\[
\sigma_{AS}\leq\sigma_{BS}+\sigma_{AB}\quad\quad\text{and}\quad\quad\sigma_{AS}+\sigma_{AB}<\sigma_{BS}\,.
\]
The only other possible case is $\sigma_{BS}+\sigma_{AB}<\sigma_{AS}$ and $\sigma_{BS}\leq\sigma_{AS}+\sigma_{AB}$, that can be treated similarly. We need to build a sequence $\{v_j\}_{j\in\N}$ satisfying \eqref{eq:conv_uj_2} and \eqref{eq:conv_energy_2b}, which in this case becomes
\begin{equation}\label{eq:conv_energy_2}
\f(v_j)\to \f(v) + (\sigma_{AS} + \sigma_{AB} - \sigma_{BS}) \hn(S_v^B).
\end{equation}

By standard results on traces of $\BV$-functions (for instance, combining equation (2.8) and Theorem~2.11 in \cite{G}), it is possible to find a sequence $\{s_j\}_{j\in\N}$ with $s_j\to0^+$ as $j\to\infty$ such that
\begin{equation}\label{eq:conv_AS-BS_1}
\hn(A_v^{(1)}\cap\{x_n=s_j\})\to \hn(S^A_v), \quad \hn(B_v^{(1)}\cap\{x_n=s_j\})\to \hn(S^B_v),
\end{equation}
and since $\hn(J_v\cap\{x_n=t\})=0$ for $\leb^1$-a.e.\ $t$ we can also assume
\begin{equation}\label{eq:convstep2-2}
\hn(J_v\cap\{x_n=s_j\}) = 0.
\end{equation}
Also note that, since $\hn(J_v)<\infty$,
\begin{equation} \label{eq:convstep2-3}
\hn(J_v\cap\{0<x_n<s_j\})\to 0.
\end{equation}
Define the function $v_j:\R^n\to\{0,-1,1,2\}$ as
\begin{equation}\label{eq:vj}
v_j(x',x_n)\coloneqq\left\{
\begin{array}{ll}
1 & \text{if } (x',x_n)\in\Omega_{h_v}^\# \text{ and } 0<x_n<s_j,\\
v(x',x_n) & \text{otherwise,}
\end{array}
\right.
\end{equation}
which satisfies $v_j\in\acreg$ for each $j\in\N$ (since $h_{v_j}=h_v$) and $\| v_j - v \|_{L^1(Q_L^+)} \leq 2 s_j \leb^{n-1}(Q_L)$,
which gives \eqref{eq:conv_uj_2}. This sequence allows to adjust the surface tensions for the substrate: namely, we have by \eqref{eq:convstep2-2}
\begin{align*}
\f(v_j) - \f(v)
& = (\sigma_A-\sigma_B)\hn(\Gamma^B_v\cap\{0<x_n<s_j\}) - \sigma_{AB}\hn(\Gamma^{AB}_v\cap\{0<x_n<s_j\}) \\
& \qquad + \sigma_{AB}\hn(B_v^{(1)}\cap\{x_n=s_j\}) + (\sigma_{AS}-\sigma_{BS})\hn(S^B_v) \\
& \qquad + \gamma \, (\nl(v_j) - \nl(v)).
\end{align*}
By passing to the limit as $j\to\infty$, the first two terms on the right-hand side vanish thanks to \eqref{eq:convstep2-3}, the third term tends to $\sigma_{AB}\hn(S^B_v)$ by \eqref{eq:conv_AS-BS_1}, and the last term tends to zero by \eqref{eq:conv_uj_2}. Hence \eqref{eq:conv_energy_2} follows.

\medskip\noindent\textit{Step 3: the graph.}
Let $v\in\acreg$ and $\{v_j\}_{j\in\N}\subset\acreg$ be the sequence constructed in the previous step, satisfying \eqref{eq:conv_uj_2} and \eqref{eq:conv_energy_2b}. We want to modify the sequence in such a way to recover the relaxed surface tensions $\overline{\Psi}(1,0)$ and $\overline{\Psi}(-1,0)$ between the two phases $A$, $B$ and the void $V$: more precisely, we want to construct another sequence $\{w_j\}_{j\in\N}\subset\acreg$ such that
\begin{equation}\label{eq:conv_uj_3}
\lim_{j\to\infty}\| w_j - v_j \|_{L^1(Q_L^+)}=0,
\end{equation}
and
\begin{multline}\label{eq:conv_energy_3b}
\lim_{j\to\infty} \big| \f(w_j) - \f(v_j) - \bigl( \overline{\Psi}(1,0)-\Psi(1,0)\bigr) \hn(\Gamma_v^A) \\
- \bigl( \overline{\Psi}(-1,0)-\Psi(-1,0)\bigr) \hn(\Gamma_v^B) \big| =0.
\end{multline}

In the case where $\sigma_{A}\leq\sigma_{B}+\sigma_{AB}$ and $\sigma_{B}\leq\sigma_{A}+\sigma_{AB}$, the relaxed surface tensions $\overline{\Psi}(1,0)$ and $\overline{\Psi}(-1,0)$ coincide with the original ones $\Psi(1,0)$ and $\Psi(-1,0)$; in this case there is nothing to do, and we just take $w_j\coloneqq v_j$ for each $j\in\N$. Assume instead
\[
\sigma_A\leq\sigma_B+\sigma_{AB}\,,\quad\quad \sigma_A+\sigma_{AB}<\sigma_B\,.
\]
The only other possible case is $\sigma_B+\sigma_{AB}<\sigma_A$ and $\sigma_B\leq\sigma_A+\sigma_{AB}$, that can be treated similarly. In this case the condition \eqref{eq:conv_energy_3b} becomes
\begin{equation}\label{eq:conv_energy_3}
\lim_{j\to\infty} \left| \f(w_j) - \f(v_j) - (\sigma_A+\sigma_{AB} - \sigma_B) \hn(\Gamma_v^B) \right| =0.
\end{equation}

Let $\delta_j\to0^+$ and define, for each $j\in\N$, the function $w_j:\R^n\to\{0,-1,1,2\}$ by
\begin{equation}\label{eq:wj}
w_j(x',x_n)\coloneqq
\begin{cases}
v_j(x',x_n) & \text{ if } (x',x_n)\in\Omega^\#_{h_{v_j}},\\[1ex]
1 & \text{ if } h_{v_j}(x')<x_n<(1+\delta_j)h_{v_j}(x'), \\[1ex]
0 & \text{ if } x_n\geq(1+\delta_j)h_{v_j}(x').
\end{cases}
\end{equation}
Note that $h_{w_j}=(1+\delta_j)h_{v_j}=(1+\delta_j)h_v$ (recalling that $h_{v_j}=h_v$ for all $j$, by the construction in Step~2), therefore $w_j\in\acreg$ and
\[
\| w_j - v_j \|_{L^1(Q_L^+)} \leq \int_{Q_L}\int_{h_{v}(x')}^{(1+\delta_j)h_{v}(x')} |1-v_j(x',x_n)|\de x_n\de x'
\leq 2\delta_j \int_{Q_L}h_{v}(x')\de x'
\]
which yields \eqref{eq:conv_uj_3}. Moreover by a Taylor expansion
\begin{align*}
\int_{Q_L\setminus S^V_{v}}\sqrt{1 + |(1+\delta_j)\nabla h_{v}(x')|^2}\de x'
&= \hn(\Gamma^A_{v_j}) + \hn(\Gamma^B_{v_j}) + o(1),
\end{align*}
therefore
\begin{align*}
\f(w_j) - \f(v_j)
& = \sigma_A \bigl( \hn(\Gamma^A_{v_j}) + \hn(\Gamma^B_{v_j}) \bigr) + o(1) \\
& \qquad + (\sigma_{AB}-\sigma_B)\hn(\Gamma^{B}_{v_j}) - \sigma_A\hn(\Gamma^A_{v_j}) + \gamma \, ( \nl(w_j) - \nl(v_j) ) .
\end{align*}
We get \eqref{eq:conv_energy_3} by using \eqref{eq:conv_uj_3} and recalling that, by the construction in Step~2,  we have $\hn(\Gamma^A_{v_j})\to\hn(\Gamma^A_{v})$ and $\hn(\Gamma^B_{v_j})\to\hn(\Gamma^B_{v})$.

\medskip\noindent\textit{Step 4: the exposed substrate.}
Let $v\in\acreg$ and $\{w_j\}_{j\in\N}\subset\acreg$ be the sequence constructed in the previous step. We want to modify again the sequence in such a way to recover the relaxed surface tension $\overline{\Psi}(0,2)$ of the exposed substrate, that is the interface between the substrate $S$ and the void $V$: more precisely, we want to construct another sequence $\{z_j\}_{j\in\N}\subset\acreg$ such that
\begin{equation}\label{eq:conv_uj_4}
\lim_{j\to\infty}\| z_j - w_j \|_{L^1(Q_L^+)}=0,
\end{equation}
and
\begin{equation}\label{eq:conv_energy_4b}
\lim_{j\to\infty} \big| \f(z_j) - \f(w_j) - \bigl( \overline{\Psi}(0,2)-\Psi(0,2) \bigr) \hn(S_v^V) \big| =0.
\end{equation}

In the case where
$\sigma_S\leq\min\{ \sigma_{AS}+\bar{\sigma}_A,
\sigma_{BS}+\bar{\sigma}_B \}$
there is nothing to do since $\overline{\Psi}(0,2)=\Psi(0,2)$, and thus we define $z_j\coloneqq w_j$ for all $j\in\N$.
Assume that
\[
\sigma_{AS}+\bar{\sigma}_A\leq\min\{ \sigma_S, \sigma_{BS}+\bar{\sigma}_B \}
\qquad\text{and}\qquad
\bar{\sigma}_A=\sigma_{AB}+\sigma_B.
\]
In this case \eqref{eq:conv_energy_4b} becomes
\begin{equation}\label{eq:conv_energy_4}
\lim_{j\to\infty} \left| \f(z_j) - \f(w_j)
	- (\sigma_{AS} + \sigma_{AB} + \sigma_B - \sigma_S) \hn(S_v^V) \right| =0.
\end{equation}
Note that the other possible cases can be treated similarly (and even more easily).

We fix two sequences $s^{(1)}_j, s^{(2)}_j\in(0,1)$, for $j\in\N$, with $s^{(1)}_j < s^{(2)}_j$ and $s^{(1)}_j, s^{(2)}_j\to0$ as $j\to\infty$, such that, by setting
$L_s\coloneqq V_{w_j}\cap \{x_n=s\},$
we have
\begin{equation}\label{eq:convstep4-2}
\hn(L_{s^{(1)}_j}) \to \hn(S^V_v),\quad\quad\quad
\hn(L_{s^{(2)}_j}) \to \hn(S^V_v),
\end{equation}
and
\begin{equation}\label{eq:convstep4-3}
\hn(\Gamma_{w_j}\cap\{0<x_n<s_j^{(2)}\}) \to 0.
\end{equation}
The existence of such sequences can be proved similarly to \eqref{eq:conv_AS-BS_1}, using also the convergence $\hn(S^V_{w_j})\to\hn(S^V_v)$ in view of the construction of $w_j$ in the previous step.
We define the function $z_j:Q_L\times\R\to\{0,-1,1,2\}$ (extended by periodicity to $\R^n$) by
\begin{equation}\label{eq:widevj}
z_j(x',x_n)\coloneqq
\begin{cases}
w_j(x',x_n) & \text{ if } (x',x_n)\in\Omega_{h_{w_j}} \cup S,\\[1ex]
1  & \text{ if } (x',x_n)\in V_{w_j} \text{ and }0<x_n<s_j^{(1)},\\[1ex]
-1 & \text{ if } (x',x_n)\in V_{w_j} \text{ and }s_j^{(1)}<x_n<s_j^{(2)},\\[1ex]
0 & \text{ else.}
\end{cases}
\end{equation}
Since $h_{z_j}=\max\{h_{w_j},s_j^{(2)}\}$ we have $z_j\in\acreg$, and also $\| w_j - z_j \|_{L^1(Q_L^+)} \leq s_j^{(2)} \leb^{n-1}(Q_L)$, which yields \eqref{eq:conv_uj_4}. Moreover
\begin{equation} \label{eq:convstep4-4}
\begin{split}
\f(z_j) - \f(w_j) 
& = \sigma_B \hn(L_{s_j^{(2)}}) + \sigma_{AB} \hn(L_{s_j^{(1)}}) + (\sigma_{AS}-\sigma_S)\hn(S^V_{w_j}) \\
&\qquad + \gamma \, (\nl(z_j) - \nl(w_j)) + R_j,
\end{split}
\end{equation}
where
\begin{align*}
R_j
&\coloneqq - \hn(\Gamma^A_{w_j}\cap\{0<x_n<s_j^{(1)}\}) + (\sigma_{AB}-\sigma_B)\hn(\Gamma^B_{w_j}\cap\{0<x_n<s_j^{(1)}\}) \\
&\qquad + (\sigma_{AB}-\sigma_A)\hn(\Gamma^A_{w_j}\cap\{s_j^{(1)}<x_n<s_j^{(2)}\}) - \hn(\Gamma^B_{w_j}\cap\{s_j^{(1)}<x_n<s_j^{(2)}\}).
\end{align*}
Notice that $R_j\to0$ thanks to \eqref{eq:convstep4-3}.
We then obtain \eqref{eq:conv_energy_4} by passing to the limit in \eqref{eq:convstep4-4}, using \eqref{eq:conv_uj_4}, \eqref{eq:convstep4-2}, and the fact that $\hn(S^V_{w_j})\to\hn(S^V_v)$.

\medskip\noindent\textit{Step 5: the mass constraint.}
By combining the constructions in the previous steps and using a diagonal argument, we have that given $u\in\ac$, there exists a sequence $\{z_j\}_{j\in\N}\subset\acreg$ such that
\begin{equation}\label{eq:good_vj_1}
\lim_{j\to\infty}\| z_j - u \|_{L^1(Q_L^+)}=0,
\qquad
\lim_{j\to\infty} \f(z_j) = \fc(u)
\end{equation}
(see in particular \eqref{eq:conv_uj_2}, \eqref{eq:conv_uj_3}, \eqref{eq:conv_uj_4} for the convergence of the functions, and \eqref{eq:conv_energy_2b}, \eqref{eq:conv_energy_3b}, \eqref{eq:conv_energy_4b} for the convergence of the energies). In order to obtain the recovery sequence, we need to restore the mass constraint: denoting by $|\Omega_{h_u}|=M$, $|A_u|=m$, we modify the sequence $\{z_j\}_{z\in\N}$ and we construct a new sequence $\{u_j\}_{j\in\N}\subset\acreg$ such that
\begin{equation}\label{proof:recovery_step5_1}
\lim_{j\to\infty}\| u_j - z_j \|_{L^1(Q_L^+)}=0,
\qquad
\lim_{j\to\infty} | \f(u_j) - \f(z_j)  | =0,
\end{equation}
and
\begin{equation}\label{proof:recovery_step5_2}
|A_{u_j}|=m, \quad |B_{u_j}|=M-m.
\end{equation}

We first adjust the volume of $\Omega_{h_{z_j}}$ by a vertical rescaling: namely, we take $\lambda_j\coloneqq\frac{M}{|\Omega_{h_{z_j}}|}$ (notice that $\lambda_j\to1$ as $j\to\infty$) and we let $h_j\coloneqq\lambda_j h_{z_j}$, so that $|\Omega_{h_j}|=M$. We now need to adjust the volume of $A_{z_j}$ and $B_{z_j}$. Let
$$
\widetilde{A}_j\coloneqq\bigl\{(x',\lambda_j x_n) \,:\, (x',x_n)\in A_{z_j}\bigr\},
\qquad
\widetilde{B}_j\coloneqq\bigl\{(x',\lambda_j x_n) \,:\, (x',x_n)\in B_{z_j}\bigr\}
$$
be the sets obtained by rescaling vertically $A_{z_j}$ and $B_{z_j}$ by the factor $\lambda_j$; notice that $\widetilde{A}_j\cup\widetilde{B}_j=\Omega_{h_j}$ and therefore $|\widetilde{A}_j|+|\widetilde{B}_j|=M$. We also remark that, as $\lambda_j\to1$ and $A_{z_j}\to A_u$, $B_{z_j}\to B_u$ in $L^1$, we have
$|\widetilde{A}_j|\to m$, $|\widetilde{B}_j|\to M-m$ as $j\to\infty$.

Suppose to fix the ideas that $|\widetilde{A}_{j}|<m$ (we proceed similarly in the other case). Let $\bar{x}\in\Omega_{h_u}$ be a point of density one for $B_u$. Since $\widetilde{B}_{j}\to B_u$ in $L^1$, we have
$$
\lim_{r\to0^+}\lim_{j\to\infty}\frac{|\widetilde{B}_{j}\cap B_r(\bar{x})|}{|B_r|} = 1.
$$
Hence, it is possible to find $r_0>0$ and $j_0\in\N$ such that
$$
\frac34 \leq \frac{|\widetilde{B}_{j}\cap B_{r_0}(\bar{x})|}{|B_{r_0}|} \leq 1 \qquad\text{for all }j\geq j_0.
$$
Therefore, for every $j\geq j_0$ (for a possibly larger $j_0$) it is possible to find $r_j\in(0,r_0)$ such that $|\widetilde{B}_{j}\cap B_{r_j}(\bar{x})|=m-|\widetilde{A}_{j}|>0$, since this quantity tends to zero as $j\to\infty$. We eventually define
$$
u_j(x',x_n)\coloneqq
\begin{cases}
1 & \text{if }(x',x_n)\in B_{r_j}(\bar{x})\cap\widetilde{B}_j, \\[1ex]
z_j(x', x_n/\lambda_j) & \text{if } (x',x_n)\in \Omega_{h_j}\setminus (B_{r_j}(\bar{x})\cap\widetilde{B}_j), \\[1ex]
0 & \text{if } (x',x_n)\in Q_L^+\setminus\Omega_{h_j}.
\end{cases}
$$
We then have $h_{u_j}=h_j=\lambda_j h_{z_j}$, so that $u_j\in\acreg$ and $|\Omega_{h_{u_j}}|=M$. Moreover, $A_{u_j}=\widetilde{A}_{j}\cup (B_{r_j}(\bar{x})\cap\widetilde{B}_j)$, hence $|A_{u_j}|=|\widetilde{A}_j|+|\widetilde{B}_j\cap B_{r_j}(\bar{x})|=m$. Thus \eqref{proof:recovery_step5_2} are satisfied. Finally, also the convergences \eqref{proof:recovery_step5_1} hold, since $\lambda_j\to1$ and $r_j\to0$.
\end{proof}


\section{Regularity of minimizers} \label{sect:regularity}

In this section we will study the regularity of solutions to the minimum problem
\begin{equation}\label{eq:min_pb}
\min \bigl\{ \fbar(u)  \,:\,  u\in\ac,\, |A_u|=m,\, |B_u|=M-m \bigr\},
\end{equation}
whose existence has been established in Theorem~\ref{thm:existence}.

The strategy to prove the regularity of minimizers relies, as it is common in these kind of problems, on the regularity theory for area quasi-minimizing clusters (see \cite[Part~IV]{Mag} and the references therein). Indeed, we will firstly show in Subsection~\ref{subsection:penalization} via a penalization technique that it is possible to remove the volume constraint in \eqref{eq:min_pb} by adding a suitable volume penalization to the functional. Furthermore, \rosso{the term $\nl(u)$ in the energy behaves as a volume-order term thanks to assumption \eqref{eq:ass-nl}.} In view of these two properties, it follows that the partition of $\R^n$ given by $(A_u,B_u,V_u,S)$, for a solution $u$ of \eqref{eq:min_pb}, is a quasi-minimizer cluster for the surface energy
\begin{equation} \label{eq:surfaceen}
\g(u) \coloneqq \int_{J_u\cap Q_L^+}\overline{\Psi}(u^+,u^-)\de\hn, \qquad u\in\ac.
\end{equation}
The precise definition of quasi-minimality in our context is given in Definition~\ref{def:quasimin}  below.

Next, in Subsection~\ref{subsection:regularity2} we exploit the quasi-minimality property to obtain the regularity of minimizers in two dimensions stated in Theorem~\ref{thm:regularity}. Technical difficulties arise from two fronts: on the one hand, we can only compare with clusters that satisfy the constraint of being the subgraph of a function of bounded variation, a fact that poses a severe restriction on the class of competitors. On the other hand, the interfaces between the phases of the cluster are weighted by different surface tension coefficients.
The challenges that arise from these two features prevent us to rely on the standard theory quasi-minimizing clusters, and requires \textit{ad hoc} modifications of the classical proofs. For this reason we develop a regularity theory only in dimension $n=2$, since the general dimensional case requires more refined arguments. We also remark that the regularity properties are obtained under the assumption that the surface tension coefficients satisfy a \emph{strict} triangle inequality (see \eqref{ass:stricttriangle}).

\subsection{Penalization and quasi-minimality} \label{subsection:penalization}

In this section we show that, in any dimension $n\geq 2$, every solution to the minimum problem \eqref{eq:min_pb} is a quasi-minimizer for the surface energy (Proposition~\ref{prop:quasimin}), in the sense of the following definition.

\begin{definition}[Quasi-minimizer] \label{def:quasimin}
We say that $u\in\ac$ is a \emph{quasi-minimizer} for the surface energy $\g$, defined in \eqref{eq:surfaceen}, if there exists $\Lambda>0$ such that for every admissible configuration $v\in\ac$ one has
\begin{equation} \label{eq:quasimin}
\g(u) \leq \g(v) + \Lambda \bigl( |A_u\asymm A_v| + |B_u\asymm B_v| \bigr).
\end{equation}
We denote, for $\Lambda>0$ and $M>0$, by $\mathcal{A}_{\Lambda,M}$ the class of all configurations $u\in\ac$ such that $u$ is a quasi-minimizer for $\g$ with quasi-minimality constant $\Lambda$, and $|\Omega_{h_u}|\leq M$.
\end{definition}

As a first step we remove the mass constraint in \eqref{eq:min_pb} by considering a suitable penalized minimum problem, see \eqref{eq:min_pb_pen_1}. The main idea of the proof is discussed in the Introduction.

\begin{lemma}[Penalization]\label{lem:penalization}
Let $0<m<M<\infty$. Then there exists $\Lambda>0$ such that every solution to the constrained minimum problem \eqref{eq:min_pb} is also a solution to the penalized problem
\begin{equation}\label{eq:min_pb_pen_1}
\min \Bigl\{ \fbar(u) + \Lambda \Bigl( \big| |A_u|-m \big| + \big| |\Omega_{h_u}|- M \big| \Bigr) \,:\,  u\in\ac \Bigr\}.
\end{equation}
\end{lemma}

\begin{proof}
Let $u\in\ac$ be a minimizer for \eqref{eq:min_pb}, consider a sequence $\{\lambda_j\}_{j\in\N}$ with $\lambda_j\to\infty$ as $j\to\infty$, and $u_j\in\ac$ solving the minimum problem
\begin{equation}\label{eq:min_pb_pen_2}
\min \Bigl\{ \mathscr{H}_{\lambda_j}(v) \coloneqq \fbar(v) + \lambda_j \Bigl( \big| |A_v|-m \big| + \big| |\Omega_{h_v}|- M \big| \Bigr) \,:\,  v\in\ac \Bigr\},
\end{equation}
whose existence can be shown arguing as in the proof of Theorem~\ref{thm:existence}.
We will show that, for $j$ large enough, we have
\begin{equation}\label{eq:mass_matches}
|A_{u_j}|=m,\qquad |\Omega_{h_{u_j}}|=M,
\end{equation}
which will imply that $u$ itself is a solution to \eqref{eq:min_pb_pen_2} for $j$ large, as desired.

To prove \eqref{eq:mass_matches} we argue by contradiction and we show that, if at least one of the equalities in \eqref{eq:mass_matches} is not satisfied, then for $j$ large enough it is possible to construct by a local variation a configuration $\widetilde{u}_j\in\ac$ such that $\mathscr{H}_{\lambda_j}(\widetilde{u}_j) < \mathscr{H}_{\lambda_j}(u_j)$.
The construction of the local variation exploits the same diffeomorphism for both of the mass constraints, applied at different points. In the first part of the proof (Steps~1--4) we thus present the construction of the general diffeomorphism and the corresponding estimates for the change of volume, perimeter and nonlocal energy under this perturbation. To simplify the notation, in the rest of the proof we will write $A_j$, $B_j$, and $\Omega_j$ in place of $A_{u_j}$, $B_{u_j}$, and $\Omega_{h_{u_j}}$ respectively.

\medskip\noindent\textit{Step 1: Definition of the diffeomorphism.}
We denote by $B'_r \coloneqq \{ x'\in\R^{n-1} : |x'|< r \}$ the $(n-1)$-dimensional ball centered at the origin with radius $r>0$, and define for $z\in\R^n$
\[
C^+(z,r) \coloneqq z + \left( B'_r \times (0,r) \right),\qquad
C^-(z,r) \coloneqq z + \left( B'_r \times (-r,0) \right),
\]
and
\[
C(z,r)\coloneqq C^+(z,r)\cup C^-(z,r)\cup \bigl( z + B_r'\times\{0\}\bigr).
\]

We next assume that $z=0$ and we define a family of local perturbations in $C(0,r)$. Precisely, for $|\sigma|<r$ we define the map $\Phi_\sigma:\R^n\to\R^n$ by
\begin{equation}\label{eq:Psi}
\Phi_\sigma(x',x_n)\coloneqq
\begin{cases}
\left(x', x_n + \sigma \bigl( 1-\frac{|x'|}{r} \bigr)\bigl(\frac{x_n}{r} - 1 \bigr) \right)  & \text{if } x\in C^+(0,r),\\[2ex]
\left(x', x_n - \sigma \bigl( 1-\frac{|x'|}{r} \bigr)\bigl(\frac{x_n}{r} + 1 \bigr)   \right) & \text{if } x\in C^-(0,r),\\[2ex]
(x',x_n) & \text{if }x\in\R^n\setminus C(0,r).
\end{cases}
\end{equation}
The function $\Phi_\sigma$ is a vertical rescaling with horizontal and vertical cut-off functions.
The role of the parameter $\sigma$ can be seen from Figure~\ref{fig:diffeo}.
Notice that for $|\sigma|<r$ the function $\Phi_\sigma$ is a bi-Lipschitz map and that
$\Phi_\sigma (C(0,s)) = C(0,s)$. Moreover, it holds
\begin{equation}\label{eq:Jacobian_Phi}
D\Phi(x',x_n) = 
\left(
\begin{tabular}{c|c}
$\mathrm{Id}_{n-1}$ & 0 \\
\hline
$v_\sigma(x',x_n)$ & $1+a_\sigma(x',x_n)$
\end{tabular}
\right),
\end{equation}
where $\mathrm{Id}_{n-1}$ is the $(n-1)\times(n-1)$ identity matrix,
\begin{equation}\label{eq:asigma}
a_\sigma(x',x_n) \coloneqq
\begin{cases}
\bigl( 1-\frac{|x'|}{r} \bigr)\frac{\sigma}{r} & \text{if } (x',x_n)\in  C^+(0,r),\\[2ex]
-\bigl( 1-\frac{|x'|}{r} \bigr)\frac{\sigma}{r} & \text{if } (x',x_n)\in  C^-(0,r),
\end{cases}
\end{equation}
and 
\begin{equation}\label{eq:vsigma}
v_\sigma(x',x_n) \coloneqq
\begin{cases}
-\frac{\sigma}{r}\bigl(\frac{x_n}{r}-1\bigr)\frac{x'}{|x'|} & \text{if } (x',x_n)\in  C^+(0,r),\\[2ex]
\frac{\sigma}{r}\bigl(\frac{x_n}{r}+1\bigr)\frac{x'}{|x'|} & \text{if } (x',x_n)\in  C^-(0,r).
\end{cases}
\end{equation}
When we will perform a perturbation localized in a cylinder centered at a point $z\in\R^n$, we will consider the map $x\mapsto z+\Phi_\sigma(x-z)$.

\begin{figure}
\begin{center}
\includegraphics[scale=0.9]{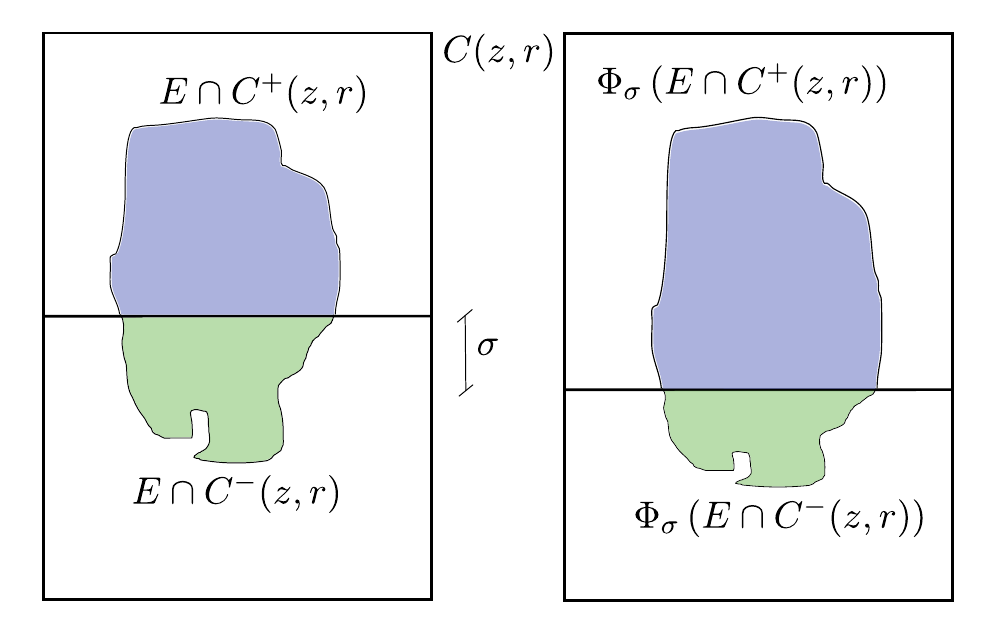}
\caption{The effect of the map $\Phi_\sigma$, for $\sigma>0$, on a set $E$: it stretches the set on $C^+(z,r)$ and it compresses it on $C^-(z,r)$.}
\label{fig:diffeo}
\end{center}
\end{figure}

\medskip\noindent\textit{Step 2: Estimate of the change in volume.}
Let $E\subset C(0,r)$ be a measurable set. We first estimate the maximal change of volume $\left||\Phi_\sigma(E)| - |E|\right|$: by using \eqref{eq:Jacobian_Phi} and \eqref{eq:asigma} we get
\begin{equation}\label{eq:step2_volume_2}
\begin{split}
\Bigl|\, |\Phi_\sigma(E)| - |E| \,\Bigr|
&\leq \frac{|\sigma|}{r}|E\cap C(0,r)|.
\end{split}
\end{equation}

Next, we prove more refined estimates on the change of volume of a set $E$ in the upper and lower cylinders $C^+(0,r)$, $C^-(0,r)$. We first consider the case $\sigma>0$. In this case, the followings hold:
\begin{itemize}
\item[(i)] For every $\varepsilon>0$ and $\sigma\in(0,r)$, if $|E\cap C^+(0,r)|<\varepsilon  r^n$ then
\begin{equation}\label{eq:step2_est_volume_2}
0\leq |\Phi_\sigma(E\cap C^+(0,r))| - |E\cap C^+(0,r))| \leq  U(\varepsilon) |\sigma| r^{n-1},
\end{equation}
where
\[
 U(\varepsilon)\coloneqq \left[\, 1 - \frac{n-1}{n}\Bigl( \frac{\varepsilon}{\omega_{n-1}} \Bigr)^{\frac{1}{n-1}} \,\right] \e.
\]
\item[(ii)] For every $\mu\in(0,\omega_{n-1})$ and $\sigma\in(0,r)$, if $|E\cap C^-(0,r)| > \mu r^n$, then
\begin{equation}\label{eq:step2_est_volume_3}
|\Phi_\sigma(E\cap C^-(0,r))| - |E\cap C^-(0,r))| \leq - L(\mu) |\sigma| r^{n-1}< 0,
\end{equation}
where
\[
L(\mu)\coloneqq \left[\, \frac{1}{n}
	-\left( 1-\frac{\mu}{\omega_{n-1}} \right)
	+\frac{n-1}{n}\left( 1-\frac{\mu}{\omega_{n-1}} \right)^{\frac{n}{n-1}} \,\right]\omega_{n-1}.
\]
\end{itemize}
To prove \eqref{eq:step2_est_volume_2} we notice that by \eqref{eq:Jacobian_Phi} and \eqref{eq:asigma}, and since $\sigma>0$,
\[
|\Phi_\sigma(E\cap C^+(0,r))| - |E\cap C^+(0,r)| = \frac{\sigma}{r}\int_{E\cap C^+(0,r)}\left( 1-\frac{|x'|}{r} \right)\de x \leq \frac{\sigma}{r} \int_{F_\varepsilon} \left( 1-\frac{|x'|}{r} \right)\de x ,
\]
where $F_\varepsilon\coloneqq B'_{r_\varepsilon}\times (0,r)$ and $r_\varepsilon\coloneqq r(\varepsilon/\omega_{n-1})^{\frac{1}{n-1}}$. Similarly we obtain \eqref{eq:step2_est_volume_3} by comparison with $G_{\mu}\coloneqq C^-(0,r)\setminus (B'_{s_{\mu}}\times (-r,0))$, $s_{\mu}\coloneqq r(1-\mu/\omega_{n-1})^{\frac{1}{n-1}}$.

In \eqref{eq:step2_est_volume_2}--\eqref{eq:step2_est_volume_3} we have written $|\sigma|$ in place of $\sigma$ to stress the fact that the same estimates hold also in the case $\sigma<0$ up to exchanging the roles of $C^+(0,r)$ and $C^-(0,r)$, as can be easily checked. Notice that $U(\varepsilon)\to0$ as $\varepsilon\to0$, and that $L(\mu)$ is strictly positive, and more precisely $L(\mu)\in(0,\frac{\omega_{n-1}}{n})$ for every choice of $\mu\in(0,\omega_{n-1})$, with $L(\mu)\to0$ as $\mu\to0$, $L(\mu)\to\frac{\omega_{n-1}}{n}$ as $\mu\to\omega_{n-1}$.

\medskip\noindent\textit{Step 3: Estimate of the change in perimeter.}
Given a countably $\mathcal{H}^{n-1}$-rectifiable set $\Sigma\subset\R^{n}$, by the generalized area formula (see \cite[Theorem~2.91]{AFP}) we have that
\begin{equation}\label{eq:area_formula}
\mathcal{H}^{n-1}(\Phi_\sigma(\Sigma)) - \mathcal{H}^{n-1}(\Sigma)
	=  \int_\Sigma \left( J_{n-1} \mathrm{d}_x^\Sigma\Phi_\sigma - 1 \right) \de\hn(x),
\end{equation}
where $\mathrm{d}_x^\Sigma\Phi_{\sigma}:\pi^\Sigma_x\to\R^n$ denotes the tangential differential of $\Phi_{\sigma}$ at $x\in\Sigma$ along the approximate tangent space $\pi_x^\Sigma$ to $\Sigma$, and the area factor $J_{n-1} \mathrm{d}_x^\Sigma\Phi_\sigma$ is defined as (see \cite[Definition~2.68]{AFP})
\begin{equation}\label{eq:Jacobian}
J_{n-1} \mathrm{d}_x^\Sigma\Phi_\sigma \coloneqq 
\sqrt{\mathrm{det} \bigl( (\mathrm{d}_x^\Sigma\Phi_{\sigma})^* \circ \mathrm{d}_x^\Sigma\Phi_{\sigma} \bigr)}
\end{equation}
(here $(\mathrm{d}_x^\Sigma\Phi_{\sigma})^*$ is the adjoint of the linear map $\mathrm{d}_x^\Sigma\Phi_{\sigma}$).
In order to estimate \eqref{eq:Jacobian}, fix $x\in\Sigma$ and let $\tau_1,\dots,\tau_{n-1}$ be an orthonormal basis for the approximate tangent space $\pi^\Sigma_x$.
By using \eqref{eq:Jacobian_Phi}, \eqref{eq:asigma}, and \eqref{eq:vsigma}, for all $i,j\in\{1,\dots,n-1\}$ we have
\begin{align*}
\bigl( (\mathrm{d}_x^\Sigma\Phi_{\sigma})^* \circ \mathrm{d}_x^\Sigma\Phi_{\sigma} \bigr)_{ij} 
& =	\tau_i\cdot\tau_j + \tau_i^n (\tau_j\cdot w_\sigma(x))+ \tau_j^n (\tau_i\cdot w_\sigma(x)) + (\tau_i\cdot w_\sigma(x))(\tau_j\cdot w_\sigma(x)),
\end{align*}
where $\Phi_{\sigma}=(\Phi_{\sigma}^1,\ldots,\Phi_{\sigma}^n)$ and $\tau_i=(\tau_i^1,\ldots,\tau_i^n)$ denote the components with respect to the canonical base of $\R^n$, and $w_\sigma(x)\coloneqq(v_\sigma(x),a_\sigma(x))$.
By using the fact that $|w_\sigma(x)|\leq \sqrt{2}|\sigma|/r$ and the general formula $\mathrm{det}(I+t A) = 1 + t\,\mathrm{trace}(A) + O(t^2)$ as $t\to0$, we get
\[
\mathrm{det} \bigl( (\mathrm{d}_x^\Sigma\Phi_{\sigma})^* \circ \mathrm{d}_x^\Sigma\Phi_{\sigma} \bigr)
	= 1 + 2\sum_{i=1}^{n-1} \tau_i^n (\tau_i\cdot w_\sigma(x))
		+ O\Bigl(\Bigl(\frac{\sigma}{r}\Bigr)^2\Bigr),
\]
where $\big|O\bigl(\bigl(\frac{\sigma}{r}\bigr)^2\bigr)\big| \leq C\bigl(\frac{\sigma}{r}\bigr)^2$ for a constant $C>0$ independent of $\sigma$, $r$, and of $x\in\Sigma$. Therefore by \eqref{eq:area_formula} we find for $|\sigma|$ sufficiently small
\[
\left| J_{n-1} \mathrm{d}_x^\Sigma\Phi_\sigma - 1 \right|
	=\left| \sqrt{\mathrm{det} \bigl( (\mathrm{d}_x^\Sigma\Phi_{\sigma})^* \circ \mathrm{d}_x^\Sigma\Phi_{\sigma} \bigr)} - 1 \right|
\leq c_0\frac{|\sigma|}{r},
\]
where $c_0>0$ is a dimensional constant. This, together with \eqref{eq:area_formula}, yields
\begin{equation}\label{eq:estimate_coarea}
\left|\, \mathcal{H}^{n-1}(\Phi_\sigma(\Sigma))
	- \mathcal{H}^{n-1}(\Sigma) \,\right|
	\leq c_0\frac{|\sigma|}{r}\mathcal{H}^{n-1}(\Sigma).
\end{equation}

\medskip\noindent\textit{Step 4: Estimate of the change of the term $\nl(\cdot)$.}
Finally, we estimate the change \rosso{in the term $\nl(\cdot)$ of the energy. We note that, by assumption \eqref{eq:ass-nl},} it is enough to get an estimate on $|\Phi_\sigma(E)\asymm E|$ for a general set $E$ with finite perimeter.
By the same computation as in \cite[Proposition~2.7]{AceFusMor13}, writing $\Phi_{\sigma}^{-1}(x',x_n)=(x',x_n+\phi_\sigma(x_n))$ with $|\phi_\sigma(x_n)|\leq |\sigma|$, for $f\in C^1(\R^n)$ we have
\begin{equation}\label{eq:estimate_smooth_NL}
\int_{C(0,r)} |f - f\circ \Phi_\sigma^{-1}|\de x\leq |\sigma|\int_{C(0,r)} |\nabla f(x)|\de x.
\end{equation}
Let now $E\subset\R^n$ be a set with finite perimeter and let $\{f_k\}_{k\in\N}$ be a sequence of smooth functions such that $f_k\to \chi_E$ in $L^1$ and $\|\nabla f_k\|_{L^1}\to \per(E)$. Then also $f_k\circ \Phi_\sigma^{-1}\to \chi_E\circ\Phi_\sigma^{-1}$ in $L^1$. Therefore applying \eqref{eq:estimate_smooth_NL} to the function $f_k$ and passing to the limit as $k\to\infty$ yields
\begin{equation}\label{eq:estimate_NL}
|\Phi_\sigma(E)\asymm E| = \int_{C(0,r)} \big| \chi_E-\chi_E\circ \Phi_{\sigma}^{-1} \big| \de x \leq |\sigma|\per(E).
\end{equation}

\medskip\noindent\textit{Step 5: General strategy.}
We can now go back to the main argument of the proof and show that any solution $u_j$ of the penalized problem \eqref{eq:min_pb_pen_2} satisfies the mass constraints \eqref{eq:mass_matches}, for $j$ large enough.
The idea of the proof is to assume by contradiction that one of the mass constraints in \eqref{eq:mass_matches} is not satisfied, and to construct a perturbation of $u_j$ in a cylinder $C(z,r)$ by means of the maps $\Phi_{\sigma_j}$. More precisely, we will choose a point $z\in\R^n$, a radius $r>0$ and scaling coefficients $\sigma_j\in(-r,r)$ and define
\begin{equation} \label{eq:penaliz_perturb}
\widetilde{u}_j (x) \coloneqq u_j (z + \Phi_{\sigma_j}^{-1}(x-z)).
\end{equation}
This is a local perturbation inside $C(z,r)$ (the center and the radius will be chosen in such a way that the cylinder does not intersect the substrate $S$) such that the phases of the new configuration $\widetilde{u}_j$ are given by
\[
\widetilde{A}_j = z + \Phi_{\sigma_j}(A_j-z),\qquad
\widetilde{B}_j = z + \Phi_{\sigma_j}(B_j-z),\qquad
\widetilde{\Omega}_j = z + \Phi_{\sigma_j}(\Omega_j-z).
\]
Thanks to \eqref{eq:estimate_coarea}, \eqref{eq:estimate_NL}, and \rosso{\eqref{eq:ass-nl}}, we get the estimate
\begin{equation} \label{eq:goal_penalization2}
\begin{split}
\mathscr{H}_{\lambda_j}(\widetilde{u}_j) - \mathscr{H}_{\lambda_j}(u_j)
& \leq \Bigl(\frac{\tilde{c}_0}{r}+\gamma L_{\nl}\Bigr)\bigl(\per(A_j) + \per(B_j)\bigr)|\sigma_j| \\
& \quad + \lambda_j \Bigl( \bigl| |\widetilde{A}_j|-m \bigr| + \bigl| |\widetilde{\Omega}_j|- M \bigr|
	- \bigl| |A_{j}|-m \bigr| - \bigl| |\Omega_{j}|- M \bigr| \Bigr),
\end{split}
\end{equation}
where $\tilde{c}_0$ depends on the constant $c_0$ in \eqref{eq:estimate_coarea} and on the surface tension coefficients.
The goal would be then to show that, if at least one of the volume constraints is not satisfied, then it is possible to choose $z$, $r$ and $\sigma_j$ so that 
\begin{equation}\label{eq:goal_penalization}
\bigl| |\widetilde{A}_j|-m \bigr| + \bigl| |\widetilde{\Omega}_j|- M \bigr|
	- \bigl| |A_{j}|-m \bigr| - \bigl| |\Omega_{j}|- M \bigr| \leq -C |\sigma_j|r^{n-1},
\end{equation}
for some $C>0$ independent of $j$. As $\lambda_j\to\infty$, the combination of \eqref{eq:goal_penalization2} and \eqref{eq:goal_penalization} shows that $\mathscr{H}_{\lambda_j}(\widetilde{u}_j)<\mathscr{H}_{\lambda_j}(u_j)$ for $j$ large enough, which is a contradiction with the minimality of $u_j$ in \eqref{eq:min_pb_pen_2}.

In the next two steps we will implement the previous strategy. We first observe that, by using $u$ as a competitor in the minimum problem \eqref{eq:min_pb_pen_2} and since $\fbar(u)<\infty$, we obtain the bounds
\begin{equation}\label{eq:upper_bound_volume_perimeter}
\sup_{j\in\N} \Bigl( \per(A_{j}) + \lambda_j\big| |A_{j}|-m \big| \Bigr) < \infty,\quad
\sup_{j\in\N} \Bigl( \per(B_{j}) + \lambda_j \big| |B_{j}|-(M-m) \big| \Bigr) < \infty .
\end{equation}
Thus, up to a subsequence (not relabeled), we get that $A_{j}\to A$ and $B_{j}\to B$ in $L^1$, with $|A|=m$, $|B|=M-m$ since $\lambda_j\to\infty$. We also have $\Omega_j\to\Omega\coloneqq A\cup B$. Notice that $\Omega$ is still the subgraph of an admissible profile.

In the following, given a point $x_0\in\R^n$, $r>0$, and a direction $\nu = (\nu',\nu_n)\in\mathbb{S}^{n-1}$ with $\nu_n\neq 0$, we define
\begin{equation} \label{eq:y_r}
y_r\coloneqq x_0 + \left( r\,\cos\left(\arctan\frac{|\nu_n|}{|\nu'|} \right) \right) \nu
\end{equation}
and we consider the corresponding cylinder $C(y_r,r)$. The choice of the point $y_r$ guarantees, since $\nu_n\neq0$, that there exists a constant $c_\nu>0$, independent of $r$, such that if $\nu_n>0$
\begin{equation}\label{eq:cnu}
C^+(y_r,r)\cap \{ (x-x_0)\cdot \nu < 0 \} = \emptyset, \qquad
\big| C^-(y_r,r)\cap \{ (x-x_0)\cdot \nu < 0 \} \big|= 2c_\nu r^n,
\end{equation}
while if $\nu_n<0$
\begin{equation}\label{eq:cnu_1}
C^-(y_r,r)\cap \{ (x-x_0)\cdot \nu < 0 \} = \emptyset, \qquad
\big| C^+(y_r,r)\cap \{ (x-x_0)\cdot \nu < 0 \} \big|= 2c_\nu r^n.
\end{equation}
Notice that the strict positivity of $c_\nu$ is a consequence of the fact that $\nu_n\neq0$.

\medskip\noindent\textit{Step 6: Fixing the total volume.}
Assume by contradiction that $|\Omega_j|\neq M$ for infinitely many $j$. We will consider for simplicity the case $|\Omega_j|>M$ for all $j$, as the other case can be treated by a similar argument.

\smallskip\textbf{Case 1.} Assume that there exists $x_0\in\partial^*\Omega\cap\partial^*B$ such that $\nu_{\Omega}(x_0)\cdot e_n > 0$ (where $\nu_{\Omega}$ denotes the exterior normal). We consider the point $y_r$ and the constant $c_\nu$ defined in \eqref{eq:y_r} and \eqref{eq:cnu} respectively, for $\nu=\nu_{\Omega}(x_0)$ and $r>0$ to be chosen later.

De Giorgi's structure theorem for sets of finite perimeter (\cite[Theorem~3.59]{AFP}) together with \eqref{eq:cnu} ensures that
\[
\lim_{r\to0} \frac{|\Omega\cap C^+(y_r,r)|}{r^n} = \lim_{r\to0} \frac{|A\cap C(y_r,r)|}{r^n} =0,\quad\quad\quad
\lim_{r\to0} \frac{|\Omega\cap C^-(y_r,r)|}{r^n} = 2c_\nu.
\]
Therefore, for every $\e>0$, the fact that $\chi_{\Omega_j}\to\chi_\Omega$, $\chi_{A_j}\to\chi_A$ and $\chi_{B_j}\to\chi_B$ in $L^1$ yields the existence of $r\in(0,1)$ and $j_0\in\N$ such that for all $j\geq j_0$ the following holds:
\begin{equation}\label{eq:stime_Omegaj_step5_case1}
|\Omega_j\cap C^+(y_r,r)|< \e r^n,\quad\quad\quad
|\Omega_j\cap C^-(y_r,r)| > c_\nu r^n,
\end{equation}
\begin{equation}\label{eq:stime_Aj_step5_case1}
|A_j\cap C(y_r,r)|< \e r^n.
\end{equation}
Moreover, for $r$ small enough we can also guarantee that the cylinder $C(y_r,r)$ is contained in the upper half-space and does not intersect the substrate.
We then choose $\sigma_j>0$ and consider the perturbation defined in \eqref{eq:penaliz_perturb} centered at the point $z=y_r$. In view of \eqref{eq:stime_Omegaj_step5_case1}, by using \eqref{eq:step2_est_volume_2} and \eqref{eq:step2_est_volume_3}, we get 
\[
|\widetilde{\Omega}_j| - |\Omega_j| \leq - \bigl( L(c_\nu)-U(\e) \bigr) \sigma_j  r^{n-1}.
\]
On the other hand, by \eqref{eq:step2_volume_2} and \eqref{eq:stime_Aj_step5_case1},
\[
\big| |\widetilde{A}_j|-|A_j|\big| \leq \frac{\sigma_j}{r}|A_j\cap C(y_r,r)| \leq \e \sigma_j r^{n-1}.
\]
Therefore, noting that we can assume $M<|\widetilde{\Omega}_j|<|\Omega_j|$ (it is sufficient to choose $\sigma_j$ and $\e$ small enough), we find
\begin{align*}
\big| |\widetilde{\Omega}_j| - M \bigr|
	- \bigl| |\Omega_j| - M \bigr|
	+ \bigl| |\widetilde{A}_j|-m \bigr|
	- \bigl| |A_j|-m \bigr|
& \leq |\widetilde{\Omega}_j| - |\Omega_j|
	+\bigl| |\widetilde{A}_j| - |A_j| \bigr| \\
& \leq -\bigl(L(c_\nu) - U(\e) - \e \bigr) \sigma_j r^{n-1}.
\end{align*}
By choosing $\e$ sufficiently small, we can ensure that $U(\e)+\e<L(c_\nu)$, which yields \eqref{eq:goal_penalization} and leads to the desired contradiction in this case.

\smallskip\textbf{Case 2.}
If the assumption of the previous case does not hold, we can find a point $x_1\in\partial^*\Omega\cap\partial^*A$ such that $\nu_{\Omega}(x_1)\cdot e_n> 0$. Since $0<m<M$, it is possible to find a second point $x_2\in\partial^* A\cap\partial^* B$ such that $\nu_A(x_2)\cdot e_n < 0$.

We will consider the composition of two perturbations of the form \eqref{eq:penaliz_perturb} localized in two disjoint cylinders $C(y_r^1,r)$ and $C(y_r^2,r)$, where (see also \eqref{eq:y_r})
\begin{align*}
y^1_r & \coloneqq x_1 + \left[ r\,\cos\left(\arctan\frac{|(\nu_{\Omega}(x_1))_n|}{|(\nu_{\Omega}(x_1))'|} \right)\right]
	\nu_{\Omega}(x_1),\\
y^2_r & \coloneqq x_2 + \left[ r\,\cos\left(\arctan\frac{|(\nu_{A}(x_2))_n|}{|(\nu_{A}(x_2))'|} \right)\right]
	\nu_{A}(x_2).
\end{align*}
Let
\[
E^1_r \coloneqq  \{ (x-x_1)\cdot \nu_{\Omega}(x_1) <0 \}  \cap C(y^1_r,r),\quad\quad
E^2_r \coloneqq  \{ (x-x_2)\cdot \nu_{A}(x_2) <0 \} \cap C(y^2_r,r).
\]
Note that $E^1_r\subset C^-(y^1_r,r)$ and that $E^2_r\subset C^+(y^2_r,r)$. We let $\mu\coloneqq c_{\nu_{\Omega}(x_1)}$, where the constant $c_\nu$, for a vector $\nu$, is defined in \eqref{eq:cnu}.
As in the previous case, fixed $\e>0$, we can find $r>0$ and $j_0\in\N$ such that for all $j\geq j_0$ we have
\begin{equation}\label{eq:penaliz_step6a}
|\Omega_j\cap C^+(y^1_r,r)|< \e r^n, \qquad 
|\Omega_j\cap C^-(y^1_r,r)| > \mu r^n,
\end{equation}
\begin{equation}\label{eq:penaliz_step6b}
|V_j\cap C(y^2_r,r)|<\e r^n,
\end{equation}
and
\begin{equation}\label{eq:penaliz_step6c}
|\left( A_j\cap C(y^1_r,r)\right) \triangle E^1_r| < \e r^n,\quad\quad
|\left( A_j\cap C(y^2_r,r)\right) \triangle E^2_r| < \e r^n.
\end{equation}
By reducing the value of $r>0$ we can further assume that the two cylinders $C(y^1_r,r)$ and $C(y^2_r,r)$ are disjoint and do not intersect the substrate.
For a fixed sequence $\{\sigma_j^1\}_{j\in\N}\subset(0,r)$, we define a second sequence $\{\sigma_j^2\}_{j\in\N}$ as
\begin{equation}\label{eq:sigma1_sigma2}
\sigma^2_j \coloneqq \alpha\sigma^1_j, \qquad\text{where }\alpha\coloneqq \frac{\int_{E^1_r-y_r^1} \bigl( 1-\frac{|x'|}{r} \bigr)\de x}{\int_{E^2_r-y_r^2} \bigl( 1-\frac{|x'|}{r} \bigr)\de x}
\end{equation}
for each $j\in\N$. Notice that $\alpha$ is independent of $r$ by scale invariance.
Then, we consider the configuration $\widetilde{u}_j$ obtained by applying to $u_j$ the composition of the two perturbations $y_r^1+\Phi_{\sigma^1_j}(\cdot-y_r^1)$ and $y_r^2+\Phi_{\sigma^2_j}(\cdot-y_r^2)$. We denote the sets of the new partition determined by $\widetilde{u}_j$ by $\widetilde{A}_j$, $\widetilde{B}_j$, $\widetilde{\Omega}_j=\widetilde{A}_j\cup\widetilde{B}_j$, $\widetilde{V_j}$.

We first consider the variation of the volume of $\Omega_j$. By \eqref{eq:penaliz_step6a}, and since $\sigma^1_j>0$, we can apply \eqref{eq:step2_est_volume_2} and \eqref{eq:step2_est_volume_3} and obtain
\[
|\widetilde{\Omega}_j\cap C(y_r^1,r)| - |\Omega_j\cap C(y_r^1,r)| \leq -\bigl(L(\mu)-U(\e)\bigr)\sigma_j^1 r^{n-1}.
\] 
On the other hand, by \eqref{eq:penaliz_step6b} and using \eqref{eq:step2_volume_2} we find
\begin{equation*}
\begin{split}
|\widetilde{\Omega}_j\cap C(y_r^2,r)| - |\Omega_j\cap C(y_r^2,r)|
& = |V_j\cap C(y_r^2,r)| - |\widetilde{V}_j\cap C(y_r^2,r)| \\
& \leq \frac{\sigma_j^2}{r}|V_j\cap C(y_r^2,r)| \leq \e \sigma_j^2 r^{n-1}.
\end{split}
\end{equation*}
By combining the two estimates and recalling \eqref{eq:sigma1_sigma2} it follows that
\begin{equation} \label{eq:penaliz_step6d}
|\widetilde{\Omega}_j| - |\Omega_j| \leq - \Bigl( L(\mu) - U(\e) - \alpha\e \Bigr) \sigma_j^1 r^{n-1}.
\end{equation}

Next, we look at the variation of the volume of $A_j$. We have for $i=1,2$
\begin{align}\label{eq:est_Aj_cylinder}
|\widetilde{A}_j\cap C(y_r^i,r)| - |A_j\cap C(y_r^i,r)| = |\Phi_{\sigma^i_j}(E^i_r-y_r^i)| - |E^i_r| + R_j^i,
\end{align}
where thanks to \eqref{eq:penaliz_step6c} and \eqref{eq:step2_volume_2}
\begin{equation}\label{eq:reminder}
|R_j^i| \leq \big| |\Phi_{\sigma^i_j}((A_j\cap C(y_r^i,r))\asymm E_r^i-y_r^i)|  - |(A_j\cap C(y_r^i,r))\asymm E_r^i| \big| \leq \e\sigma_j^i r^{n-1}.
\end{equation}
Also notice that the choice of $\sigma_j^2$ in \eqref{eq:sigma1_sigma2} guarantees exactly that
\begin{multline}\label{eq:eliminate}
\Bigl( |\Phi_{\sigma^1_j}(E^1_r-y_r^1)| - |E^1_r| \Bigr) + \Bigl( |\Phi_{\sigma^2_j}(E^2_r-y_r^2)| - |E^2_r| \Bigr) \\
=  -\frac{\sigma^1_j}{r}\int_{E^1_r-y_r^1}\Bigl(1-\frac{|x'|}{r}\Bigr)\de x + \frac{\sigma_j^2}{r}\int_{E^2_r-y_r^2}\Bigl(1-\frac{|x'|}{r}\Bigr)\de x = 0.
\end{multline}
Therefore, from \eqref{eq:est_Aj_cylinder}, \eqref{eq:reminder}, and \eqref{eq:eliminate}, we get
\begin{equation}\label{eq:penaliz_step6e}
\begin{split}
\big| |\widetilde{A}_j| - |A_j| \big|
& = \big| |\widetilde{A}_j\cap C(y_r^1,r)| - |A_j\cap C(y_r^1,r)| + |\widetilde{A}_j\cap C(y_r^2,r)| - |A_j\cap C(y_r^2,r)|\big| \\
& \leq \big| |\Phi_{\sigma^1_j}(E^1_r-y_r^1)| - |E^1_r| + |\Phi_{\sigma^2_j}(E^2_r-y_r^2)| - |E^2_r| \big| + |R_j^1| + |R_j^2| \\
& \leq \e (\sigma_j^1+\sigma_j^2) r^{n-1} = \e\bigl(1+\alpha\bigr)\sigma_j^1 r^{n-1}.
\end{split}
\end{equation}

We can now conclude as follows. Similarly to \eqref{eq:goal_penalization2} we find
\begin{align}\label{eq:penaliz_step6e-bis}
\mathscr{H}_{\lambda_j}(\widetilde{u}_j) - \mathscr{H}_{\lambda_j}(u_j)
& \leq \Bigl(\frac{\tilde{c}_0}{r}+\gamma L_{\nl}\Bigr)\bigl(\per(A_j) + \per(B_j)\bigr)(\sigma_j^1+\sigma_j^2) \nonumber\\
& \qquad + \lambda_j \Bigl( \bigl| |\widetilde{A}_j|-m \bigr| + \bigl| |\widetilde{\Omega}_j|- M \bigr| - \bigl| |A_{j}|-m \bigr| - \bigl| |\Omega_{j}|- M \bigr| \Bigr) \nonumber\\
& \leq C(1+\alpha)\sigma_j^1
	+ \lambda_j \Bigl( \bigl| |\widetilde{A}_j|- |A_{j}| \bigr|
	+ |\widetilde{\Omega}_j| - |\Omega_{j}| \Bigr) \nonumber\\
& \leq \biggl[ C(1+\alpha) - \lambda_j\Bigl( L(\mu) - U(\e) - \alpha\e - (1+\alpha)\e \Bigr)r^{n-1} \biggr] \sigma_j^1
\end{align}
where we used \eqref{eq:upper_bound_volume_perimeter} in the second inequality, and \eqref{eq:penaliz_step6d}, \eqref{eq:penaliz_step6e} in the last one. We can therefore choose $\e>0$ small enough so that the constant multiplying $\lambda_j$ is strictly negative; as $\lambda_j\to+\infty$, this provides the desired contradiction with the minimality of $u_j$.

\medskip\noindent\textit{Step 7: Fixing the volume of each phase.}
In this step we conclude the proof by showing that $|A_j|=m$ for $j$ large. Thanks to the previous step, we can assume that $|\Omega_j|=M$ for all $j\in\N$. Suppose by contradiction that $A_j\neq m$ for infinitely many $j$. We consider for simplicity only the case $|A_j|>m$ for all $j$, as the other case can be treated with similar computations.

\smallskip\textbf{Case 1.} Assume that there exists $x_0\in\partial^* A\cap\partial^* B$ such that $\nu_A(x_0)\cdot e_n \neq 0$. We assume to fix the ideas to be in the case $\nu_A(x_0)\cdot e_n>0$; in the other case, it is sufficient to exchange the roles of the upper and lower cylinders in the computations below. We consider, for $r>0$ to be chosen, the point $y_r$ and the constant $c_\nu$ defined in \eqref{eq:y_r} and \eqref{eq:cnu} respectively, corresponding to $\nu=\nu_A(x_0)$.

Fix $\varepsilon>0$. By De Giorgi's structure theorem and the convergence $\chi_{A_j}\to\chi_A$, there exist $r>0$ and $j_0\in\N$ such that for all $j\geq j_0$ it holds
\begin{equation}\label{eq:stime_Aj_step7}
|A_j\cap C^+(y_r,r)|<\varepsilon r^n,\quad\quad\quad
|A_j\cap C^-(y_r,r)| > c_\nu r^n,
\end{equation}
and
\begin{equation}\label{eq:stime_Vj_step7}
|V_j\cap C(y_r,r)| < \varepsilon r^n.
\end{equation}
Moreover, for $r$ small enough we can also guarantee that the cylinder $C(y_r,r)$ is contained in the upper half-space and does not intersect the substrate.
We then choose $\sigma_j>0$ and consider the perturbation defined in \eqref{eq:penaliz_perturb} centered at the point $z=y_r$.
From \eqref{eq:stime_Aj_step7}, \eqref{eq:step2_est_volume_2}, and \eqref{eq:step2_est_volume_3}, we have
\[
|\widetilde{A}_j|- |A_j|\leq -(L(c_\nu) - U(\varepsilon))\sigma_j r^{n-1}.
\]
Moreover, from \eqref{eq:stime_Vj_step7} and \eqref{eq:step2_volume_2} we can estimate
\[
\bigl| |\widetilde{V}_j\cap C(y_r,r)|
	- |V_j\cap C(y_r,r)| \bigr| \leq \varepsilon \sigma_j r^{n-1}.
\]
Thus, using the fact that $|\Omega_j|=M$ for all $j\in\N$, and that $m<|\widetilde{A}_j|<|A_j|$ (by choosing $\sigma_j$ and $\varepsilon$ small enough), we obtain
\begin{align*}
\bigl| |\widetilde{A}_j|-m \bigr| + \bigl| |\widetilde{\Omega}_j|- M \bigr|
	& - \bigl| |A_{j}|-m \bigr| - \bigl| |\Omega_{j}|- M \bigr| 
 = \bigl| |\widetilde{\Omega}_j| - |\Omega_j| \bigr| + |\widetilde{A}_j| - |A_j| \\
& = \bigl| |\widetilde{V}_j\cap C(y_r,r)| - |V_j\cap C(y_r,r)| \bigr| + |\widetilde{A}_j| - |A_j| \\
&\leq - ( L(c_\nu) - U(\varepsilon) - \varepsilon )\sigma_j r^{n-1}.
\end{align*}
Therefore, by choosing $\varepsilon>0$ small enough we get \eqref{eq:goal_penalization}, as desired.

\smallskip\textbf{Case 2.}
Finally, assume that $\nu_A(x)\cdot e_n = 0$ for all $x\in\partial^* A\cap\partial^* B$. The construction in this case is similar to the one in Step~6, Case~2. Since $0<m<M$, we get that there exist $x_1\in\partial^* A\cap\partial^*V$ and $x_2\in\partial^* B\cap\partial^*V$. We let, for $r>0$, $y_r^1$ and $y_r^2$ be the points defined by \eqref{eq:y_r} corresponding to the choice of $\nu_{\Omega}(x_1)$ and $\nu_{\Omega}(x_2)$, respectively, and
\[
E^1_r \coloneqq  \{ (x-x_1)\cdot \nu_{\Omega}(x_1)<0 \}  \cap C(y^1_r,r),\quad\quad
E^2_r \coloneqq  \{ (x-x_2)\cdot \nu_{\Omega}(x_2)<0 \}  \cap C(y^2_r,r).
\]
We let $\mu\coloneqq c_{\nu_{\Omega}(x_1)}>0$, where the constant $c_\nu$, for a vector $\nu$, is defined in \eqref{eq:cnu}.
As in the previous cases, for fixed $\e>0$, we can find $r>0$ and $j_0\in\N$ such that for all $j\geq j_0$ we have
\begin{equation}\label{eq:stime_x1_step7}
|\left( \Omega_j\cap C(y^1_r,r) \right) \triangle E^1_r |<\varepsilon r^n,\quad\quad
|A_j\cap C^-(y^1_r,r)| > \mu r^n,\quad\quad
|A_j\cap C^+(y^1_r,r)|<\varepsilon r^n
\end{equation}
and
\begin{equation}\label{eq:stime_x2_step7}
|\left( \Omega_j\cap C(y^2_r,r) \right) \triangle E^2_r |<\varepsilon r^n,\quad\quad
|A_j\cap C(y^2_r,r)|<\varepsilon r^n.
\end{equation}
By reducing the value of $r>0$ we can further assume that the two cylinders $C(y_r^1,r)$ and $C(y_r^2,r)$ are disjoint and do not intersect the substrate. For a fixed sequence $\{\sigma_j^1\}_{j\in\N}\subset(0,r)$, we define a second sequence $\{\sigma_j^2\}_{j\in\N}$ as
\begin{equation}\label{eq:sigma1_sigma2b}
\sigma^2_j \coloneqq -\alpha\sigma^1_j, \qquad\text{where }\alpha\coloneqq \frac{\int_{E^1_r-y_r^1} \bigl( 1-\frac{|x'|}{r} \bigr)\de x}{\int_{E^2_r-y_r^2} \bigl( 1-\frac{|x'|}{r} \bigr)\de x}
\end{equation}
for each $j\in\N$. Notice that $\sigma_j^2<0$ and that $\alpha$ is independent of $r$ by scale invariance.
Then, we consider the configuration $\widetilde{u}_j$ obtained by applying to $u_j$ the composition of the perturbations $y_r^1+\Phi_{\sigma^1_j}(\cdot-y_r^1)$, $y_r^2+\Phi_{\sigma^2_j}(\cdot-y_r^2)$. We denote the sets of the new partition determined by $\widetilde{u}_j$ by $\widetilde{A}_j$, $\widetilde{B}_j$, $\widetilde{\Omega}_j=\widetilde{A}_j\cup\widetilde{B}_j$, $\widetilde{V_j}$.

We first consider the variation of the volume of $A_j$. By using the last two inequalities in \eqref{eq:stime_x1_step7} and the last inequality in \eqref{eq:stime_x2_step7}, together with \eqref{eq:step2_est_volume_2}, \eqref{eq:step2_est_volume_3}, \eqref{eq:step2_volume_2}, we get
\begin{equation}\label{eq:step7_A}
|\widetilde{A}_j| - |A_j| \leq
	-\bigl( L(\mu) - U(\e) \bigr) \sigma_j^1 r^{n-1}
	+  \e |\sigma_j^2| r^{n-1}.
\end{equation}
The choice \eqref{eq:sigma1_sigma2b} guarantees that
\begin{multline*}
\Bigl( |\Phi_{\sigma^1_j}(E^1_r-y_r^1)| - |E^1_r| \Bigr) + \Bigl( |\Phi_{\sigma^2_j}(E^2_r-y_r^2)| - |E^2_r| \Bigr) \\
=  -\frac{\sigma^1_j}{r}\int_{E^1_r-y_r^1}\Bigl(1-\frac{|x'|}{r}\Bigr)\de x - \frac{\sigma_j^2}{r}\int_{E^2_r-y_r^2}\Bigl(1-\frac{|x'|}{r}\Bigr)\de x = 0,
\end{multline*}
therefore by arguing as in \eqref{eq:penaliz_step6e} we find
\begin{equation}\label{eq:step7_Omega}
\big| |\widetilde{\Omega}_j| - |\Omega_j| \big|
 \leq \e (\sigma_j^1+|\sigma_j^2|) r^{n-1} = \e\bigl(1+\alpha\bigr)\sigma_j^1 r^{n-1},
\end{equation}
and in turn, similarly to \eqref{eq:penaliz_step6e-bis} (using \eqref{eq:upper_bound_volume_perimeter} \eqref{eq:step7_A}, and \eqref{eq:step7_Omega})
\begin{align*}
\mathscr{H}_{\lambda_j}(\widetilde{u}_j) - \mathscr{H}_{\lambda_j}(u_j)
 \leq \biggl[ C(1+\alpha) - \lambda_j\Bigl( L(\mu) - U(\e) - \alpha\e - (1+\alpha)\e \Bigr)r^{n-1} \biggr] \sigma_j^1.
\end{align*}
We can therefore choose $\e>0$ small enough so that the constant multiplying $\lambda_j$ is strictly negative; as $\lambda_j\to+\infty$, this provides the desired contradiction with the minimality of $u_j$.
\end{proof}

The result of Lemma~\ref{lem:penalization} together with the \rosso{Lipschitz continuity assumption \eqref{eq:ass-nl} of the term $\nl(\cdot)$} easily yields the quasi-minimality of solutions to \eqref{eq:min_pb}.

\begin{proposition} \label{prop:quasimin}
Let $u$ be a solution to the minimum problem \eqref{eq:min_pb}. Then $u$ is a quasi-minimizer for the surface energy $\g$, according to Definition~\ref{def:quasimin}.
\end{proposition}

A standard truncation argument shows that quasi-minimizers are bounded.

\begin{proposition}[Boundedness] \label{prop:truncation}
Let $u\in\mathcal{A}_{\Lambda,M}$. Then $h_u\in L^\infty(Q_L)$.
\end{proposition}

\begin{proof}
The boundedness is a consequence of the finiteness of the volume of $\Omega_{h_u}$ and of the following elimination-type property: there exists $\varepsilon>0$ and $t_0>0$, depending on $n$, $\Lambda$ and on the surface tension coefficients, such that if $u\in\mathcal{A}_{\Lambda,M}$ and $|\Omega_{h_u}\cap\{x_n>\bar{t}\}| <\e$ for some $\bar{t}>0$, then $|\Omega_{h_u}\cap\{x_n>\bar{t}+t_0\}| =0$.

To show this, \rosso{define $m(t)\coloneqq |\Omega_{h_u}\cap \{x_n> t\}|$, suppose that $m(\bar{t})<\e$} and notice that for a.e.\ $t>0$ it holds $m'(t)=-\mathcal{H}^{n-1}(\Omega_{h_u}^{(1)}\cap \{x_n= t\})$. By comparing with the configuration $v\coloneqq u\chi_{\{x_n\leq t\}}$, using the quasi-minimality of $u$ and choosing $\e$ small enough, by standard computations one obtains the differential inequality $m(t)^{\frac{n-1}{n}} \leq -Cm'(t)$ for a.e. $t>\overline{t}$, that allows to conclude that $m(t)=0$ for $t>\overline{t}+C\varepsilon^{\frac{1}{n}}$ by integration.
\end{proof}


\subsection{Partial regularity of quasi-minimizers in dimension 2} \label{subsection:regularity2}

From now on we assume that the dimension of the space is $n=2$. We also assume that the surface tension coefficients satisfy the \emph{strict} triangle inequalities
\begin{equation} \label{ass:stricttriangle}
\sigma_{AB} < \sigma_A + \sigma_B, \quad
\sigma_{A} < \sigma_B + \sigma_{AB}, \quad
\sigma_{B} < \sigma_A + \sigma_{AB}.
\end{equation}
Under these assumptions we will show a series of regularity properties satisfied by a quasi-minimizer $u$ of the surface energy $\g$, according to Definition~\ref{def:quasimin}. Notice that for $u\in\mathcal{A}_{\Lambda,M}$ we have a uniform bound
$$
\sup_{u\in\mathcal{A}_{\Lambda,M}}\Bigl( |\Omega_{h_u}| + \hu(\Gamma_{h_u}) \Bigr) \leq C
$$
for a constant $C$ depending on $M$, $\Lambda$ and on the surface tension coefficients. In two dimensions, this bound immediately yields boundedness from above of the film, that is there exists $\overline{M}>0$ (depending on $M$, $\Lambda$ and on the surface tension coefficients) such that
\begin{equation} \label{eq:bound}
\overline{\Omega}_{h_u} \subset [0,L]\times[0,\overline{M}] \qquad\text{for all }u\in\mathcal{A}_{\Lambda,M}. 
\end{equation}
The first regularity fact that we establish is an elimination property for the empty region above the film, in the spirit of \cite{Leo01}.

\begin{proposition}[Infiltration for $V$] \label{prop:infiltrationV}
Let $u\in\mathcal{A}_{\Lambda,M}$. There exists $\e_1>0$, depending on $\Lambda$, $M$, and on the surface tension coefficients, such that if for some $z_0\in\R^2$ and $r\in(0,1)$
\begin{equation} \label{eq:infiltration_V1}
|V_u^\#\cap \q_r(z_0)| < \e_1 r^2,
\end{equation}
where $\q_r(z_0) \coloneqq z_0+r\q$ and $\q\coloneqq (-\frac12,\frac12)\times(-\frac12,\frac12)$, then
\begin{equation} \label{eq:infiltration_V2}
|V_u^\#\cap \q_{\frac{r}{2}}(z_0)| = 0.
\end{equation}
\end{proposition}

\begin{proof}
Let $\cil^-\coloneqq \bigl(-\frac12,\frac12\bigr)\times\bigl(-\infty,\frac12\bigr)$ and, for $z_0\in\R^2$ and $r>0$, set $\cil_r^- \coloneqq r\cil^-$, $\cil_r^-(z_0) \coloneqq z_0+\cil_r^-$. Along the proof, to lighten the notation we will drop the subscript $u$ from the sets \eqref{ABV} of the partition determined by $u$ and from the corresponding interfaces \eqref{interfaces1}--\eqref{interfaces2}. The proof is divided into two steps.

\medskip\noindent\textit{Step 1: infiltration in strips.} We first show that there exists $\tilde{\e}_1>0$, depending on $\Lambda$ and on the surface tension coefficients, such that for every $z_0\in\R^2$ and $r\in(0,1)$ the following implication holds:
\begin{equation} \label{eq:infiltration_V5}
|V^\#\cap \cil^-_r(z_0)| < \tilde{\e}_1 r^2
\quad\quad\Longrightarrow\quad\quad
|V^\#\cap \cil^-_{\frac34r}(z_0)| = 0.
\end{equation}
We assume for notation convenience, and without loss of generality, that $\cil^-_r(z_0)\subset(0,L)\times\R$; the general case is obtained by periodicity.
For $s\in[0,r]$ we let $m(s)\coloneqq| V\cap \cil^-_s(z_0)|$, so that $m(r)\leq \tilde{\e}_1r^2$ by the assumption. The function $m(s)$ is monotone nondecreasing, with 
\begin{equation} \label{proof:infiltration0}
m'(s)=\frac12\hu(\partial\cil_s^-(z_0)\cap V^{(1)}) \qquad\text{for $\leb^1$-almost every }s>0.
\end{equation}
Fix now $s\in(0,r)$ such that \eqref{proof:infiltration0} holds and $\hu(J_u\cap\partial \cil^-_s(z_0))=0$ (notice that $\leb^1$-almost every $s>0$ has this property). We define a competitor by ``filling'' the empty region in $\cil^-_s(z_0)$ above the substrate by the phase $A$ or $B$. More precisely, assume that
\begin{equation} \label{proof:infiltration1}
\hu(\Gamma^B\cap\cil^-_s(z_0)) \leq \hu(\Gamma^A\cap\cil^-_s(z_0))
\end{equation}
and define
$$
u_s(z)\coloneqq
\begin{cases}
1 & \text{if }z\in V\cap \cil_s^-(z_0),\\
u(z) & \text{otherwise}
\end{cases}
$$
(which corresponds to fill the region $V\cap\cil_s^-(z_0)$ by the phase $A$). The proof in the other case, when one has the opposite inequality in \eqref{proof:infiltration1}, follows similarly by filling $V\cap\cil_s^-(z_0)$ by the phase $B$.

We then have $A_{u_s}=A\cup(V\cap\cil_s^-(z_0))$, $B_{u_s}=B$, and $A_{u_s}\cup B_{u_s}$ is the subgraph of an admissible profile; therefore $u_s\in\ac$ is an admissible configuration and by quasi-minimality of $u$ we find
\begin{equation} \label{proof:infiltration2}
\begin{split}
\g(u) &\leq \g(u_s) + \Lambda\bigl(|A_{u_s}\asymm A| + |B_{u_s}\asymm B|\bigr) \\
& \leq \g(u) - \sigma_A\hu(\Gamma^A\cap\cil_s^-(z_0)) + (\sigma_{AB}-\sigma_B)\hu(\Gamma^B\cap\cil_s^-(z_0)) \\
& \qquad + \sigma_A\hu(V^{(1)}\cap\partial\cil_s^-(z_0)) + (\sigma_{AS}-\sigma_S)\hu(S^V\cap\cil_s^-(z_0)) + \Lambda m(s).
\end{split}
\end{equation}
Observe now that by \eqref{proof:infiltration1}
\begin{align*}
- \sigma_A\hu(\Gamma^A & \cap\cil_s^-(z_0)) + (\sigma_{AB}-\sigma_B)\hu(\Gamma^B\cap\cil_s^-(z_0)) \\
& \leq -\sigma_A\hu(\Gamma^A\cap \cil_s^-(z_0)) + \max\{\sigma_{AB}-\sigma_B,0\}\hu(\Gamma^A\cap\cil_s^-(z_0))  \\
& = - 2c_1\hu(\Gamma^A\cap\cil_s^-(z_0)) 
\leq - c_1\bigl(\hu(\Gamma^A\cap\cil_s^-(z_0))+\hu(\Gamma^B\cap\cil_s^-(z_0))\bigr)
\end{align*}
where we set $c_1\coloneqq -\frac12\max\{\sigma_{AB}-\sigma_B-\sigma_A,-\sigma_A\}>0$ thanks to the strict triangular inequality between the coefficients.
Hence
\begin{equation} \label{proof:infiltration3}
\begin{split}
- \sigma_A\hu(&\Gamma^A\cap\cil_s^-(z_0)) + (\sigma_{AB}-\sigma_B)\hu(\Gamma^B\cap\cil_s^-(z_0)) -\sigma_S\hu(S^V\cap \cil_s^-(z_0))\\
& \leq -\min\{c_1,\sigma_S\} \bigl( \hu(\Gamma^A\cap\cil_s^-(z_0))+\hu(\Gamma^B\cap\cil_s^-(z_0)) + \hu(S^V\cap \cil_s^-(z_0)) \bigr)\\
& = -\min\{c_1,\sigma_S\} \per(V;\cil_s^-(z_0))
 = -\min\{c_1,\sigma_S\} \bigl( \per(V\cap \cil_s^-(z_0)) - 2m'(s)\bigr) \\
& \leq -c_2 |V\cap \cil_s^-(z_0)|^\frac12 + c_3 m'(s)
= -c_2 m(s)^\frac12 + c_3 m'(s),
\end{split}
\end{equation}
where we used the isoperimetric inequality in the last inequality, and $c_2$, $c_3$ are positive constants depending on the surface tension coefficients. Furthermore, by the geometry of the set $V$ we have $\hu(S^V\cap\cil_s^-(z_0))\leq \hu(V^{(1)}\cap \partial\cil_s^-(z_0))$, hence
\begin{equation} \label{proof:infiltration4}
\begin{split}
\sigma_A\hu(V^{(1)}\cap\partial\cil_s^-(z_0)) + \sigma_{AS} \hu(S^V\cap\cil_s^-(z_0))
& \leq (\sigma_A+\sigma_{AS}) \hu(V^{(1)}\cap\partial\cil_s^-(z_0)) \\
& = 2(\sigma_A+\sigma_{AS}) m'(s).
\end{split}
\end{equation}
By inserting \eqref{proof:infiltration3}--\eqref{proof:infiltration4} into \eqref{proof:infiltration2} and setting $c_4\coloneqq c_3+2(\sigma_A+\sigma_{AS})$ we find
\begin{align*} 
c_2 m(s)^\frac12
 \leq c_4 m'(s) + \Lambda m(s)
\leq c_4m'(s) + \Lambda m(s)^\frac12 m(r)^\frac12  \leq c_4m'(s) + \sqrt{\tilde{\e}_1} \Lambda m(s)^\frac12 r.
\end{align*}
The previous estimate holds for almost every $s\in(0,r)$ obeying \eqref{proof:infiltration1}, but one can obtain the same estimate also for almost every $s$ satisfying the opposite inequality (with possibly different constants $c_2$, $c_4$). Therefore
\begin{equation*}
\bigl( c_2 - \sqrt{\tilde{\e}_1}\Lambda \bigr)m(s)^\frac12 \leq  c_4 m'(s) \quad\text{for a.e. }s\in(0,r).
\end{equation*}
Then, by choosing $\tilde{\e}_1>0$ small enough, depending on $\Lambda$ and on the surface tension coefficients, we obtain that
\begin{equation*}
m(s)^\frac12 \leq C m'(s) \quad\text{for a.e. }s\in(0,r),
\end{equation*}
for a constant $C>0$ depending only on the surface tension coefficients. From this it is easy to obtain by integration that $m(\frac34r)=0$, which proves the implication \eqref{eq:infiltration_V5}.

\medskip\noindent\textit{Step 2.}
We now claim that there exists $\e_1>0$ such that
\begin{equation} \label{proof:infiltration5}
|V^\#\cap \q_r(z_0)| < \e_1 r^2
\qquad\Longrightarrow\qquad
|V^\#\cap \cil^-_{\frac34 r}(z_0)| < \tilde{\e}_1 \Bigl(\frac34r\Bigr)^2,
\end{equation}
where $\tilde{\e}_1$ is given by the previous step. Once this claim is proved, the conclusion of the proposition follows easily by combining this property with Step~1.
Let us now prove \eqref{proof:infiltration5}. As before by periodicity we can assume $\q_r(z_0)\subset(0,L)\times\R$. We denote by $\q'_r(z_0)\coloneqq z_0 + (-\frac{r}{2},\frac{r}{2})\times\{-\frac{r}{2}\}$ the bottom side of the square $\q_r(z_0)$. We first observe that, by the geometry of the set $V$, we have
\begin{equation} \label{proof:infiltration6}
\hu(V^{(1)}\cap \q_r'(z_0)) \leq \frac{|V\cap\q_r(z_0)|}{r} \leq \e_1r.
\end{equation}
Then (assuming without loss of generality that $\e_1<\frac18$) we can find $\rho\in(\frac{3}{4}r,r)$ such that the two points $z_0+(-\frac{\rho}{2},-\frac{r}{2})$, $z_0+(\frac{\rho}{2},-\frac{r}{2})$, on $\q_r'(z_0)$, are not points of density one for $V$. We then consider the strip $U\coloneqq \cil_{\rho}^-(z_0)\setminus \overline{\q_r(z_0)}$
and, by the choice of $\rho$, it follows that the lateral boundary of $U$ is outside $V^{(1)}$, hence
\begin{equation} \label{proof:infiltration6b}
\hu(V^{(1)}\cap\partial U) \leq \hu(V^{(1)}\cap \q'_r(z_0)) \xupref{proof:infiltration6}{\leq}\e_1 r.
\end{equation}
We can further assume that
\begin{equation} \label{proof:infiltration6c}
\hu(J_u\cap\partial\cil_{\rho}^-(z_0))=0,
\end{equation}
since this is valid for $\leb^1$-almost every $\rho\in(0,r)$.
To continue, we assume that
\begin{equation} \label{proof:infiltration7}
\hu(\Gamma^A\cap U) \leq \hu(\Gamma^B\cap U)
\end{equation}
and we construct a competitor by filling the region $V\cap U$ by the phase $B$ (the proof in the other case, when one has the opposite inequality in \eqref{proof:infiltration7}, follows similarly by filling $V\cap U$ by the phase $A$): we define
\begin{equation*}
\tilde{u}(z)\coloneqq
\begin{cases}
-1 & \text{if }z\in V\cap U,\\
u(z) & \text{otherwise.}
\end{cases}
\end{equation*}
Since $\tilde{u}\in\ac$ is admissible, by quasi-minimality of $u$ we find, similarly to \eqref{proof:infiltration2},
\begin{equation} \label{proof:infiltration8}
\begin{split}
\g(u) &\leq \g(u) + (\sigma_{AB}-\sigma_A)\hu(\Gamma^A\cap U) - \sigma_B\hu(\Gamma^B\cap U) \\
& \qquad + \sigma_B\hu(V^{(1)}\cap\partial U) + (\sigma_{BS}-\sigma_S)\hu(S^V\cap U) + \Lambda |V\cap U|,
\end{split}
\end{equation}
and in turn, arguing similarly to the proof of \eqref{proof:infiltration3}, using the assumption \eqref{proof:infiltration7},
\begin{multline*}
(\sigma_{AB}-\sigma_A)\hu(\Gamma^A\cap U) - \sigma_B\hu(\Gamma^B\cap U) -\sigma_S\hu(S^V\cap U) \\
\leq -c_2 |V\cap U|^\frac12 + c_1\hu(V^{(1)}\cap\partial U),
\end{multline*}
where $c_1$, $c_2$ are strictly positive constants depending on the surface tension coefficients. By inserting this inequality into \eqref{proof:infiltration8} we obtain
\begin{align*}
c_2|V\cap U|^\frac12 \leq (c_1+\sigma_B)\hu(V^{(1)}\cap \partial U) + \sigma_{BS}\hu(S^V\cap U) + \Lambda|V\cap U|.
\end{align*}
Now observe that, by the geometry of the set $V$, we have $\hu(S^V\cap U)\leq \hu(V^{(1)}\cap \q_r'(z_0))$; hence from the previous inequality and \eqref{proof:infiltration6b} it follows that
\begin{align*}
c_2|V\cap U|^\frac12 
\leq (c_1+\sigma_B+\sigma_{BS})\e_1 r + \Lambda|V\cap U|.
\end{align*}
Finally, observe that by using the uniform bound \eqref{eq:bound} and the vertical geometry of $V$ we have $|V\cap U|\leq \overline{M}\hu(V^{(1)}\cap \q_r'(z_0))\leq \e_1\overline{M}r$, so that by inserting this estimate in the previous inequality we obtain
\begin{equation*}
c_2|V\cap U|^\frac12 \leq  (c_1+\sigma_B+\sigma_{BS})\e_1 r + \e_1\Lambda\overline{M}r,
\end{equation*}
that is, $|V\cap U| \leq  C\e_1^2 r^2$ for some constant $C>0$ depending on $\Lambda$, $M$ and on the surface tension coefficients (recall that the bound $\overline{M}$ in \eqref{eq:bound} depends only on these quantities). Eventually
\begin{align*}
|V\cap \cil^-_{\frac34 r}(z_0)|
&\leq |V\cap U| + |V\cap\q_r(z_0)| \leq C\e_1^2r^2 + \e_1 r^2 \leq \tilde{\e}_1\Bigl(\frac34r\Bigr)^2,
\end{align*}
provided that we choose $\e_1$ small enough. This completes the proof of \eqref{proof:infiltration5}.
\end{proof}

We also have a dual statement for the region occupied by the film. The proof follows by the same argument used in the proof of Proposition~\ref{prop:infiltrationV}. However, since we can obtain this result as a consequence of the stronger property proved in Proposition~\ref{prop:innerball}, we omit the proof.

\begin{proposition}[Infiltration for $A\cup B$] \label{prop:infiltrationAB}
Let $u\in\mathcal{A}_{\Lambda,M}$. There exists $\e_2>0$, depending on $\Lambda$, $M$, and on the surface tension coefficients, such that if for some $z_0\in\R^2$ and $r\in(0,1)$ such that $\q_r(z_0)\cap S=\emptyset$ we have
\begin{equation} \label{eq:infiltration_1}
|\Omega_{h_u}^\#\cap \q_r(z_0)| < \e_2 r^2,
\end{equation}
then
\begin{equation} \label{eq:infiltration_2}
|\Omega_{h_u}^\#\cap \q_{\frac{r}{2}}(z_0)| = 0.
\end{equation}
\end{proposition}

\begin{remark}
It is worth to notice that, in the proof of Proposition~\ref{prop:infiltrationV} (and also of Proposition~\ref{prop:infiltrationAB}), the constraint of being subgraphs prevents us to construct competitors by means of local variations in a square $\q_r(z_0)$, but imposes to consider the full vertical region below or above the square.
Notice also that, if a single set $\Omega_h$ is a quasi-minimizer of the perimeter in the class of subgraphs, a standard argument \cite[Theorem~14.8]{G} shows that it is actually a quasi-minimizer among \emph{all} possible competitors of finite perimeter, without the constraint (we will exploit this fact in the proof of Proposition~\ref{prop:C1alpha_graph}). However, in our case we have a partition of the subgraph into two sets $A$, $B$ of finite perimeter, and this argument fails due to the presence of different types of interfaces between the phases.
\end{remark}

We continue by proving that a quasi-minimizer satisfies an interior ball condition. The proof of this result follows a strategy devised in \cite{ChaLar03} (see also \cite{FonFusLeoMor07,FusMor12}) and adapted to our setting. For the reader's convenience, we report here the details of the main nontrivial technical changes that have to be made. Since we are in dimension 2, the function $h_u^-$ (see \eqref{h+-}) is a lower semicontinuous representative of $h_u$; in the following it will be convenient to identify $h_u$ with $h_u^-$, so that in particular the subgraph $\Omega_{h_u}^\#$ is an open set. From now on, we work under this convention.

\begin{proposition}[Interior ball] \label{prop:innerball}
Let $u\in\mathcal{A}_{\Lambda,M}$ and let $\rho_0<\frac{\min\{\sigma_A,\sigma_B\}}{\Lambda}$. Then for every $\bar{z}\in{\Gamma^\#_{h_u}}$ there exists an open ball $B_{\rho_0}(z_0)$ such that
\begin{equation} \label{eq:innerball}
B_{\rho_0}(z_0)\subset\bigl\{(x,y)\in\R^2 \,:\, y<h_u(x)\bigr\}, \qquad \partial B_{\rho_0}(z_0) \cap {\Gamma^\#_{h_u}}=\{\bar{z}\}.
\end{equation}
\end{proposition}

\begin{proof}
We divide the proof into two steps. To simplify the notation we drop the subscripts on the various objects depending on $u$, which is fixed along this proof. We also denote by
\begin{equation} \label{eq:subgraph}
\Omega^-_h \coloneqq \bigl\{(x,y)\in\R^2 \,:\, y<h(x)\bigr\}
\end{equation}
where $h=h_u$ is the profile associated with the configuration $u$. Recall that, by the convention of identifying $h$ with its lower semicontinuous representative $h^-$, the set $\Omega_h^-$ is open.

\medskip\noindent\textit{Step 1.} We claim that for every ball $B_{\rho_0}(z_0)\subset\Omega^-_h$, where $\rho_0$ is as in the statement, the set $\partial B_{\rho_0}(z_0) \cap \Gamma^\#_h$ consists of at most one point.

Suppose on the contrary that there exists a ball $B_{\rho}(z_0)\subset\Omega_h^-$, $z_0=(x_0,y_0)$, such that $\partial B_{\rho}(z_0) \cap \Gamma_h^\#$ contains at least two points $a=(x_a,y_a)$, $b=(x_b,y_b)$, with $x_a\leq x_b$, $h^-(x_a)\leq y_a\leq h^+(x_a)$, $h^-(x_b)\leq y_b\leq h^+(x_b)$.  We will prove that this is not possible if $\rho\leq\rho_0$. By periodicity, we assume without loss of generality that $B_{\rho}(z_0)\subset(0,L)\times\R$.

\begin{figure}
	\centering
	\includegraphics[scale=0.3]{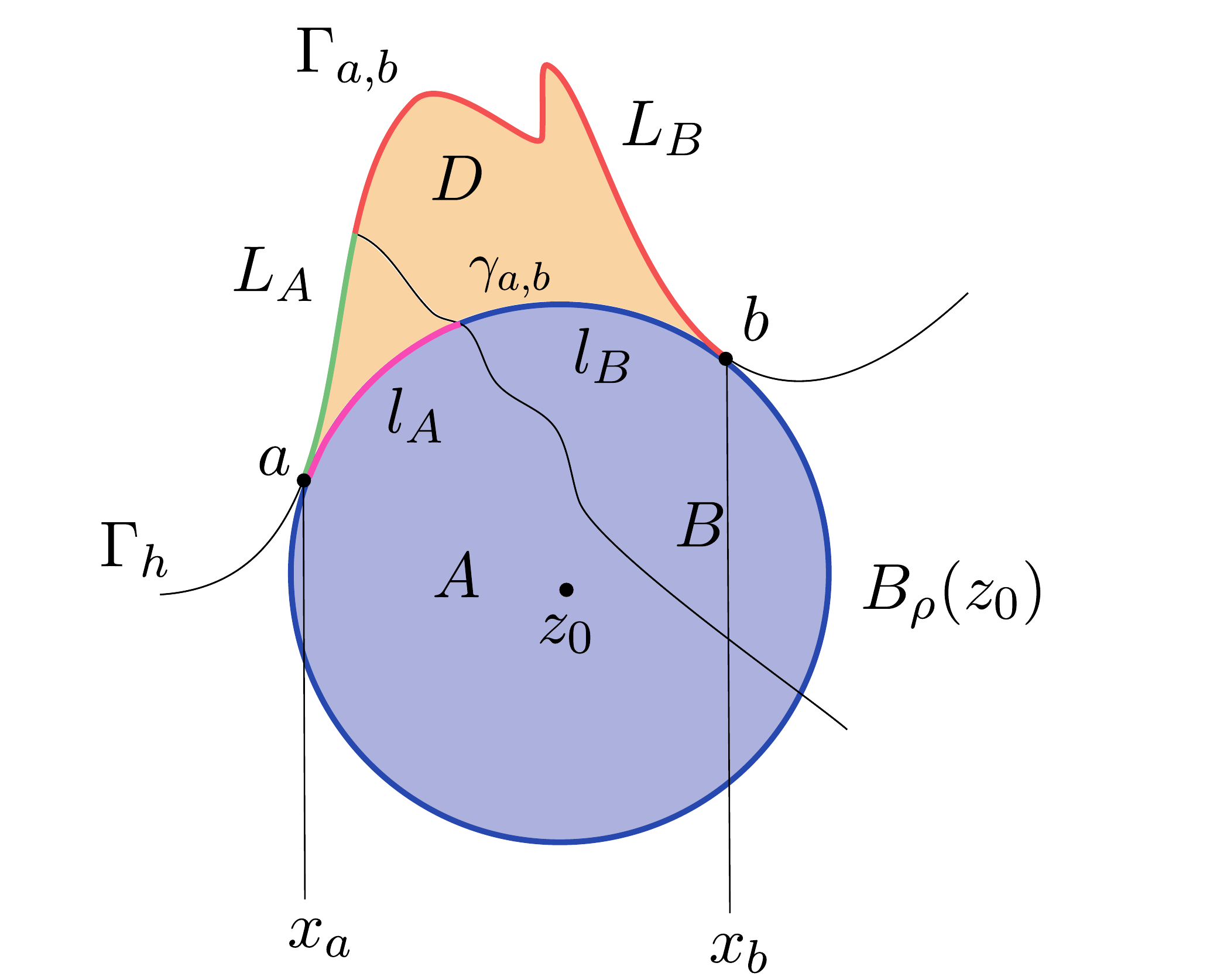}
	\caption{The construction in the proof of Proposition~\ref{prop:innerball}.}
	\label{fig:ball}
\end{figure}

We define (see Figure~\ref{fig:ball}) $\Gamma_{a,b}$ to be the arc on $\Gamma_h$ connecting $a$ with $b$, and $\gamma_{a,b}$ to be the arc on $\partial B_\rho(z_0)\cap\{y\geq y_0\}$ connecting $a$ with $b$. Notice that by construction $\Gamma_{a,b}$ lies above $\gamma_{a,b}$. We also let $D$ to be the region enclosed by $\Gamma_{a,b}$ and $\gamma_{a,b}$ (i.e.\ the bounded component of $\R^2\setminus(\Gamma_{a,b}\cup\gamma_{a,b})$). Also let
$$
L_A\coloneqq\hu(\Gamma_{a,b}\cap\Gamma^A), \quad L_B\coloneqq\hu(\Gamma_{a,b}\cap\Gamma^B), \quad L\coloneqq L_A+L_B=\hu(\Gamma_{a,b}),
$$
$$
\ell_A\coloneqq \hu\bigl( \gamma_{a,b}\cap \partial^*(A\cap B_\rho(z_0)) \bigr), \quad \ell_B\coloneqq \hu(\gamma_{a,b})-\ell_A, \quad \ell\coloneqq\ell_A+\ell_B=\hu(\gamma_{a,b}).
$$

We construct a competitor by removing the set $D$ from the subgraph $\Omega_h$: more precisely, we define $\tilde{u}\coloneqq u\chi_{D^c}$. In this way $\tilde{u}$ is an admissible configuration, $A_{\tilde{u}}=A\setminus D$, $B_{\tilde{u}}=B\setminus D$, and the profile $h_{\tilde{u}}$ coincides with $h$ outside $[x_a,x_b]$, and its graph on $[x_a,x_b]$ is given by $\gamma_{a,b}$. Then, by using the quasi-minimality of $u$ we obtain
\begin{align*}
\g(u)
& \leq \g(\tilde{u}) + \Lambda |D| 
 = \g(u) - \sigma_AL_A - \sigma_BL_B - \sigma_{AB}\hu(\Gamma^{AB}\cap\overline{D}) + \sigma_A\ell_A + \sigma_B\ell_B + \Lambda|D|,
\end{align*}
hence
\begin{equation} \label{proof:innerball1}
\sigma_A (L_A-\ell_A) + \sigma_B(L_B-\ell_B) +\sigma_{AB}\hu(\Gamma^{AB}\cap\overline{D}) \leq \Lambda|D|.
\end{equation}
To continue, we distinguish between two cases.

\smallskip\noindent\textbf{Case 1:} assume that $\ell_A\leq L_A$, $\ell_B\leq L_B$. In this case we obtain from \eqref{proof:innerball1}
\begin{equation} \label{proof:innerball2}
\min\{\sigma_A,\sigma_B\}(L-\ell) \leq \Lambda|D|.
\end{equation}
We estimate the difference $L-\ell$ as in \cite[Lemma~6.6]{FusMor12}; we reproduce the argument here for the reader's convenience. Assume first that $h$ is Lipschitz continuous. Since $\Gamma_{a,b}$ and $\gamma_{a,b}$ are the graphs on $[x_a,x_b]$ of $h$ and $h_{\tilde{u}}$, respectively, and $h(x_a)=h_{\tilde{u}}(x_a)$, $h(x_b)=h_{\tilde{u}}(x_b)$, we have
\begin{align*}
L-\ell
& = \hu(\Gamma_{a,b})-\hu(\gamma_{a,b})
= \int_{x_a}^{x_b} \Bigl( \sqrt{1+(h'(x))^2} - \sqrt{1+(h_{\tilde{u}}'(x))^2}\Bigr)\de x \\
& \geq - \int_{x_a}^{x_b} \biggl( \frac{h_{\tilde{u}}'(x)}{\sqrt{1+(h_{\tilde{u}}'(x))^2}} \biggr)' \bigl( h(x)-h_{\tilde{u}}(x)\bigr) \de x
= \frac{1}{\rho}|D|.
\end{align*}
By combining this inequality with \eqref{proof:innerball2}, we see that necessarily $\rho\geq\frac{\min\{\sigma_A,\sigma_B\}}{\Lambda}$.
Therefore, any ball $B_{\rho_0}(z_0)\subset\Omega_h^-$ can touch the graph $\Gamma_h^\#$ at most once.

If $h$ is not Lipschitz, then we can approximate $h$ by a sequence of Lipschitz functions $g_k$ as in \cite[Lemma~6.2]{FusMor12}, write the previous inequality for $g_k$ and obtain the same conclusion by passing to the limit.

\smallskip\noindent\textbf{Case 2:} assume that $\ell_A>L_A$, $\ell_B\leq L_B$ (the other case, $\ell_A\leq L_A$ and $\ell_B>L_B$, is completely analogous). In this case, thanks to the triangle inequality $\sigma_A<\sigma_B+\sigma_{AB}$ (see \eqref{ass:stricttriangle}), we find by \eqref{proof:innerball1}
\begin{equation*}
(\sigma_B+\sigma_{AB})(L_A-\ell_A) +\sigma_B(L_B-\ell_B) + \sigma_{AB}\hu(\Gamma^{AB}\cap\overline{D}) \leq \Lambda |D|,
\end{equation*}
and in turn
\begin{equation*}
\sigma_B(L-\ell) + \sigma_{AB}(L_A-\ell_A+\hu(\Gamma^{AB}\cap\overline{D})) \leq \Lambda |D|.
\end{equation*}
One can check that
\begin{equation}\label{eq:ineq_circle}
L_A-\ell_A+\hu(\Gamma^{AB}\cap\overline{D})\geq0
\end{equation}
(this inequality follows essentially by \cite[Ex.~15.14]{Mag}).
Therefore, using \eqref{eq:ineq_circle} it follows that $\sigma_B(L-\ell)\leq \Lambda |D|$ and in particular \eqref{proof:innerball2} holds. We can therefore repeat the same argument as in the previous case.

\medskip\noindent\textit{Step 2.} By the result in Step~1, the existence of an interior ball at every point of $\Gamma_h^\#$ can be proved by following the same lines as in \cite[Lemma~2]{ChaLar03} or \cite[Proposition~3.3, Step~2]{FonFusLeoMor07}.
\end{proof}

As a consequence of the interior ball condition proved in Proposition~\ref{prop:innerball}, we obtain the Lipschitz regularity of the free profile $\Gamma_{h_u}^\#$ of a quasi-minimizer $u$ outside a finite set. The proof is an adaptation of the strategy of \cite[Lemma~3]{ChaLar03}.

\begin{proposition}[Lipschitz regularity] \label{prop:lipschitz}
Let $u\in\mathcal{A}_{\Lambda,M}$. There exists a finite set $\Sigma\subset Q_L$, with $J_{h_u}\subset\Sigma$, such that $h_u$ is locally Lipschitz in $Q_L\setminus\Sigma$ and has left and right derivatives at every point of $Q_L\setminus\Sigma$, that are respectively left and right continuous.
\end{proposition}

\begin{proof}
We denote by $h\coloneqq h_u$ the admissible profile associated with the configuration $u$. For $z\in\Gamma_{h}^\#$, we let
$$
n(z)\coloneqq \bigl\{ \nu\in\s^1 \,:\, B_{\rho_0}(z+\rho_0\nu)\subset\Omega_h^-\bigr\}
$$
where $\rho_0>0$ is given by Proposition~\ref{prop:innerball} and $\Omega_h^-$ is defined in \eqref{eq:subgraph}. In view of Proposition~\ref{prop:innerball}, $n(z)\neq\emptyset$ for all $z\in\Gamma_h^\#$. Notice also that, if $\nu\in n(z)$ for some $z$, then necessarily $\nu\cdot e_2\leq0$ (or otherwise $\Gamma_h^\#$ would not be an extended graph). We define the singular set
\begin{equation} \label{eq:singular_set}
\Sigma \coloneqq \pi_x\big(\bigl\{ z\in\Gamma_h \,:\, e_1\in n(z)\text{ or }-e_1\in n(z)  \bigr\}\bigr) \subset Q_L,
\end{equation}
where $\pi_x$ denotes the projection on the $x$-axis, and by $\Sigma^\#$ its periodic extension.

Similarly to \cite[Lemma~3]{ChaLar03}, one can show that the set $n(z)$ is a closed  (possibly degenerate) arc on $\s^1$ with length strictly smaller than $\pi$. In turn, this allows to show that $\Sigma$ is a finite set with $J_h\subset\Sigma$, and that $h$ is locally Lipschitz in $\R\setminus\Sigma^\#$.
\end{proof}

\begin{remark} \label{rmk:singularset}
The points in the singular set $\Sigma$ identified in Proposition~\ref{prop:lipschitz} are of two possible kinds: they are either jump points of the function $h_u$, or continuity points of $h_u$ at which the left or the right derivative of $h_u$ is infinite. At the upper point $(x,h_u^+(x))$ of a jump $x\in J_{h_u}$, the graph has a vertical tangent. Notice also that the graph of $h_u$ does not contain cusp points as a consequence of the infiltration property and the inner ball condition (see Step~1 in the proof of Proposition~\ref{prop:lipschitz}).
\end{remark}

\begin{remark}\label{rem:interior_regularity}
Let $u\in\mathcal{A}_{\Lambda,M}$. Since the set $\Omega_h^\#=\{(x,y)\in\R^2 : 0<y<h(x)\}$ is open, $A$ is a quasi-minimizer of the perimeter in $\Omega_h^\#$ in the classical sense. Thus it is possible to apply standard regularity results (see \cite[Theorems~26.5 and 28.1]{Mag}) to obtain that $\Gamma^{AB}_u$ is a locally a $C^{1,\alpha}$-curve in $\Omega_h^\#$ for every $\alpha\in(0,1/2)$, and that it coincides with $\partial A\cap \partial B$ in $\Omega_h^\#$.
\end{remark}

Finally, we show that the graph as a better regularity around points of $\partial^*A\cup\partial^*B$. Notice that if $h_u(x_0)=0$, then $(x_0,h_u(x_0))\notin\partial^*A\cup\partial^*B$, or otherwise $V_u$ would have Lebesgue density zero at that point, which is not permitted by the infiltration property in Proposition~\ref{prop:infiltrationV}.

\begin{proposition}\label{prop:C1alpha_graph}
Let $u\in\mathcal{A}_{\Lambda,M}$. If $x_0\in Q_L\setminus\Sigma$ is such that $(x_0,h_u(x_0))\in\partial^*A\cup\partial^*B$, then $h_u$ is of class $C^{1,\alpha}$ in a neighbourhood of $x_0$, for every $\alpha\in(0,1/2)$.
\end{proposition}

\begin{proof}
To simplify the notation we denote by $h\coloneqq h_u$ the admissible profile associated with the configuration $u$, and we remove the subscript $u$ from the sets of the corresponding partition.
Recalling \eqref{ass:stricttriangle}, we let
\begin{equation}\label{eq:strict_ineq_delta}
\delta\coloneqq \min\{\sigma_{AB}+\sigma_A - \sigma_B, \sigma_{AB},\sigma_{A}\}>0.
\end{equation}

\medskip\noindent\textit{Step 1.}
Fix $x_0$ as in the statement and let $z_0\coloneqq(x_0,h(x_0))$.
Thanks to Proposition~\ref{prop:lipschitz} we can find $r_0>0$ (depending on $x_0$) such that $h$ is Lipschitz continuous in $(x_0-r_0,x_0+r_0)$, with Lipschitz constant $\ell\in(0,\infty)$.
For $s>0$, set
$R_s\coloneqq z_0 + sR,$
where $R\coloneqq (-1,1)\times (-2\ell,2\ell).$
Then
\begin{equation}\label{eq:no_exit_above}
\Gamma_h \cap \partial^\pm R_s = \emptyset
\end{equation}
for all $s\in(0,r_0)$, where
$\partial^\pm R_s \coloneqq z_0 + (-s,s)\times\{\pm 2\ell s\}.$
Moreover, by possibly reducing the value of $r_0$, we can also assume that $R_s\cap S=\emptyset$.

We now prove an infiltration-type property, similar to Proposition~\ref{prop:infiltrationV}, for the two phases $A$, $B$ at the point $z_0$. Precisely, we claim that there exists $\varepsilon>0$ (depending on $z_0$) such that if
\begin{equation}\label{eq:inf_1}
| A\cap R_r | \leq \varepsilon r^2,
\end{equation}
for some $0<r<r_0$, then 
\begin{equation}\label{eq:inf_2}
| A\cap R_{r/2} |=0.
\end{equation}
The same property holds if the set $A$ is replaced by the set $B$.

For $s\in(0,r_0)$ set $m(s) \coloneqq |A\cap R_s|$. Then for $\leb^1$-a.e.\ $s\in(0,r_0)$ we have that
\begin{equation}\label{eq:mprime}
m'(s)=2\ell\,\hu(A^{(1)}\cap(\partial^+ R_s\cup\partial^-R_s)) +\hu(A^{(1)}\cap \partial R_s\setminus(\partial^+R_s\cup\partial^- R_s))
\end{equation}
and that
\begin{equation}\label{eq:no_interface_boundary}
\hu(\partial R_s\cap (\partial^*A\cup\partial^*B ))=0.
\end{equation}
We claim that there exist $C_1,C_2>0$ such that for $s\in(0,r_0)$ satisfying \eqref{eq:mprime} and \eqref{eq:no_interface_boundary}, the following differential inequality is true:
\begin{equation}\label{eq:diff_ineq_il}
C_1 m(s)^{\frac{1}{2}}\leq C_2 m'(s) + 3\Lambda m(s).
\end{equation}
Once \eqref{eq:diff_ineq_il} is established, \eqref{eq:inf_2} will follow by a standard argument by using \eqref{eq:inf_1} and choosing $\e$ sufficiently small, as in the last part of Step~1 in the proof of Proposition~\ref{prop:infiltrationV}.

We are thus left with proving \eqref{eq:diff_ineq_il}.
The idea is to construct a suitable competitor and to use the quasi-minimality inequality for $u$; we have to pay attention that the competitor satisfies the graph constraint. If
\begin{equation}\label{eq:assumption_il_1}
\hu(\Gamma^A\cap R_s) > \hu(\Gamma^{AB}\cap R_s) 
\end{equation}
then we set (see Figure~\ref{fig:competitor1})
\begin{equation}\label{eq:comeptitor_il_1}
A_s\coloneqq A\setminus R_s,\quad\quad\quad
B_s\coloneqq (B\setminus R_s) \cup \left( B \cap R_s \right)^G,
\end{equation}
where, for a measurable set $E\subset R_s$ we define 
\begin{equation} \label{eq:giustizzato}
E^G \coloneqq \bigl\{ (x,y)\in R_s \,:\, h(x_0)-2\ell s \leq y \leq h(x_0)-2\ell s + \hu(E_x)  \bigr\},
\end{equation}
with $E_x\coloneqq\{t\in\R \,:\, (x,t)\in E \}$. In the case where
\begin{equation}\label{eq:assumption_il_2}
\hu(\Gamma^A\cap R_s) \leq \hu(\Gamma^{AB}\cap R_s) 
\end{equation}
we set instead
\begin{equation}\label{eq:comeptitor_il_2}
A_s\coloneqq A\setminus R_s,\quad\quad\quad
B_s\coloneqq B \cup \left( A \cap R_s \right). 
\end{equation}
Note that the configuration $v_s=\chi_{A_s}-\chi_{B_s}$ is an admissible competitor for the quasi-minimality inequality in Definition~\ref{def:quasimin}. In the first case this follows from \eqref{eq:no_exit_above}, while in the second case the free profile of the configuration is left unchanged.
Denote by $h_s:Q_L\to[0,\infty)$ the admissible profile such that $\Omega_{h_s}=A_s\cup B_s$.

\begin{figure}
\begin{center}
\includegraphics[scale=0.6]{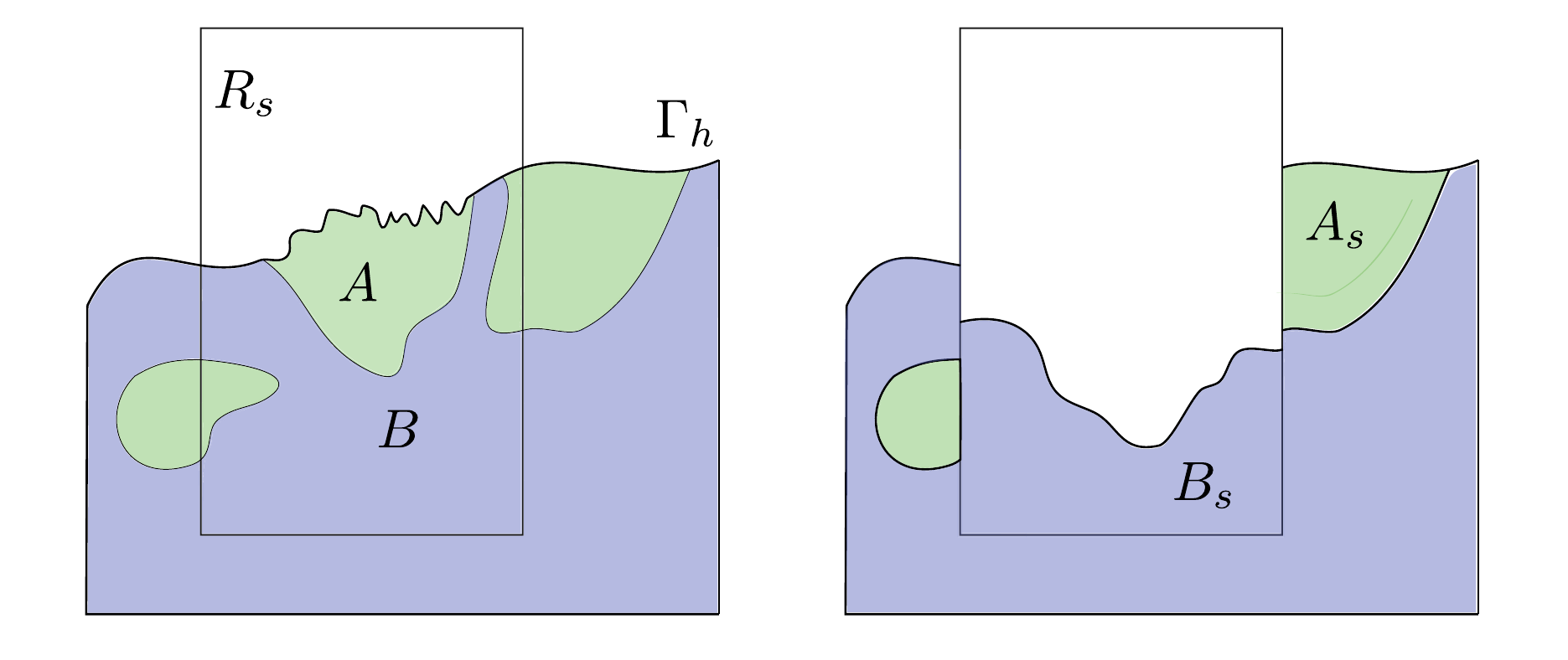}
\end{center}
\caption{The construction of the competitor in the case $\hu(\Gamma^A\cap R_s) > \hu(\Gamma^{AB}\cap R_s)$: we remove $A\cap R_s$ and we `move down' $B\cap R_s$.}
\label{fig:competitor1}
\end{figure}

Assume that \eqref{eq:assumption_il_1} holds, and thus $A_s$ and $B_s$ are defined as in \eqref{eq:comeptitor_il_1}.
Then, by an argument similar to \cite[Lemma~14.7]{G} one can prove that
\begin{equation}\label{eq:ineq_perimeter_1}
\hu\left( \partial^*B_s \cap R_s \right) \leq \hu(\partial^*B \cap R_s) + \hu(A^{(1)}\cap \partial^-R_s).
\end{equation}
Moreover, one can check that
\begin{equation}\label{eq:ineq_perimeter_2}
|h_s^+(x_0-s)-h_s^-(x_0-s)|+|h_s^+(x_0+s)-h_s^-(x_0+s)|\leq \hu(A^{(1)}\cap\partial R_s) \xupref{eq:mprime}{\leq} m'(s),
\end{equation}
and that
\begin{equation}\label{eq:ineq_perimeter_4}
|B_s\asymm B| \leq 2|A\cap R_s| = 2m(s).
\end{equation}
The quasi-minimality of $u$, together with \eqref{eq:mprime}, \eqref{eq:no_interface_boundary}, \eqref{eq:ineq_perimeter_1}, \eqref{eq:ineq_perimeter_2}, and \eqref{eq:ineq_perimeter_4}, yields
\begin{equation}\label{eq:qm_il_1}
\begin{split}
\sigma_A \hu(&\Gamma^A\cap R_s) + \sigma_B  \hu(\Gamma^B\cap R_s) + \sigma_{AB}\hu(\Gamma^{AB}\cap R_s) \\
& \leq \sigma_B\hu(\partial^*B_s\cap R_s) + \max\{\sigma_A,\sigma_{AB}\}\hu(A^{(1)}\cap\partial R_s) \\
& \qquad + \max\{\sigma_A, \sigma_B\}\bigl(|h_s^+(x_0-s)-h_s^-(x_0-s)|+|h_s^+(x_0+s)-h_s^-(x_0+s)|\bigr) \\
& \qquad + \Lambda m(s) + \Lambda|B_s\asymm B| \\
& \leq \sigma_B\hu(\Gamma^B\cap R_s) + \sigma_{B}\hu(\Gamma^{AB}\cap R_s) + 3\max\{\sigma_A, \sigma_B, \sigma_{AB}\} m'(s) +3\Lambda m(s).
\end{split}			
\end{equation}
By using \eqref{eq:strict_ineq_delta}, \eqref{eq:assumption_il_1}, and the isoperimetric inequality, we estimate
\begin{equation}\label{eq:ineq_il_1}
\begin{split}
\sigma_A\hu(\Gamma^A\cap R_s) & + (\sigma_{AB}-\sigma_B) \hu(\Gamma^{AB}\cap R_s) \\
& \geq \sigma_A\hu(\Gamma^A\cap R_s) + \min\{\sigma_{AB}-\sigma_B,0\}\hu(\Gamma^{A}\cap R_s) \\
& \geq \delta\hu(\Gamma^{A}\cap R_s)
\geq \frac{\delta}{2} \hu(\Gamma^{AB}\cap R_s) + \frac{\delta}{2} \hu(\Gamma^A\cap R_s) \\
& = \frac{\delta}{2}  \Bigl( \per(A\cap R_s) - \hu(A^{(1)}\cap\partial R_s) \Bigr)\geq C \bigl( m(s)^{\frac{1}{2}} - m'(s) \bigr).
\end{split}
\end{equation}
By inserting \eqref{eq:ineq_il_1} into \eqref{eq:qm_il_1} we obtain the desired inequality \eqref{eq:diff_ineq_il} in the case where assumption \eqref{eq:assumption_il_1} holds.

Assume now that \eqref{eq:assumption_il_2} is in force, and thus $A_s$ and $B_s$ are defined as in \eqref{eq:comeptitor_il_2}. In this case, the quasi-minimality inequality for $u$ yields, using also \eqref{eq:mprime} and \eqref{eq:no_interface_boundary},
\begin{align}\label{eq:qm_il_2}
\sigma_A \hu(\Gamma^A\cap R_s) + \sigma_{AB}\hu(\Gamma^{AB}\cap R_s)
&\leq \sigma_B\hu(\Gamma^{A}\cap R_s) + \sigma_{AB} \hu(A^{(1)}\cap\partial R_s) + 2 \Lambda m(s) \nonumber \\
&\leq \sigma_B\hu(\Gamma^{A}\cap R_s) + \sigma_{AB}m'(s) +2 \Lambda m(s).
\end{align}
By using \eqref{eq:strict_ineq_delta}, \eqref{eq:assumption_il_2}, and the isoperimetric inequality, we estimate (with similar computations as those used to get \eqref{eq:ineq_il_1})
\begin{equation}\label{eq:ineq_il_3}
(\sigma_A-\sigma_B) \hu(\Gamma^A\cap R_s) + \sigma_{AB} \hu(\Gamma^{AB}\cap R_s)
\geq C \bigl( m(s)^{\frac{1}{2}} - m'(s) \bigr).
\end{equation}
By inserting \eqref{eq:ineq_il_3} into \eqref{eq:qm_il_2} we obtain the desired inequality \eqref{eq:diff_ineq_il} also in the case where assumption \eqref{eq:assumption_il_2} holds.

\medskip\noindent\textit{Step 2.}
We conclude as follows. Let $x_0$ be as in the statement, and let $z_0\coloneqq(x_0,h(x_0))\in \partial^*A\cup\partial^*B$. Assume $z_0\in\partial^*B$.
Let $r_0>0$ be as in Step~1, so that $h$ is Lipschitz continuous in $(x_0-r_0,x_0+r_0)$ with Lipschitz constant $\ell$. 
Let also $\e>0$ be given by Step~1. Then, since $z_0\in\partial^*B\cap\partial^*V$, it is possible to find $r\in(0,r_0)$ such that
$|A\cap R_r(z_0)|<\varepsilon r^2.$
From Step~1 we get that $|A\cap R_{r/2}|=0$. This implies that in $R_{r/2}$ there are only the sets $V$ and $B$. In particular, we also have $\partial^+R_{r/2}\subset B^{(0)}$, $\partial^-R_{r/2}\subset B^{(1)}$ (recall \eqref{eq:no_exit_above}), and $B\cap R_{r/2}$ is the subgraph of a function of bounded variation.

Let now $E\subset R_{r/2}$ be any set of finite perimeter such that $E\asymm B\subset\subset R_{r/2}$, and let $E^G$ be the set defined in \eqref{eq:giustizzato}. We can test the quasi-minimality inequality with the competitor obtained by replacing the phase $B\cap R_{r/2}$ by $E^G$, which is admissible since it satisfies the graph constraint. We therefore find
\begin{align}\label{eq:quasiminper}
\sigma_B\per(B;R_{r/2}) \leq \sigma_B\per(E^G;R_{r/2}) + \Lambda|(E^G\asymm B)\cap R_{r/2}|.
\end{align}
We observe now that, similarly to \eqref{eq:ineq_perimeter_1}, we have $\per(E^G;R_{r/2})\leq \per(E;R_{r/2})$; moreover (since $B\cap R_{r/2}=(B\cap R_{r/2})^G$) we also have
\begin{align*}
|(E^G\asymm B)\cap R_{r/2}|
& = |E^G\asymm (B\cap R_{r/2})^G|
 = |E^G| + |(B\cap R_{r/2})^G| - 2|E^G\cap(B\cap R_{r/2})^G| \\
& \leq |E| + |B\cap R_{r/2}| - 2|E\cap(B\cap R_{r/2})| = |E\asymm (B\cap R_{r/2})|.
\end{align*}
 Then, inserting the previous inequalities inside \eqref{eq:quasiminper} we find
\begin{align*}
\per(B;R_{r/2}) \leq \per(E;R_{r/2}) + \frac{\Lambda}{\sigma_B} |(E\asymm B)\cap R_{r/2}|
\end{align*}
for every set of finite perimeter $E\subset R_{r/2}$ such that $E\asymm B\subset\subset R_{r/2}$.
This shows that $B$ is a quasi-minimizer of the perimeter inside $R_{r/2}$ in the classical sense, that is \emph{without} the graph constraint.
Thus, the regularity of $\Gamma_h\cap R_{r/2}=\partial^*B\cap R_{r/2}$ follows from classical regularity results for quasi-minimizers of the perimeter (see, for instance, \cite[Theorems~26.5 and 28.1]{Mag}). This concludes the proof.
\end{proof}

By collecting all the previous statements, we obtain the properties listed in Theorem~\ref{thm:regularity}.

\begin{proof}[Proof of Theorem~\ref{thm:regularity}]
The infiltration property \ref{item1-reg} follows by Proposition~\ref{prop:infiltrationV} and Proposition~\ref{prop:infiltrationAB}.
The Lipschitz regularity \ref{item2-reg} and the characterization of the singular set \ref{item3-reg} are proved in Proposition~\ref{prop:lipschitz} and Remark~\ref{rmk:singularset}. The internal regularity of the interface \ref{item4-reg} is discussed in Remark~\ref{rem:interior_regularity}. The $C^{1,\alpha}$-regularity of the graph \ref{item5-reg} is proved in Proposition~\ref{prop:C1alpha_graph}.
\end{proof}


\bigskip
\bigskip
\noindent
{\bf Acknowledgments.}
MB is member of 2020 INdAM - GNAMPA project \textit{Variational Analysis of nonlocal models in applied science}. During the period at Heriot-Watt University, the research of RC was supported by the UK Engineering and Physical Sciences Research Council (EPSRC) via the grant EP/R013527/2 \textit{Designer Microstructure via Optimal Transport Theory}.

\bibliographystyle{siam}
\bibliography{bibliography}

\end{document}